\documentclass[11pt]{article}
\usepackage{amsfonts,amssymb,amsmath,amsthm,epsfig,euscript}
\usepackage{graphicx}
\usepackage{enumerate}
\usepackage{mathrsfs}

\newtheorem{theorem}{Theorem}

\newtheorem{definition}[theorem]{Definition}

\long\def\symbolfootnote[#1]#2{\begingroup
	\def\thefootnote{\fnsymbol{footnote}}\footnote[#1]{#2}\endgroup}

\newcommand{\first}{\mathrm{first}}
\newcommand{\last}{\mathrm{last}}

\newcommand{\col}{\text{Col}}

\newcommand{\la}{\lambda}

\newcommand{\ris}{\mathrm{ris}}

\newcommand{\des}{\mathrm{des}}

\newcommand{\qbinom}[2]{\genfrac{[}{]}{0pt}{}{#1}{#2}_{q}}

\newcommand{\sg}{\sigma}

\newcommand{\cref}[1]{Corollary \ref{corollary:#1}}

\newcommand{\inv}{\mathrm{inv}}
\newcommand{\coinv}{\mathrm{coinv}}

\newcommand{\red}{\mathrm{red}}

\newcommand{\Pmch}{\text{$P$-$\mathrm{mch}$}}

\newcommand{\tmch}{\text{$\tau$-$\mathrm{mch}$}}

\newcommand{\Gmch}{\text{$\Gamma$-$\mathrm{mch}$}}

\newcommand{\fig}[3]
{\begin{figure}[ht]
		\centerline{\scalebox{#1}{\epsfig{file=#2.eps}}}
		\vspace{-1mm}
		\caption{#3}
		\label{fig:#2}
	\end{figure}}

\title{Block patterns in generalized Euler Permutations}
	
	\author{
		Ran Pan \\
		\small Department of Mathematics\\[-0.8ex]
		\small University of California, San Diego\\[-0.8ex]
		\small La Jolla, CA 92093-0112. USA\\[-0.8ex]
		\small \texttt{ran.pan.math@gmail.com}
		\and
		Jeffrey Brian Remmel \\
		\small Department of Mathematics\\[-0.8ex]
		\small University of California, San Diego\\[-0.8ex]
		\small La Jolla, CA 92093-0112. USA\\[-0.8ex]
		\small \texttt{remmel@math.ucsd.edu}
		\and
	}
	
	\date{\small Submitted: Date 1;  Accepted: Date 2;
		Published: Date 3.\\
		\small MR Subject Classifications: 05A15, 05E05}

\begin{document}
\maketitle

\begin{abstract}
Goulden and Jackson introduced a very powerful method to study the distributions of certain consecutive 
patterns in permutations, words, and other combinatorial 
objects which is now called the cluster method. There are a number of natural 
classes of combinatorial objects which start with either permutations or words 
and add additional restrictions.  These include up-down permutations, 
generalized Euler permutations, words with no consecutive repeated letters, Young tableaux, and non-backtracking random walks. We develop an extension 
of the cluster method which we call the {\em generalized cluster method} to study the distribution of certain consecutive patterns in such restricted combinatorial objects. In this paper, we focus on block patterns in generalized Euler permutations. 
\end{abstract}

\section{Introduction}

Goulden and Jackson \cite{GJ} 
introduced a very powerful method to study the 
distributions of certain consecutive 
patterns in permutations, words, and other combinatorial 
objects which is now called the cluster method. 
There are now numerous applications of the cluster 
method to study consecutive patterns in permutations and words, 
see, for example, \cite{DK,EN,KS,N,NZ}. 

The main goal of this paper is to give an extension of cluster method 
which we call the {\bf generalized cluster method}
 which is suitable for studying 
the distribution of consecutive patterns in permutations and words 
that satisfy additional restrictions. These include various subsets of 
permutations with regular descent patterns such as up-down permutations,   
generalized Euler permutations,  
permutations and words associated with various classes of 
Young tableaux, and various subsets 
of words such as words with no consecutive repeated letters and words 
associated with non-backtracking random walks. Our work was motivated 
by a paper from the second author \cite{R} who studied 
consecutive patterns in up-down permutations.  In particular, \cite{R} 
introduced the notion of generalized maximum packings which is an extension of the notion of maximum packings introduced by Duane and Remmel \cite{DR}. We realized that one could reformulate the results of \cite{DR} in terms of the cluster method and that a similar reformulation of the results \cite{R} led to a natural generalization of the cluster method 
which we call the 
generalized cluster method.

 In this paper, we will illustrate our method by 
studying the distribution 
of certain consecutive patterns permutations in permutations 
which have regular descent patterns. Let $\sg = \sg_1 \ldots \sg_n$ 
be a permutation in the symmetric group $S_n$. Then we let 
\begin{eqnarray*}
Des(\sg) &=& \{i:\sg_i > \sg_{i+1}\}, \ \  \ \des(\sg) =|Des(\sg)|, \\
Ris(\sg) &=& \{i:\sg_i < \sg_{i+1}\}, \ \ \ \ris(\sg) =|Ris(\sg)|, \\ 
\inv(\sg) &=&  \sum_{1 \leq i < j \leq n} \chi(\sg_i > \sg_j) \  \mbox{and} \\ 
\coinv(\sg) &=& \sum_{1 \leq i < j \leq n} \chi(\sg_i < \sg_j),
\end{eqnarray*}
where for any statement $A$, $\chi(A)$ equals 1 if $A$ is true 
and $\chi(A) =0$ if $A$ is false. 

For any $n > 0$, we let 
\begin{align*}
&[n]_{q} = \frac{1-q^n}{1-q} =1+q+ \cdots + q^{n-1}, \\[1ex]
&[n]_{q}! =  [n]_q [n-1]_q \cdots [1]_q, \\[1ex]
&\qbinom{n}{k} = \frac{[n]_q!}{[k]_q![n-k]_q!}, \qquad \text{and} \\[1ex]
&\qbinom{n}{\la_1,\dots,\la_\ell} =
\frac{[n]_q!}{[\la_1]_q!\cdots [\la_\ell]_q!} 
\end{align*}
be the usual $q$-analogues of $n$, $n!$, $\binom{n}{k}$, 
and $\binom{n}{\la_1, \ldots, \la_{\ell}}$, respectively.

Given integers $k \geq 2$ and $i,j,n \geq 0$, we let 
$\mathcal{C}^{i,j,k}_{i+kn+j}$ denote the set of permutations 
$\sigma = \sg_1 \ldots \sg_{i+kn+j}$ in $S_{i+kn+j}$ such that 
$Des(\sg) \subseteq \{i,i+k, \ldots, i+nk\}$ and
$C^{i,j,k}_{i+kn+j}=|\mathcal{C}^{i,j,k}_{i+kn+j}|$.  Thus permutations 
in $\mathcal{C}^{i,j,k}_{i+kn+j}$ start with an increasing block 
of size $i$ followed by $n$ increasing blocks of size $k$ and end  
with an increasing block of size $j$. 

We let 
$\mathcal{E}^{i,j,k}_{i+kn+j}$ denote the set of
permutations  $\sigma \in S_{i+kn+j}$ with
$Des(\sg) = \{i,i+k, \ldots, i+nk\}$ and $E^{i,j,k}_{i+kn+j} = |\mathcal{E}^{i,j,k}_{i+kn+j}|$.  For any $\sg \in S_{i+kn+j}$, let
$$Ris_{i,k}(\sg) = \{s:0 \leq s \leq n \text{ and } \sg_{i+sk} < \sg_{i+sk+1}\}$$ and
$ris_{i,k}(\sg) =|Ris_{i,k}(\sg)|$.  Then $E^{i,j,k}_{i+k n+j}$ is the number of
$\sg \in \mathcal{C}^{i,j,k}_{i+kn+j}$ such that
$Ris_{i,k}(\sg) = \varnothing$. Thus permutations $\sg$ in $\mathcal{E}^{i,j,k}_{i+kn+j}$ 
have the same block structure as permutations in 
$\mathcal{C}^{i,j,k}_{i+kn+j}$, but we require the additional restriction 
that for any two consecutive blocks $B$ and $C$ in $\sg$, the last element 
of block $B$ must be larger than the first element of block $C$.

In the special case where $k=2$, $i=0$, and $j=2$,
$E^{0,2,2}_{2n+2}$ is the number of permutations in $S_{2n+2}$ with
descent set $\{2,4,\dots,2n\}$.  These are the classical
up-down permutations.  Andr\'{e} \cite{Andre1,Andre2}
proved that
\begin{equation*}
1+ \sum_{n\geq 0} \frac{E^{0,2,2}_{2n+2}}{(2n+2)!} t^{2n+2} = \sec t.
\end{equation*}
Similarly, $E^{0,1,2}_{2n+1}$ counts the number of odd up-down permutations and
 Andr\'{e} proved that 
\begin{equation*}
\sum_{n\geq 0} \frac{E^{0,1,2}_{2n+1}}{(2n+1)!} t^{2 n + 1} = \tan t.
\end{equation*}
These numbers are also called the Euler numbers. When $k > 2$,
$E^{0,j,k}_{kn+j}$ are called generalized Euler numbers \cite{LM}.
There are well-known generating functions for $q$-analogues of the
generalized Euler numbers; see Stanley's book \cite{StanBook}, page 148.
Various divisibility properties of the $q$-Euler numbers have been
studied in \cite{AF,AG,Foata} and of the generalized $q$-Euler
numbers in \cite{Gessel,SZ}. Prodinger \cite{P2} also studied
$q$-analogues of the number $E^{1,2,2}_{2n+1}$ and
$E^{1,1,2}_{2n+2}$.

Mendes, Remmel, and Riehl \cite{MenRemRie} computed the generating 
functions 
\begin{equation}
\label{q=1gf}
\sum_{n \geq 0}
\frac{t^{i+kn+j}}{[i+kn+j]_q!}\sum_{\sg \in \mathcal{C}^{i,j,k}_{i+kn+j}}
x^{ris_{i,k}(\sg)} q^{\inv(\sg)}.
\end{equation}
Setting $x=0$ in \eqref{q=1gf} gives the generating function 
\begin{equation}
\label{q=1gf2}
\sum_{n \geq 0}
\frac{t^{i+kn+j}}{[i+kn+j]_q!}\sum_{\sg \in \mathcal{E}^{i,j,k}_{i+kn+j}}
q^{\inv(\sg)}.
\end{equation}

For our purposes, we shall picture a $\sg$ in  
$\mathcal{C}^{i,j,k}_{i+kn+j}$ as an array $F(\sg)$ starting 
with a column of size $i$, followed by $n$ columns of size $k$, 
and ending with a column of size $j$ filled with the permutation 
$\sg$ so that  one recovers $\sg$ by reading the elements in 
each column from bottom to top and reading the columns from left to right. 
This means that in each column, the numbers are increasing when 
read from bottom to top. 
For example, the array associated with the 
permutation $$\sg = 2~5~6~8~9~1~7~10~4~11~12~3$$
in $\mathcal{C}^{2,1,3}_{12}$ is pictured in Figure \ref{fig:Cijk}. 
Elements of $\mathcal{E}^{i,j,k}_{i+kn+j}$ can be viewed as 
restricted arrays where the top element of each column has to 
be bigger than the bottom element in the column immediately to its right. 

\fig{1.20}{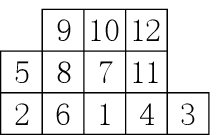}{The array for an element of 
$\mathcal{C}^{2,1,3}_{12}$.}

More generally,  let $D^{i,j,k}_{i+kn+j}$ 
denote the diagram which consists of a column of height $i$, followed 
by $n$ columns of height $k$, and ending with a column of height $j$. 
We let $(s,t)$ denote the cell which is in the $s^{th}$ column 
reading from left to right, and the $t^{th}$ row reading from 
bottom to top. For example, for the filling 
of $D^{2,1,3}_{12}$ is pictured in Figure \ref{fig:Cijk},  
the number $7$ is in cell $(3,2)$. Let $F^{i,j,k}_{i+kn+j}$ denote 
the set of all fillings of $D^{i,j,k}_{i+kn+j}$ with the 
elements $1, \ldots, i +kn +j$. The word $w(F)$ of any filling 
$F \in  F^{i,j,k}_{i+kn+j}$ is 
the word obtained by reading the columns from bottom to top and 
the columns from left to right. 
We let $\mathcal{P}^{i,j,k}_{i+kn+j}$ denote the set 
of all fillings of  $D^{i,j,k}_{i+kn+j}$ with the elements of $1, \ldots, 
i+kn+j$ such that the elements are increasing reading 
from bottom to top in each column. Thus for any 
$F \in \mathcal{P}^{i,j,k}_{i+kn+j}$, $w(F) \in \mathcal{C}^{i,j,k}_{i+kn+j}$.

In this paper, we will be mostly interested in patterns that occur 
between columns of height $k$ for elements of 
$\mathcal{P}^{i,j,k}_{i+kn+j}$. These types of patterns were first studied 
by Harmse and Remmel \cite{HR} for elements in $\mathcal{P}^{0,0,k}_{nk}$.

If $F$ is any filling of a $k \times n$-rectangle with pairwise distinct positive 
integers, then  
we let $\red(F)$ denote the filling which results from $F$ by replacing 
the $i$-th smallest element of $F$ by $i$. 
For example, Figure \ref{fig:red} we picture  
a filling, $F$, with its corresponding reduced filling, $\red(F)$.

\fig{1.30}{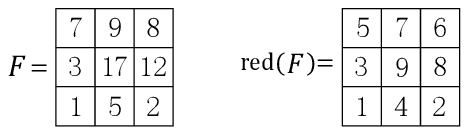}{An example of the reduction operation.}

If $F \in \mathcal{F}^{i,j,k}_{i+kn+j}$ and $2 \leq c_1 < \cdots < c_s\leq n+1$, then 
we let $F[c_1,\ldots,c_s]$ be the filling of the $k \times j$ rectangle 
where the elements in column $a$ of $F[c_1,\ldots,c_s]$ equal the elements 
in column $c_a$ in $F$ for $a = 1, \ldots, j$.   We can 
then extend the usual pattern matching definitions from permutations 
to elements of  $\mathcal{F}^{i,j,k}_{i+kn+j}$ as follows. 

\begin{definition}\label{def1} Let $\Gamma$ be a set of  elements of 
$\bigcup_{r \geq 2}\mathcal{F}^{0,0,k}_{kr}$ and 
$F \in \mathcal{F}^{i,j,k}_{i+kn+j}$. 
Then we say
\begin{enumerate}
\item  $\Gamma$ {\bf occurs} in $F$ if there is an $r \geq 1$ and 
$2 \leq i_1 < i_2 < \cdots < i_r \leq n+1$ such that
$\red(F[i_1, \ldots, i_r]) \in \Gamma$,

\item $F$ {\bf avoids} 
$\Gamma$ if there is no occurrence of $\Gamma$ in $F$, and   

\item there is a {\bf $\Gamma$-match  in $F$ starting at
column $i$} if there is an $r \geq 1$ such that \\
$\red(F[i,i+1, \ldots, i+r-1]) \in \Gamma$.
\end{enumerate}
\end{definition}
We let $\Gmch(F)$ denote the number of $\Gamma$-matches in $F$. 
In the case where $\Gamma =\{P\}$ is a singleton, we  
shall simply write $P$ occurs in $F$ ($P$ avoids $F$, $F$ has a $P$-match 
starting at column $i$) for $\Gamma$ occurs in $F$ ($\Gamma$ avoids $F$,
 $F$ has a $\Gamma$-match 
starting at column $i$).
When $i =j =0$ and $k=1$, then $\mathcal{P}^{0,0,1}_n=S_n$, where $S_n$ is the symmetric 
group, and our definitions reduce to the standard definitions that 
have appeared in the pattern matching literature. We note 
that Kitaev, Mansour, and Vella \cite{K1} have studied pattern matching 
in matrices which is a more general setting than the one we are considering 
for $i=j=0$ in this paper.  

We will illustrate our methods by finding generating functions 
for the distribution of such $\Gamma$-matches in various restricted 
arrays in $\mathcal{F}^{i,j,k}_{i+kn+j}$. To illustrate our methods 
in this paper, we will restrict ourselves to studying restrictions that arise 
by specifying some binary relation $\mathscr R$ between pairs of columns 
of integers. Then we let  
$\mathcal{P}^{i,j,k}_{i+kn+j,\mathscr R}$ denote the set of all 
elements $F$ in  $\mathcal{P}^{i,j,k}_{i+kn+j}$ such that for all 
$1 \leq i \leq n+1$, $(F[i],F[i+1]) \in \mathscr R$. For example, consider the following 
relations $\mathscr R$. 
\begin{enumerate}
\item Let $\mathscr R$ be the relation that holds between 
a pair if columns of integers $(C,D)$ if and only if 
the top element of $C$ is greater than the bottom element of 
$D$. Then it is easy to see that  $\mathcal{E}^{i,j,k}_{i+kn+j}$ 
equals $\mathcal{P}^{i,j,k}_{i+kn+j,\mathscr R}$. 

\item Let $\mathscr R$ be the relation that holds between 
a pair of columns of integers $(C,D)$ if and only if 
in the array $CD$, the rows are strictly increasing. 
Then it is easy to see that set of standard tableaux of shape 
$n^k$ equals $\mathcal{P}^{0,0,k}_{kn,\mathscr R}$.

\item Let $\mathscr R$ be the relation that holds between 
a pair of columns of integers $(C,D)$ of integers if and only if 
the bottom element in column $C$ is less than the bottom element of 
column $D$.  It is easy to see that if $i,j \geq 1$ and $k \geq 2$, then 
$\mathcal{P}^{i,j,k}_{i+kn+j,\mathscr R}$ is the set of all elements 
$F \in \mathcal{P}^{i,j,k}_{i+kn+j}$ such that 
$F(1,1) < F(2,1) < \cdots < F(n+2,1)$. 
\end{enumerate}
We should note that our methods can be applied to study more complicated 
restrictions than those that are determined by binary relations 
on columns of integers, but we shall not study such restrictions in this paper.

Our methods can also be applied to study more general 
types of patterns in such fillings. For example, suppose 
that we wanted to find the distribution of the pattern 
$Q=\begin{array}{|c|c|}
 \hline 3 & 4 \\
 \hline 1 & 2 \\
 \hline
 \end{array}$ in the first two rows of fillings 
in $\mathcal{P}^{2,3,4}_{25}$. Then for any $F \in \mathcal{P}^{2,3,4}_{25}$, 
we would let 
$$Q\text{-mch}(F) = |\{i: \red\left( \begin{array}{|c|c|}
 \hline F(i,2) & F(i+1,2) \\
 \hline F(i,1) & F(i+1,1) \\
 \hline
 \end{array} \right) =Q\}|.$$  
For example, for the $F 
\in  \mathcal{P}^{2,3,4}_{25}$ pictured in Figure \ref{fig:P23427}, 
$Q\text{-mch}(F) =3$ as there are $Q$-matches starting in 
columns 1, 3, and 6.

\fig{1.20}{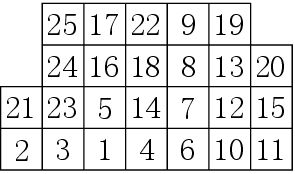}{$Q$-matches.}

The main goal of the paper is to study the generating 
functions of the distribution of $\Gamma$-matches 
in restricted arrays such at the ones described 
above.  For example, suppose that $\mathscr R$ is some binary relation 
on pairs of columns whose entries strictly increase, reading from 
bottom to top. Suppose that we are given a sequence 
of pairwise disjoint sets of patterns $\Gamma_1, 
\ldots, \Gamma_s$ such that for each $i$, 
$\Gamma_i \subseteq \mathcal{P}^{0,0,k}_{r_ik}$ where $r_i \geq 2$. 
Then we will be interested in computing generating 
functions of the form 
\begin{equation}
G^{i,j,k}_{\Gamma_1, \ldots, \Gamma_s, \mathscr R}(x_1, \ldots, x_s, t) = 
\sum_{n \geq 0} \frac{t^{i+kn+j}}{(i+kn+j)!} 
\sum_{F \in \mathcal{P}^{i,j,k}_{i+kn+j}} \prod_{r=1}^s x_r^{\Gamma_r\text{-}\mathrm{mch}(F)}.
\end{equation}
In the case where $s = 1$ and $\Gamma_1 =\{P\}$, then we shall let 
\begin{equation}
G^{i,j,k}_{P, \mathscr R}(x, t) = 
\sum_{n \geq 0} \frac{t^{i+kn+j}}{(i+kn+j)!} 
\sum_{F \in \mathcal{P}^{i,j,k}_{i+kn+j}} x^{\Pmch(F)}.
\end{equation}

The outline of this paper is as follows, 
In Section 2, we shall recall the cluster method and extend it to 
cover the types of filled arrays described above. We will 
also introduce the generalized cluster method to study generating 
functions of the form $G^{0,0,k}_{P, \mathscr R}(x, t)$. 
In section 3, we shall describe how one can extend the methods 
of Section 2 to study generating functions of the form 
$G^{i,0,k}_{P, \mathscr R}(x, t)$, $G^{0,j,k}_{P, \mathscr R}(x, t)$, 
and $G^{i,j,k}_{P, \mathscr R}(x, t)$ where $i,j > 0$ and $k \geq 2$. 
In section 4, we shall introduce the notion of joint clusters and joint 
generalized clusters to study generating functions 
of the form 
$G^{i,j,k}_{\Gamma_1, \ldots, \Gamma_s, \mathscr R}(x_1, \ldots, x_s, t)$.

\section{Clusters and Generalized Clusters for $\mathcal{P}^{0,0,k}_{kn}$.}

In this section, we shall describe the cluster method and the generalized 
cluster method for elements in $\mathcal{P}^{0,0,k}_{kn}$.

We start by recalling the definition of clusters for 
permutations. Let $\tau \in S_j$ be a permutation. Then 
for any $n \geq 1$, we let $\mathcal{M}S_{n,\tau}$ 
denote the set of all permutations in $\sg \in S_n$ where we have 
marked some of the $\tau$-matches in $\sg$ by placing an 
$x$ at the start of $\tau$-match in $\sg$. Given a $\sg \in \mathcal{M}S_{n,\tau}$, we 
let $m_{\tau}(\sg)$ be the number of marked $\tau$-matches in $\sg$. 
For example, 
suppose that $\tau = 132$ and $\sg = 15478263$. Then there are 
two $\tau$-matches in $\sg$, one starting at position 1 and one starting 
at position 6.  Thus $\sg$ gives rises to four elements 
of $\mathcal{M}S_{8,\tau}$.
$$\begin{array}{cc}
 15478263 &  \overset{x}{1}5478263 \\
 15478\overset{x}{2}63 & \overset{x}{1}5478\overset{x}{2}63
\end{array}$$
A {\bf $\tau$-cluster} is an element of $\sg = \sg_1 \ldots \sg_n \in \mathcal{M}S_{n,\tau}$ such that 
\begin{enumerate}
 \item every $\sg_i$ is an element of a marked 
$\tau$-match in $\sg$ and  
\item any two consecutive marked $\tau$-matches share at 
least one element. 
\end{enumerate}
We let $\mathcal{CM}S_{n,\tau}$ denote the set of all $\tau$-clusters 
in $\mathcal{M}S_{n,\tau}$. 

For each $n \geq 2$, we define  the cluster polynomial 
$$C_{n,\tau}(x) = \sum_{\sg \in \mathcal{CM}S_{n,\tau}} x^{m_{\tau}(\sg)}.$$
By convention, we set $C_{1,\tau}(x) =1$.

We say that a permutation 
$\tau \in S_j$ is {\bf minimal overlapping} if the smallest 
$n$ such that there exists a $\sg \in S_n$ where $\tmch(\sg) =2$ is 
$2j-1$. This means that two consecutive $\tau$-matches in a permutation 
$\sg$ can share at most one element which must be the element at 
the end of the first $\tau$-match and the element which is at the start of 
the second $\tau$-match. In such a situation, the smallest 
$m$ such that there exists a $\sg \in S_m$ such that $\tmch(\sg) =n$ is 
$n(j-1)+1$.  We call elements of $\sg \in S_{n(j-1)+1}$ such 
that $\tmch(\sg) =n$ {\bf maximum packings} of $\tau$.  We let 
$\mathcal{MP}_{n(j-1)+1}$ denote the set of maximum packings for $\tau$ in 
$S_{n(j-1)+1}$ and $mp_{n(j-1)+1,\tau} = |\mathcal{MP}_{n(j-1)+1}|$. 
It is easy 
to see that if $\tau \in S_j$ is minimal overlapping, then the only 
$\tau$-clusters are maximum packings for $\tau$ where the start of each 
$\tau$-match is marked with an $x$. For example, $\tau = 132$ is a minimal overlapping 
 permutation and it is easy to compute the number of 
maximum packings of size $2n+1$ for any $n \geq 1$. That is, 
if $\sg = \sg_1 \ldots \sg_{2n+1}$ is in $\mathcal{MP}_{2n+1,132}$, 
then there must be $132$-matches starting at positions $1,3, 5, \ldots, 2n-1$. 
It easily follows that for each $i =0, \ldots,n-1$, 
$\sg_{2i+1}$ is smaller than $\sg_j$ for 
all $j > 2i+1$. Hence $\sg_1 =1$ and $\sg_3 =2$. We then have 
$2n-1$ choices for $\sg_3$. Hence, it follows that  
$$mp_{2n+1,132} =(2n-1)mp_{2n-1,132} = \prod_{i=0}^{n-1}(2i+1).$$
Thus $C_{2n+1,132}(x) = x^n \prod_{i=0}^{n-1}(2i+1)$ for all $n \geq 1$ 
and $C_{2n,132}(x) =0$ for all $n \geq 1$. 

On the other hand, suppose that $\tau =1234$.  It is easy to see 
that for any $\tau$-cluster in $\mathcal{CM}S_{n,\tau}$ where  
$n \geq 4$, the underlying permutation must be the identity permutation. 
Moreover, if $\sg = \sg_1 \ldots \sg_n \in \mathcal{CM}S_{n,\tau}$, 
then $\sg_1$ must be marked with an $x$ because $\sg_1$ must be 
an element in a marked $\tau$-match and $\sg_{n-3}$ must be marked since 
$\sg_n$ must be an element of a marked $\tau$-match. Thus for 
$n =7$, we are forced to mark 1 and 4, $\displaystyle \overset{x}{1}23\overset{x}{4}567$. 
However we free to mark either 2 or 3 with an $x$. Hence $C_{7,1234}(x) = x^2(1+x)^2.$

Then for example, Elizalde and Noy \cite{EN} proved the following 
theorem which is a simple application of the Goulden-Jackson 
cluster method \cite{GJ1}. 
  
\begin{theorem}\label{thm:GJperm}
 Let $\tau \in S_j$ where $j \geq 2$. Then 
$$1+\sum_{n \geq 1} \frac{t^n}{n!} 
\sum_{\sg \in S_n} x^{\tmch(\sg)} = \frac{1}{1-t - \sum_{n \geq 2} 
\frac{t^n}{n!} C_{n,\tau}(x-1)}.$$
\end{theorem}

It is easy to generalize this result to deal with elements of 
$\mathcal{P}^{0,0,k}_{nk}$.   
Suppose that we are given a set of fillings 
$ \Gamma \subseteq \mathcal{P}^{0,0,k}_{tk}$ where $t \geq 2$.
For any $n \geq 1$, we let $\mathcal{MP}^{0,0,k}_{kn,\Gamma}$ 
denote the set of all fillings $F \in \mathcal{P}^{0,0,k}_{nk}$ where we have 
marked some of the $\Gamma$-matches in $F$ by placing an 
$x$ on top of the column that starts a $\Gamma$-match in $F$. 
For example, 
suppose that suppose $\Gamma = \{P\}$ where $P = \begin{array}{|c|c|c|}
\hline 
4 & 6 & 5 \\
\hline
1 & 2 & 3 \\
\hline
\end{array}$
and $F \in \mathcal{P}^{0,0,2}_{12}$ pictured in Figure \ref{fig:markedP}.
Then there are 
two $P$-matches in $F$, one starting at column 1 and one starting 
at column 4.  Thus $F$ gives rise to four elements 
of $\mathcal{MP}^{0,0,2}_{12}$. Given a $F \in \mathcal{MP}^{0,0,k}_{kn,\Gamma}$, 
we let $m_{\Gamma}(F)$ be the number of marked $\Gamma$-matches in $F$. 

\fig{1.20}{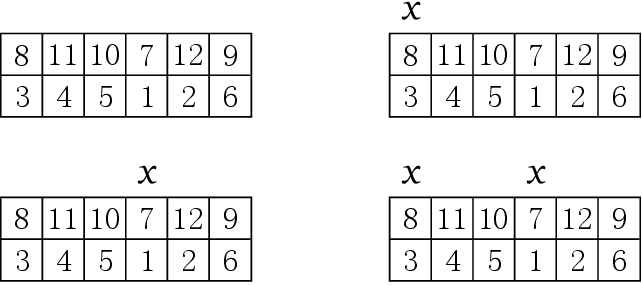}{$P$-marked fillings.}

We can also extend the reduction operation to $\Gamma$-marked fillings. That is, suppose 
$\Gamma \subseteq \mathcal{P}^{0,0,k}_{tk}$ where $t \geq 2$ and 
$F$ is a filling of the $k \times n$ array with integers which 
strictly increasing in columns, reading from bottom to top, where we 
have marked some of the $\Gamma$-matches by placing an $x$ at the top of the column 
that starts a marked $\Gamma$-match. Then by $\red(F)$, we mean the element 
of $\mathcal{MP}^{0,0,k}_{kn}$ that results by replacing the $i^{th}$ smallest 
element in $F$ by $i$ and marking a column in $\red(F)$ if and only if it 
is marked in $F$.

A {\bf $\Gamma$-cluster} is a filling  of $F \in \mathcal{MP}^{0,0,k}_{kn,P}$ such that 
\begin{enumerate}
 \item every column of $F$ is contained in a marked 
$\Gamma$-match of $F$ and  
\item any two consecutive consecutive marked $\Gamma$-matches share at 
least one column. 
\end{enumerate}
We let $\mathcal{CM}^{0,0,k}_{kn,\Gamma}$ denote the set of all $\Gamma$-clusters 
in $\mathcal{MP}^{0,0,k}_{kn,\Gamma}$. 
For each $n \geq 2$, we define the cluster polynomial 
$$C^{0,0,k}_{kn,\Gamma}(x) = \sum_{F \in \mathcal{CM}^{0,0,k}_{kn,\Gamma}} x^{m_{\Gamma}(F)}$$ where $m_{\Gamma}(F)$ is the number of marked $\Gamma$-matches in $F$. 
By convention, we let $C^{0,0,k}_{k,\Gamma}(x) =1$.

\begin{theorem}\label{thm:Gammacluster}
 Let $\Gamma \subseteq  \mathcal{P}^{0,0,k}_{rk}$ where $r \geq 2$. Then 
\begin{equation}\label{eq:pcluster}
1+\sum_{n \geq 1} \frac{t^n}{(kn)!} 
\sum_{F \in \mathcal{P}^{0,0,k}_{kn}} x^{\Gmch(F)} =  
\frac{1}{1-\sum_{n \geq 1} 
\frac{t^{kn}}{(kn)!} C^{0,0,k}_{kn,\Gamma}(x-1)}.
\end{equation}
\end{theorem}
\begin{proof}
Replace $x$ by $x+1$ in (\ref{eq:pcluster}).  Then the left-hand side 
of (\ref{eq:pcluster}) is the generating function 
of $m_\Gamma(F)$ over all $F \in \mathcal{MP}^{0,0,k}_{kn,\Gamma}$. That is, 
it easy to see that 
\begin{equation}\label{eq:cluster2} 
1+\sum_{n \geq 1} \frac{t^n}{(kn)!} 
\sum_{F \in \mathcal{P}^{0,0,k}_{kn}} (x+1)^{\Gmch(F)} = 
1+\sum_{n \geq 1} \frac{t^n}{(kn)!} 
\sum_{F \in \mathcal{MP}^{0,0,k}_{kn,\Gamma}} x^{m_\Gamma (F)}.
\end{equation}
Thus we must show that 
\begin{equation}\label{eq:cluster3}
 1+\sum_{n \geq 1} \frac{t^n}{(kn)!} 
\sum_{F \in \mathcal{MP}^{0,0,k}_{kn,\Gamma}} x^{m_\Gamma (F)} =
\frac{1}{1-\sum_{n \geq 1} 
\frac{t^{kn}}{(kn)!} C^{0,0,k}_{kn,\Gamma}(x)}.
\end{equation}
Now 
\begin{equation}\label{eq:cluster33}
\frac{1}{1-\sum_{n \geq 1} 
\frac{t^{kn}}{(kn)!} C^{0,0,k}_{kn,\Gamma}(x)} = 
1+ \sum_{m \geq 1} \left(\sum_{n \geq 1} 
\frac{t^{kn}}{(kn)!} C^{0,0,k}_{kn,\Gamma}(x)\right)^m.
\end{equation} 
Taking the coefficient of $\frac{t^{ks}}{(ks)!}$ on both sides of 
(\ref{eq:cluster3}) where $n \geq 1$, we see that we must show that 
\begin{eqnarray}\label{eq:cluster4}
\sum_{F \in \mathcal{MP}^{0,0,k}_{sn,\Gamma}} x^{m_\Gamma (F)} &=& 
\sum_{m=1}^\infty \left(\sum_{n=1}^\infty  
\frac{t^{kn}}{(kn)!} C^{0,0,k}_{kn,\Gamma}(x) \right)^m|_{\frac{t^{ks}}{(ks)!}} 
\nonumber \\
&=& \sum_{m=1}^s \left(\sum_{n = 1}^s  
\frac{t^{kn}}{(kn)!} C^{0,0,k}_{kn,\Gamma}(x)\right)^m|_{\frac{t^{ks}}{(ks)!}}
\nonumber \\
&=&  \sum_{m =1}^s \sum_{\overset{a_1+ a_2 + \cdots +a_m=s}{a_i \geq 1}} 
\binom{ks}{ka_1,\ldots,ka_m} \prod_{j=1}^m C^{0,0,k}_{ka_j,\Gamma}(x).
\end{eqnarray}

The right-hand side of (\ref{eq:cluster4}) is now easy to interpret. 
First we pick an $m$ such that $1 \leq m \leq s$. Then we 
pick $a_1, \ldots, a_m \geq 1$ such that $a_1+a_2+ \cdots +a_m = s$. 
Next the binomial coefficient $\binom{kn}{ka_1,\ldots,ka_m}$ allows 
us to pick a sequence of sets $S_1, \ldots, S_m$ which partition 
$\{1, \ldots, ks\}$ such that $|S_i| =ka_i$ for $i=1, \ldots ,m$. 
Finally the product $\prod_{j=1}^m C^{0,0,k}_{ka_j,\Gamma}(x)$ allows 
us to pick clusters $C_i \in \mathcal{CM}^{0,0,k}_{ka_i,\Gamma}$ for 
$i=1,\ldots, m$ with weight $\prod_{j=1}^m x^{m_\Gamma(C_i)}$. Note 
that in the cases where $a_i =1$, we will interpret 
$C_i$ as just a column of height $k$ filled with the numbers $1, \ldots, k$ 
which is increasing, reading from bottom to top.

\fig{1.20}{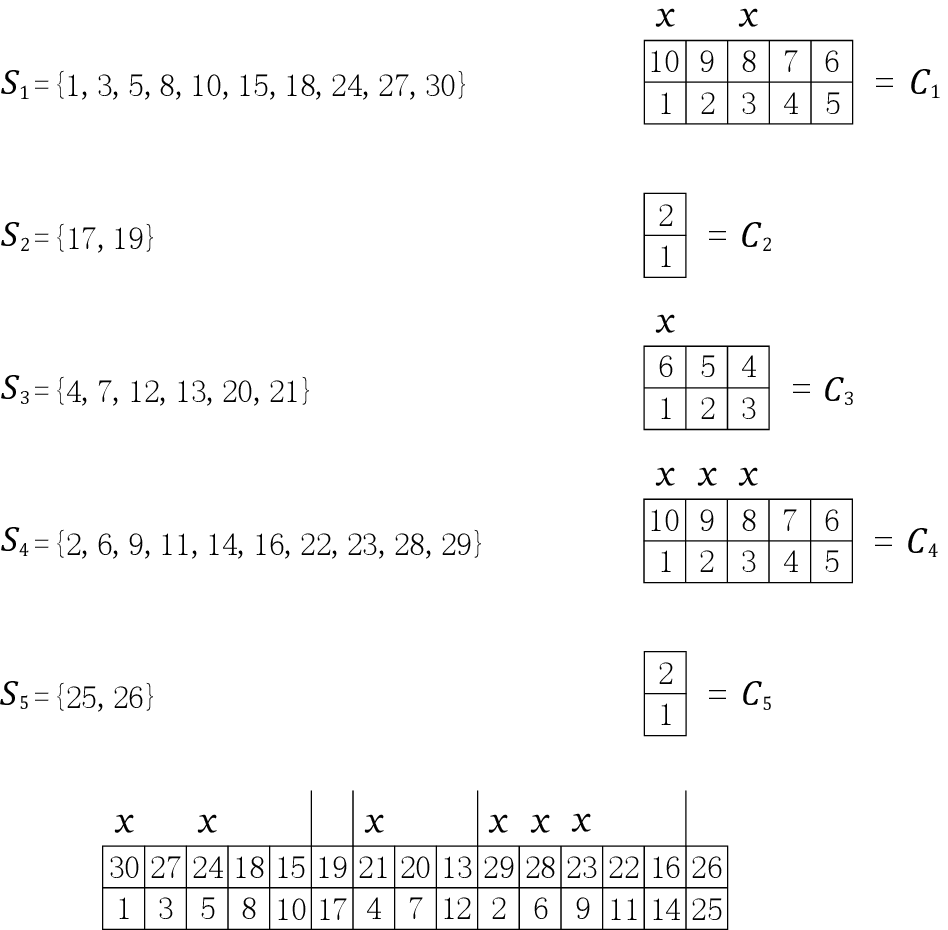}{Construction for the right-hand side of (\ref{eq:cluster4}).}

For example, suppose that $k=2$ and $\Gamma = \{P\}$ where 
$P = \begin{array}{|c|c|c|}
\hline 
6 & 5 & 4 \\
\hline
1 & 2 & 3 \\
\hline
\end{array}$.
Then in Figure \ref{fig:Cluster}, we have pictured 
$S_1, S_2, S_3, S_4, S_5$ which partition $\{1, \ldots, 30\}$ and 
corresponding clusters $C_1, \ldots, C_5$. Then for each $i =1, 
\ldots, m$, we create a cluster $D_i$ which results by replacing 
each $j$ in $C_i$ by the $j^{th}$ smallest element of $S_i$.  If we 
concatenate $D_1 \ldots D_m$ together, then we will obtain 
an element of $Q \in \mathcal{M}_{kn,\Gamma}^{0,0,k}$. It is easy to see 
that one can recover $D_1, \ldots, D_5$ from $Q$.  That is, 
given an element $F \in \mathcal{M}^{0,0,k}_{kn,\Gamma}$, we say that a marked 
subsequence $F[i,i+1, \ldots, j]$ is a {\bf maximal $\Gamma$-subcluster} of $F$ if 
$\red(F[i,i+1, \ldots, j])$ is a $\Gamma$-cluster and $F[i,i+1, \ldots, j]$ is not 
properly contained in a marked subsequence $F[a,a+1, \ldots, b]$ such that 
$\red(F[a,a+1, \ldots, b])$ is a $\Gamma$-cluster. In the special case 
where $i=j$ and the column $F[i]$ is not marked, then we say  
$F[i]$ is maximal $\Gamma$-subcluster if $F[i]$ is not 
properly contained in a marked subsequence $F[a,a+1, \ldots, b]$ such that 
$\red(F[a,a+1, \ldots, b])$ is a $\Gamma$-cluster. Thus $D_1, \ldots ,D_5$ are 
the maximal $\Gamma$-subclusters of $Q$. Of course, once we have 
recovered $D_1, \ldots, D_5$, we can recover the sets 
$S_1, \ldots, S_5$ and the $\Gamma$-clusters $C_1, \ldots, C_5$. 

In this manner, we can see that 
the right-hand side of (\ref{eq:cluster4}) just classifies 
the elements of $\mathcal{M}^{0,0,k}_{kn,\Gamma}$ by its maximal $\Gamma$-subclusters 
which proves our theorem. 

\end{proof}

Next suppose that we are given a binary relation $\mathscr R$ between 
$k \times 1$ arrays of integers and a set of patterns 
$\Gamma \subseteq   \mathcal{P}^{0,0,k}_{rk}$ where $r \geq 2$.  
\begin{definition}\label{def:gc}
We say that $Q \in \mathcal{MP}^{0,0,k}_{kn,\Gamma}$ is a 
{\bf generalized $\Gamma,\mathscr R$-cluster} if 
we can write \\
$Q=B_1B_2\cdots B_m$ where $B_i$ are blocks of consecutive columns 
in $Q$ such that  
\begin{enumerate}
	\item either $B_i$ is a single column or $B_i$ consists of $r$-columns 
where $r \geq 2$, $\red(B_i)$ is a $\Gamma$-cluster in $\mathcal{MP}_{kr,\Gamma}$, 
and any pair of consecutive columns in $B_i$ are in $\mathscr R$ and
	\item for $1\leq i\leq m-1$, the pair $(\last(B_i),\first(B_{i+1})$ is 
not in $\mathscr R$ where for any $j$, $\last(B_j)$ is the right-most column of $B_j$ and $\first(B_j)$ is 
the left-most column of $B_j$.  
\end{enumerate}
\end{definition}
Let $\mathcal{GC}^{0,0,k}_{kn,\Gamma, \mathscr R}$ denote the set of all generalized 
$\Gamma,\mathscr R$-clusters which have $n$ columns of height $k$. For example, 
suppose that $\mathscr R$ is the relation that holds for a pair of columns $(C,D)$ if 
and only if the top element of column $C$ is greater than the bottom element 
of column $D$ and $\Gamma = \{P\}$ where $P = \begin{array}{|c|c|c|}
\hline 
6 & 5 & 4 \\
\hline
1 & 2 & 3 \\
\hline
\end{array}$.
Then in Figure \ref{fig:PRcluster}, we have pictured a generalized $\Gamma,\mathscr R$-cluster with 
5 blocks $B_1,B_2,B_3,B_4,B_5$.

\fig{1.20}{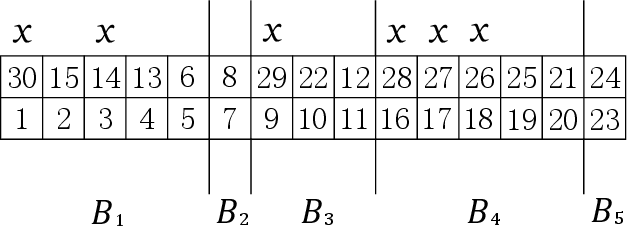}{A generalized $\Gamma,\mathscr R$-cluster.}

Given $Q =B_1 B_2 \ldots B_m \in \mathcal{GC}^{0,0,k}_{kn,\Gamma, \mathscr R}$, we 
define the weight of $B_i$, $w_{\Gamma, \mathscr R}(B_i)$, to be 1 if $B_i$ is a single column and 
$x^{m_{\Gamma}(\red(B_i))}$ if $B_i$ is order isomorphic to a $\Gamma$-cluster. Then we 
define the weight of $Q$, $w_{\Gamma, \mathscr R}(Q)$, to be $(-1)^{m-1}\prod_{i=1}^m 
w_{\Gamma, \mathscr R}(B_i)$. We 
let 
\begin{equation}\label{GPRC}
GC^{0,0,k}_{kn,\Gamma, \mathscr R}(x) = \sum_{Q \in \mathcal{GC}^{0,0,k}_{kn,\Gamma, \mathscr R}} w_{\Gamma,\mathscr R}(Q).
\end{equation}

Then we have the following theorem. 

\begin{theorem}\label{thm:mainGamma}
Let $\mathscr R$ be a binary relation on pairs of columns $(C,D)$ of height $k$ 
which are filled 
with integers which are increasing from bottom to top. 
 Let $\Gamma \subseteq \mathcal{P}^{0,0,k}_{rk}$ where $r \geq 2$. Then 
\begin{equation}\label{eq:prcluster}
1+\sum_{n \geq 1} \frac{t^{kn}}{(kn)!} 
\sum_{F \in \mathcal{P}^{0,0,k}_{kn,\mathscr R}} x^{\Gmch(F)} =  
\frac{1}{1-\sum_{n \geq 1} 
\frac{t^{kn}}{(kn)!} GC^{0,0,k}_{kn,\Gamma,\mathscr R}(x-1)}.
\end{equation}
\end{theorem}

\begin{proof}
Replace $x$ by $x+1$ in (\ref{eq:prcluster}).  Then 
it easy to see that 
\begin{equation}\label{eq:prcluster2} 
1+\sum_{n \geq 1} \frac{t^{kn}}{(kn)!} 
\sum_{F \in \mathcal{P}^{0,0,k}_{kn,\mathscr R}} (x+1)^{\Gmch(F)} = 
1+\sum_{n \geq 1} \frac{t^{kn}}{(kn)!} 
\sum_{F \in \mathcal{MP}^{0,0,k}_{kn,\Gamma,\mathscr R}} x^{m_\Gamma(F)}.
\end{equation}
Thus we must show that 
\begin{equation}\label{eq:prcluster3}
 1+\sum_{n \geq 1} \frac{t^{kn}}{(kn)!} 
\sum_{F \in \mathcal{MP}^{0,0,k}_{kn,\Gamma,\mathscr R}} x^{m_\Gamma(F)} =
\frac{1}{1-\sum_{n \geq 1} 
\frac{t^{kn}}{(kn)!} GC^{0,0,k}_{kn,\Gamma,\mathscr R}(x)}.
\end{equation}
Now 
\begin{equation}\label{eq:prcluster33}
\frac{1}{1-\sum_{n \geq 1} \frac{t^{kn}}{(kn)!} GC^{0,0,k}_{kn,\Gamma,\mathscr R}(x)} = 
1+ \sum_{m \geq 1} \left(\sum_{n \geq 1} 
\frac{t^{kn}}{(kn)!} GC^{0,0,k}_{kn,\Gamma,\mathscr R}(x)\right)^m.
\end{equation} 
Taking the coefficient of $\frac{t^{ks}}{(ks)!}$ on both sides of 
(\ref{eq:cluster3}) where $n \geq 1$, we see that we must show that 
\begin{eqnarray}\label{eq:prcluster4}
\sum_{F \in \mathcal{MP}^{0,0,k}_{sn}} x^{m_\Gamma(F)} &=& 
\sum_{m=1}^\infty \left(\sum_{n \geq 1} 
\frac{t^{kn}}{(kn)!} GC^{0,0,k}_{kn,\Gamma,\mathscr R}(x)\right)^m|_{\frac{t^{ks}}{(ks)!}} 
\nonumber\\
&=& \sum_{m=1}^s \left(\sum_{n = 1}^s  
\frac{t^{kn}}{(kn)!} GC^{0,0,k}_{kn,\Gamma,\mathscr R}(x)\right)^m|_{\frac{t^{ks}}{(ks)!}}
\nonumber \\
&=& \sum_{m =1}^s \sum_{\overset{a_1+ a_2 + \cdots +a_m=s}{a_i \geq 1}} 
\binom{ks}{ka_1,\ldots,ka_m} \prod_{j=1}^m GC^{0,0,k}_{ka_j,\Gamma,r}(x).
\end{eqnarray}
The right-hand side of (\ref{eq:prcluster4}) is now easy to interpret. 
First we pick an $m$ such that $1 \leq m \leq s$. Then we 
pick $a_1, \ldots, a_m \geq 1$ such that $a_1+a_2+ \cdots +a_m = s$. 
Next the binomial coefficient $\binom{ks}{ka_1,\ldots,ka_m}$ allows 
us to pick a sequence of sets $S_1, \ldots, S_m$ which partition 
$\{1, \ldots, ks\}$ such that $|S_i| =ka_i$ for $i=1, \ldots ,m$. 
Finally the product $\prod_{j=1}^m GC^{0,0,k}_{ka_j,\Gamma,\mathscr R}(x)$ allows 
us to pick generalized $\Gamma,\mathscr R$-clusters $G_i \in \mathcal{GC}^{0,0,k}_{ka_i,\Gamma,\mathscr R}$ for 
$i=1,\ldots, m$ with weight $\prod_{j=1}^m w_{\Gamma,\mathscr R}(G_i)$. Note 
that in the cases where $a_i =1$, our definitions imply that  
$C_i$ is just a column of height $k$ filled with the numbers $1, \ldots, k$ 
which is increasing, reading from bottom to top.

\fig{1.20}{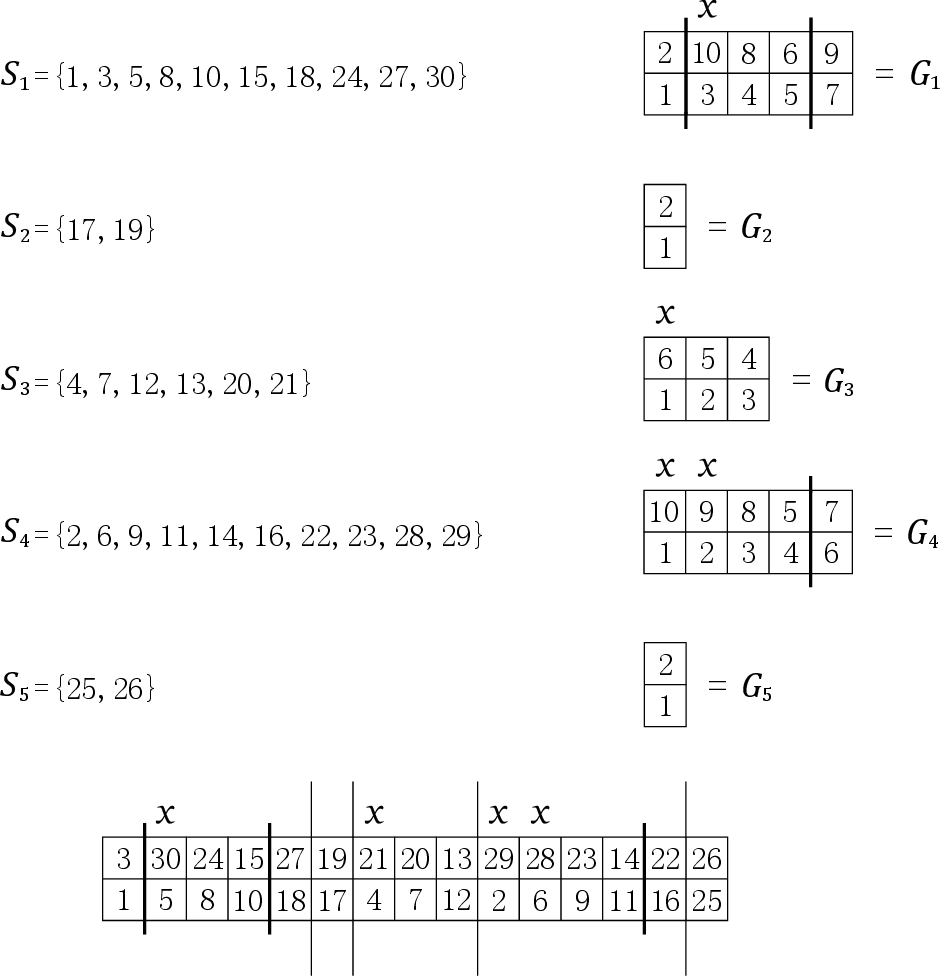}{Construction for the right-hand side of (\ref{eq:prcluster4}).}

For example, suppose that $k=2$ and $\Gamma =\{P\}$ where 
$P = \begin{array}{|c|c|c|}
\hline 
6 & 5 & 4 \\
\hline
1 & 2 & 3 \\
\hline
\end{array}$.
Suppose that $\mathscr R$ is relation where,  for any two columns $C$ and $D$ which filled with 
integers and are strictly increasing in columns, $(C,D) \in \mathscr R$ if and only if 
the top element of $C$ is greater than the bottom element of $D$. 
Then in Figure \ref{fig:PRCluster2}, we have pictured 
$S_1, S_2, S_3, S_4, S_5$ which partition $\{1, \ldots, 30\}$ and 
corresponding generalized $\Gamma,\mathscr R$-clusters $G_1, \ldots, G_5$. For each $i$, 
we have indicated the separation between the blocks of $G_i$ by dark black 
lines.  Then for each $i =1, 
\ldots, m$, we create a cluster $E_i$ which results by replacing 
each $j$ in $G_i$ by the $j^{th}$ smallest element of $S_i$.  If we 
concatenate $E_1 \ldots E_5$ together, then we will obtain 
an element of $Q \in \mathcal{M}_{kn,\Gamma}^{0,0,k}$. The weight 
of $Q$ equals $\prod_{j=1}^5 w_{\Gamma,\mathscr R}(G_i)$ where 
\begin{eqnarray*}
w_{\Gamma,\mathscr R}(G_1) &=& (-1)^2 x, \\
w_{\Gamma,\mathscr R}(G_2) &=& 1, \\
w_{\Gamma,\mathscr R}(G_3) &=& x, \\
w_{\Gamma,\mathscr R}(G_4) &=& (-1)^1 x^2, \ \mbox{and}\\
w_{\Gamma,\mathscr R}(G_5) &=& 1.
\end{eqnarray*}
In Figure \ref{fig:PRCluster2}, we have indicated the boundaries between 
the $E_i$s by 
light lines. 

We let $\mathcal{HGC}_{ks,\Gamma,\mathscr R}$ denote the set of all elements that can be constructed 
in this way. Thus $Q =E_1 \ldots E_m$ is an element of 
$\mathcal{HGC}_{ks,\Gamma,\mathscr R}$ if and only if for each $i =1, \ldots, m$, 
$\red(E_i)$ is a generalized $\Gamma,\mathscr R$-cluster. Next we define a sign 
reversing involution $\theta:\mathcal{HGC}_{ks,\Gamma,\mathscr R} \rightarrow \mathcal{HGC}_{ks,\Gamma,\mathscr R}$. 
Given $Q =E_1 \ldots E_m \in \mathcal{HGC}_{ks,\Gamma,\mathscr R}$, look for the first $i$ such that either 
\begin{enumerate}
\item the block structure of $\red(E_i) = B^{(i)}_1 \ldots B^{(i)}_{k_i}$ consists 
of more than one block or 
\item $E_i$ consists of a single block $B^{(i)}_1$ and $(\last(B^{(i)}_1),\first(E_{i+1}))$ is 
not in $\mathscr R$.
\end{enumerate}
In case (1), we let $\theta(E_1 \ldots E_m)$ 
be the result of  replacing $E_i$ by two generalized $\Gamma,\mathscr R$-clusters, $E_i^*$ and 
$E_i^{**}$ where $E_i*$ consists just of $B^{(i)}_1$ and $E_i^{**}$ consists 
of $B^{(i)}_2 \ldots B^{(i)}_{k_i}$. Note that in this case 
$w_{\Gamma,\mathscr R}(E_i) = (-1)^{k_i-1}\prod_{j=1}^{k_i} w_{\Gamma,\mathscr R}(B^{(i)}_j)$ while 
$w_{\Gamma,\mathscr R}(E_i^*)w_{\Gamma,\mathscr R}(E_i^{**}) =(-1)^{k_i-2}\prod_{j=1}^{k_i} 
w_{\Gamma,\mathscr R}(B^{(i)}_j)$. 
In case (2), we let $\theta(E_1 \ldots E_m)$ 
be the result of  replacing  $E_i$ and $E_{i+1}$ by the single generalized 
$\Gamma,\mathscr R$-cluster $E= B^{(i)}_1 B^{(i+1)}_1 \ldots B^{(i+1)}_{k_{i+1}}$. Note 
that since $(\last(B^{(i)}_1),\first(E_{i+1}))$ is 
not in $\mathscr R$, $B^{(i)}_1 B^{(i+1)}_1 \ldots B^{(i+1)}_{k_{i+1}}$ reduces 
to a generalized $\Gamma,\mathscr R$-cluster. In this 
case,  $w_{\Gamma,\mathscr R}(E_i) w_{\Gamma,\mathscr R}(E_{i+1}) = (-1)^{k_{i+1}-1} w_{\Gamma,\mathscr R}(B^{(i)}_1) 
\prod_{j=1}^{k_{i+1}} w_{\Gamma,\mathscr R}(B^{(i+1)}_j)$ and  
$w_{\Gamma,\mathscr R}(E) = (-1)^{k_{i+1}} w_{\Gamma,\mathscr R}(B^{(i)}_1) 
\prod_{j=1}^{k_{i+1}} w_{\Gamma,\mathscr R}(B^{(i+1)}_j)$.
If neither case (1) or case (2) applies, then we let $\theta(E_1 \ldots E_m) =
E_1 \ldots E_m$.  For example, suppose that $\mathscr R$ is the binary 
relation where, for any two columns $C$ and $D$, which filled with 
integers and are strictly increasing in columns, $(C,D) \in \mathscr R$ if and only if 
the top element of $C$ is greater than the bottom element of $D$ and 
$\Gamma = \{P\}$ where $P = \begin{array}{|c|c|c|}
\hline 
6 & 5 & 4 \\
\hline
1 & 2 & 3 \\
\hline
\end{array}$.
Then  if $Q=E_1 \ldots E_5$ is the 
generalized $\Gamma,\mathscr R$-cluster pictured in Figure \ref{fig:PRCluster2}, 
then we are in case (1) with $i=1$ since $E_1$ consists of more than one block. 
Thus $\theta(Q)$ results by breaking that generalized $\Gamma,\mathscr R$-cluster into to 
two clusters $E_1^*$ of size 1 and $E^{**}_1$ of size 4. $\theta(Q)$ is 
pictured in Figure \ref{fig:PRCluster3}.

\fig{1.20}{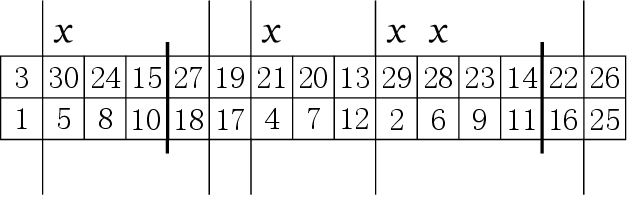}{The involution $\theta$.}

It is easy to see that $\theta$ is an involution. That is, if 
$Q=E_1 \ldots E_m$ is in case (1) using $E_i$, then $\theta(Q)$ will be 
in case (2) using $E_i^*$ and $E_i^{**}$.  Similarly 
if $Q=E_1 \ldots E_m$ is in case (1) using $E_i$ and $E_{i+1}$, then $\theta(Q)$ will be 
in case (2) using $E =E_iE_{i+1}$. It follows that if 
$\theta (E_1 \ldots E_m) \neq E_1 \ldots E_m$, then 
$w_{\Gamma,\mathscr R}(E_1 \ldots E_m) = - w_{\Gamma,\mathscr R}(\theta(E_1 \ldots E_m))$ so 
that the right-hand side of (\ref{eq:prcluster4}) equals 
$$\sum_{Q = E_1 \ldots E_m \in \mathcal{HGC}_{ks,\Gamma,\mathscr R}, \theta(Q) =Q}
\prod_{i=1}^m w_{\Gamma,\mathscr R}(E_i).$$
Thus we must examine the fixed points of $\theta$.  

If $Q = E_1 \ldots E_m \in \mathcal{HGC}_{ks,\Gamma,\mathscr R}$ and $\theta(Q) =Q$, 
then it must be the case that each $E_i$ consists of single column of weight 1 or 
it reduces to generalized 
$\Gamma,\mathscr R$-cluster $\overline{E}_i$ consisting of a single block $B^{(i)}_1$
whose weight is the weight of $\red(B^{(i)}_1)$ as
a $\Gamma$-cluster. Moreover, it must be the case that for all $i=1, \ldots m-1$, 
$(\last(E_i),\first(E_{i+1})$ is in $\mathscr R$. But this means for all $j =1, \ldots, n-1$, 
$(Q[j],Q[j+1])$ is in $\mathscr R$.  That is, either $Q[j]$ equals $\last(E_i)$ for some $i$ 
or column $j$ is contained in one of the $\Gamma$-clusters $E_i$ in which 
case $(Q[j],Q[j+1])$ is in $\mathscr R$ by our definition of generalized $\Gamma,\mathscr R$-clusters. Thus any 
fixed point $Q$ of $\theta$ is an element $\mathcal{MP}^{0,0,k}_{ks,\mathscr R}$. Then just like 
our proof Theorem \ref{thm:Gammacluster}, it follows that  $E_1, \ldots, E_m$ are just 
the maximal $\Gamma$-subclusters of an element in $\mathcal{P}^{0,0,k}_{ks,\mathscr R}$. 
Vice versa, if $T =F_1 \ldots F_r$ is an element of $\mathcal{P}^{0,0,k}_{ks,\mathscr R}$ 
where $F_1, \ldots, F_r$ are the maximal $\Gamma$-subclusters of $T$, then 
$T = F_1 \ldots F_r$ is a fixed point of $\theta$. Thus we have proved 
that the right-hand side of (\ref{eq:prcluster4}) equals 
$$\sum_{F \in \mathcal{MP}^{0,0,k}_{ks,\mathscr R}} x^{m_\Gamma(F)}$$
which is what we wanted to prove. 
\end{proof}

We note that Theorem \ref{thm:Gammacluster} is a special case 
of Theorem \ref{thm:mainGamma} in the case where $\mathscr R$ holds 
for any pair of columns $(C,D)$.

\subsection{Examples}

We end this section with two examples.  Our goal is to study patterns 
in generalized Euler permutations $E^{0,0,k}_{kn}$. Thus the relation 
that we want to consider is that if $C$ and $D$ are fillings of column 
of height $k$ in which the numbers increase, reading from bottom to top, 
then $(C,D) \in \mathscr R$ if and only if the top element of column 
$C$ is greater than the bottom element of column $D$. Let $A_{k,3}$ be the 
element of $\mathcal{P}^{0,0,k}_{3}$ whose $2j+1^{\mathrm{st}}$ row consists 
of the elements $6j+1,6j+2,6j+3$, reading from left to right. and 
whose $2j+2^{\mathrm{th}}$ row consists of the $6j+4,6j+5,6j+6$, reading from 
right to left. For example, we have pictured $A_{3,3}$ and $A_{4,3}$ 
in Figure \ref{fig:Adiagrams}. There is a natural poset associated 
to $A_{3,3}$ and $A_{4,3}$. The Hasse diagrams of 
such posets are obtained by having a node for each integer and drawing 
a directed arrow from the node corresponding to $i$ to the node corresponding 
to $i+1$ for each $i \geq 1$.  The Hasse diagrams of the 
posets corresponding to $A_{3,3}$ and $A_{4,3}$ are pictured in 
Figure \ref{fig:Adiagrams}.

\fig{1.20}{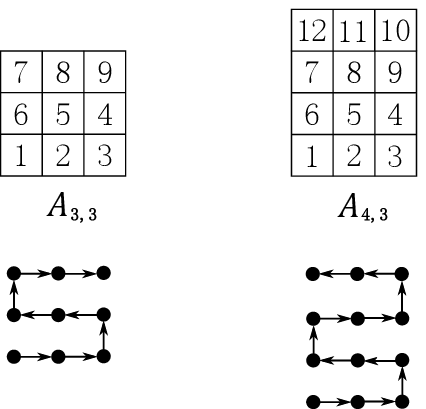}{$A_{3,3}$ and $A_{4,3}$.}

It is easy to see that in a $A_{k,3}$-cluster, the fact that every column 
is part of marked cluster and the fact that any two consecutive marked 
clusters must share a column implies that in a $A_{k,3}$-cluster with $n$ 
columns, the elements in the first row are the $n$ smallest elements, 
the elements in the second row are the next $n$ smallest elements, etc. 
Moreover, elements in odd rows are increasing, reading from left to right, 
and the elements in even rows are increasing, reading from right to left. 
For example, we have pictured the Hasse diagrams corresponding  
$A_{3,3}$-clusters at the top of Figure \ref{fig:Aclusters} and 
the Hasse diagrams corresponding to $A_{4,3}$-clusters 
at the bottom of Figure \ref{fig:Aclusters}. 
This means that in a $A_{k,3}$-cluster with $n$ columns, 
there are $A_{k,3}$-matches starting a positions $1, \ldots, n-2$. 
Thus to determine $C_{A_{k,3}}(x)$, we need only have to consider the 
possible ways to mark $A_{k,3}$-matches. Clearly we must 
mark the $A_{k,3}$-match starting at column 1 and the 
$A_{k,3}$-match starting at column $n-2$ since the first and last column 
must be part of marked $A_{k,3}$-matches. It follows 
that for all $k \geq 2$, 
\begin{eqnarray*}
C^{0,0,k}_{k,A_{k,3}}(x) &=& 1, \\
C^{0,0,k}_{2k,A_{k,3}}(x) &=& 0, \\
C^{0,0,k}_{3k,A_{k,3}}(x) &=& x, \ \mbox{and} \\
C^{0,0,k}_{4k,A_{k,3}}(x) &=& x^2.
\end{eqnarray*}
For $C^{0,0,k}_{nk,A_{k,3}}(x)$, for $k \geq 5$, there are two cases depending 
on whether the $A_{k,3}$-match starting at column $n-3$ is marked or not. 
If the $A_{k,3}$-match starting at column $n-3$ is marked, then we can remove 
the last column, the elements $3n,3n-1,3n-2$, and the mark on column $n-2$ 
to obtain of $A_{k,3}$-cluster with $n-1$-columns after we reorder the elements. 
If the $A_{k,3}$-match starting at column $n-3$ is not marked, then we can remove 
the last two columns, the elements $3n,3n-1,3n-2,3n-3,3n-4,3n-5$, 
and the mark on column $n-2$ 
to obtain of $A_{k,3}$-cluster with $n-2$-columns after we reorder the elements. 
These two cases are pictured in Figure \ref{fig:Aclusters}. It follows 
that for $n \geq 5$, 
\begin{equation}
C^{0,0,k}_{nk,A_{k,3}}(x) = 
x C^{0,0,k}_{(n-1)k,A_{k,3}}(x) + xC^{0,0,k}_{(n-2)k,A_{k,3}}(x). 
\end{equation}
Then one can easily compute that for all $k \geq 2$, 
\begin{eqnarray*}
C^{0,0,k}_{5k,A_{k,3}}(x) &=& x^2+x^3, \\
C^{0,0,k}_{6k,A_{k,3}}(x) &=& 2x^3+x^4, \\
C^{0,0,k}_{7k,A_{k,3}}(x) &=& x^3+3x^4+x^5, \\
C^{0,0,k}_{8k,A_{k,3}}(x) &=& 3x^4+4x^5+x^6, \\
C^{0,0,k}_{9k,A_{k,3}}(x) &=& x^4+6x^5+5x^6+x^7, \ \mbox{and} \\
C^{0,0,k}_{10k,A_{4k,3}}(x) &=& 4x^5+10x^6+6x^7+x^8.
\end{eqnarray*}

\fig{1.20}{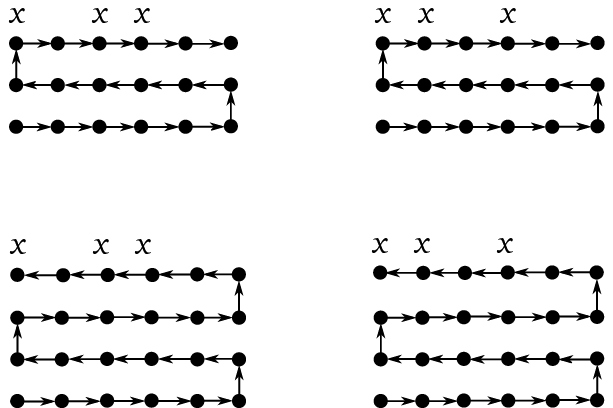}{$A_{3,3}$-clusters  and $A_{4,3}$-clusters.}

Next we consider generalized $A_{k,3},\mathscr R$-clusters. There are two types 
depending on whether $k$ is even or odd. 
That is, when $k =2j+1$ is odd, then it is easy to see that for each 
$A_{2j+1,3}$-cluster $C$ with more than 2 columns, the smallest element 
in the cluster occupies the bottom left-hand position and the largest 
element in the cluster occupies the top right-hand position. In 
a generalized $A_{2j+1,3}, \mathscr R$-cluster with blocks $B_1B_2 \ldots B_r$, 
we must have that the top right-hand element of $B_i$ is less than 
the bottom left-most element of $B_{i+1}$ for $i =1, \ldots, r-1$. 
If we draw the Hasse diagram of such a generalized cluster, it is 
easy to see that these conditions force a total order on elements 
of the generalized  $A_{2j+1}, \mathscr R$-cluster. See for example, 
Figure \ref{fig:Aoddclusters} where we have pictured two different 
 $A_{3,3}, \mathscr R$-clusters.

\fig{1.20}{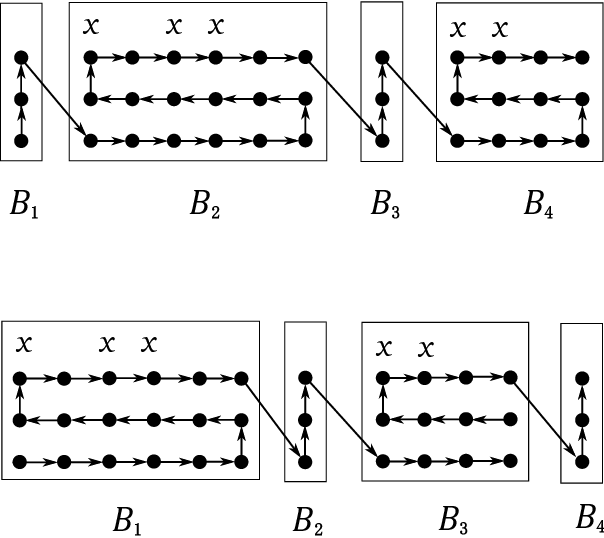}{Generalized $A_{3,3},\mathscr R$-clusters.}

This means that we can develop simple recursions for the 
generalized cluster polynomials $GC^{0,0,2j+1}_{(2j+1)n,A_{2j+1,3},\mathscr R}(x)$. 
That is, it easy to check that 
\begin{eqnarray*}
GC^{0,0,2j+1}_{(2j+1),A_{2j+1,3},\mathscr R}(x) &=& 1 \ \mbox{and} \\ 
GC^{0,0,2j+1}_{(2j+1)2,A_{2j+1,3},\mathscr R}(x) &=& -1.
\end{eqnarray*}
Then for $n \geq 3$, we can classify the generalized 
$A_{2j+1,3}, \mathscr R$-clusters with $n$ columns 
$B_1 \ldots B_s$ by the size of 
the first block $B_1$.  It could be that $s=1$ so that $B_1$ is 
a $A_{2j+1,3}$-cluster. Such generalized clusters contribute 
$C^{0,0,2j+1}_{(2j+1)n,A_{2j+1,3}}(x)$ to $GC^{0,0,2j+1}_{(2j+1)n,A_{2j+1,3},\mathscr R}(x)$. 
Otherwise $B_1$ has $r$ columns where $1 \leq r \leq n-1$. In that situation, 
our remarks imply that $B_1$ consists of the elements $1, \ldots, (2j+1)r$ 
and if we subtract $(2j+1)r$ from the elements of $B_2 \ldots B_s$, we will 
end up with a generalized $A_{2j+1,3}, \mathscr R$-clusters with $n-r$ columns. Thus 
if $B_1$ has $r$ columns, then such generalized clusters contribute 
$-C^{0,0,2j+1}_{(2j+1)r,A_{2j+1,3}}(x) GC^{0,0,2j+1}_{(2j+1)(n-r),A_{2j+1,3},\mathscr R}(x)$  to 
$GC^{0,0,2j+1}_{(2j+1)n,A_{2j+1,3},\mathscr R}(x)$. It follows that for $n \geq 3$,
\begin{equation}\label{G2j+1A:rec}
GC^{0,0,2j+1}_{(2j+1)n,A_{2j+1,3},\mathscr R}(x) = C^{0,0,2j+1}_{(2j+1)n,A_{2j+1,3}}(x) - 
\sum_{r=1}^{n-1} C^{0,0,2j+1}_{(2j+1)r,A_{2j+1,3}}(x) GC^{0,0,2j+1}_{(2j+1)(n-r),A_{2j+1,3},\mathscr R}(x).
\end{equation}

For example, we have used this recursion to compute some initial values 
of $GC^{0,0,3}_{3n,A_{3,3},\mathscr R}(x)$.

\begin{center}
\begin{tabular}{|l|l|}
\hline
$n$ & $GC^{0,0,3}_{3n,A_{3,3},\mathscr R}(x)$ \\
\hline
1& 1 \\
\hline 
2&$-1$\\
\hline 
3&$1+x$\\
\hline 
4&$-1-2x+x^2$\\
\hline 
5&$1+3x-x^2+x^3$\\
\hline 
6&$-1-4x+x^4$\\
\hline 
7&$1+5x+2x^2-2x^3+x^4+x^5$\\
\hline 
8&$-1-6x-5x^2+4x^3-3x^4+2x^5+x^6$\\
\hline
\end{tabular}
\end{center}

Then one can use these initial 
values of $GC^{0,0,3}_{3n,A_{3,3},\mathscr R}(x)$ 
and Theorem \ref{thm:mainGamma} 
to compute the following initial 
values of the polynomials $\sum_{\sg \in \mathcal{E}^{0,0,3}_{3n}} 
x^{A_{3,3}\text{-}\mathrm{mch}(\sg)}$.

\begin{center}
\begin{tabular}{|l|l|}
\hline
$n$ & $\sum_{\sg \in \mathcal{E}^{0,0,3}_{3n}} 
x^{A_{3,3}\text{-}\mathrm{mch}}(\sg)$\\
\hline
1& 1 \\
\hline 
2& 19\\
\hline 
3&$ 1512+x$\\
\hline 
4&$315086+436 x+x^2$\\
\hline 
5&$ 135797476+286658 x+906 x^2+x^3$\\
\hline 
6&$ 104962186084+297928120 x+1118810 x^2+1628 x^3+x^4$\\
\hline 
7&$ 132231677979632+471533554572 x+2006137956 x^2+3724536 x^3+2656 x^4+x^5$\\
\hline 
\end{tabular}
\end{center}

When $k =2j$ is even, then it is easy to see that for each 
$A_{2j,3}$-cluster $C$ with $n \geq 2$ columns, the smallest element 
in the cluster occupies the bottom left-most position  
and the element in the top right-most position equals $(2j-1)n+1$. In 
a generalized $A_{2j,3}, \mathscr R$-cluster $G$ with 
blocks $B_1B_2 \ldots B_r$, 
we must have that top right-most element of $B_i$ is less than 
the bottom left-most element of $B_{i+1}$ for $i =1, \ldots, r-1$. There are 
three cases to consider. First the generalized 
$A_{2j+1,3}, \mathscr R$-clusters which consist of a single block $B_1$ 
contribute $C^{0,0,2j}_{(2j)n,A_{2j,3}}(x)$ to $GC^{0,0,2j}_{(2j)n,A_{2j,3},\mathscr R}(x)$.
There are two additional cases if $G$ has more than one block 
which are pictured in Figure \ref{fig:Aevencluster}. 
It is easy to see 
that if $B_1$ is size one, which is pictured at the top of  
Figure \ref{fig:Aevencluster}, then the elements in $B_1$ must be the smallest 
$2j$ elements in $G$. Then $B_2 \ldots B_r$ is order isomorphic 
to a  $A_{2j+1,3}, \mathscr R$-cluster with $n-1$ columns. 
Hence such generalized  
$A_{2j+1,3}, \mathscr R$-clusters contribute $-GC^{0,0,2j}_{(2j)(n-1),A_{2j,3},\mathscr R}(x)$ to 
 $GC^{0,0,2j}_{(2j)n,A_{2j,3},\mathscr R}(x)$.   
However if $B_1$ has $2 \leq s \leq n-1$ columns, which is the situation 
pictured at the bottom of  
Figure \ref{fig:Aevencluster}, then it is easy to see that elements 
in the first $2j-1$ rows of $B_1$ plus right-most element of the top row of $B_1$, 
must be the elements $1, \ldots, s(2j-1) +1$. Then we are free to 
choose the remaining elements in the top row of $B_1$ in 
$\binom{2jn-((2j-1)s+1)}{s-1}$ ways and $B_2 \ldots B_r$ will be  order isomorphic 
to a  $A_{2j+1,3}, \mathscr R$-cluster with $n-s$ columns. 
Hence such generalized  
$A_{2j+1,3}, \mathscr R$-clusters contribute 
$-\binom{2jn-((2j-1)s+1)}{s-1}C^{0,0,2j}_{2js}(x)
GC^{0,0,2j}_{(2j)(n-s),A_{2j,3},\mathscr R}(x)$ to 
 $GC^{0,0,2j}_{(2j)n,A_{2j,3},\mathscr R}(x)$. 

\fig{1.20}{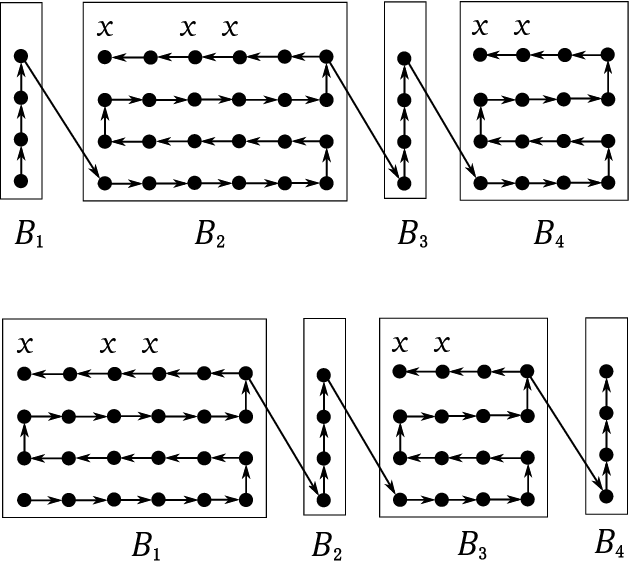}{Generalized $A_{4,3},\mathscr R$-clusters.}

It is easy to check that for $ j \geq 1$, 
 \begin{eqnarray*}
GC^{0,0,2j}_{(2j)n,A_{2j,3},\mathscr R}(x) &=& 1 \ \mbox{and} \\ 
GC^{0,0,2j}_{(2j)2,A_{2j,3},\mathscr R}(x) &=& -1.
\end{eqnarray*}
Then it follows from our arguments above that for $n \geq 3$,
\begin{eqnarray}\label{G2jA:rec}
GC^{0,0,2j}_{(2j)n,A_{2j,3},\mathscr R}(x) &=& 
C^{0,0,2j}_{(2j)n,A_{2j,3}}(x) - 
GC^{0,0,2j}_{(2j+1)(n-1),A_{2j+1,3},\mathscr R}(x) - \nonumber \\
&&\sum_{s=2}^{n-1} \binom{2jn-((2j-1)s+1)}{s-1}
C^{0,0,2j}_{(2j)s,A_{2j,3}}(x) GC^{0,0,2j}_{(2j)(n-s),A_{2j,3},\mathscr R}(x).
\end{eqnarray}

For example, we have used this recursion to compute some initial values 
of $GC^{0,0,2}_{2n,A_{2,3},\mathscr R}(x)$.

\begin{center}
\begin{tabular}{|l|l|}
\hline
$n$ & $GC^{0,0,2}_{2n,A_{2,3},\mathscr R}(x)$ \\
\hline
1& 1 \\
\hline 
2&$-1$\\
\hline 
3&$1+x$\\
\hline 
4&$-1-7x+x^2$\\
\hline 
5&$1+22x-10x^2+x^3$\\
\hline 
6&$-1-50x+2x^2-14x^3+x^4$\\
\hline 
7&$1+95x+299x^2-86x^3-19x^4+x^5$\\
\hline 
8&$-1-161x-1796x^2+1705x^3-377x^4-25x^5+x^6$\\
\hline
\end{tabular}
\end{center}

One can then use these values of $GC^{0,0,2}_{2n,A_{2,3},\mathscr R}(x)$ 
and Theorem \ref{thm:mainGamma} 
to compute the following initial 
values of the polynomials $\sum_{\sg \in \mathcal{E}^{0,0,3}_{3n}} 
x^{A_{2,3}\text{-}\mathrm{mch}(\sg)}$.

\begin{center}
\begin{tabular}{|l|l|}
\hline
$n$ & $\sum_{\sg \in \mathcal{E}^{0,0,3}_{3n}} 
x^{A_{3,3}\text{-}\mathrm{mch}}(\sg)$\\
\hline
1& 1 \\
\hline 
2& 5\\
\hline 
3&$ 60+x$\\
\hline 
4&$1337+47x+x^2$\\
\hline 
5&$ 47848+2595x+77x^2+x^3$\\
\hline 
6&$2511532+184040x+7178x^2+114x^3+x^4$\\
\hline 
7&$ 181751841+16779902x+810333x^2+18746x^3+158x^4+x^5$\\
\hline 
\end{tabular}
\end{center}

\section{Generalized Clusters for fillings 
of $D^{0,j,k}_{kn+j}$, $D^{i,0,k}_{i+kn}$, and $D^{i,j,k}_{ikn+j}$.}

In this section, we shall extend the generalized cluster 
method to deal with various types of fillings of 
$D^{0,j,k}_{kn+j}$, $D^{i,0,k}_{i+kn}$, and $D^{i,j,k}_{ikn+j}$ 
in the case where $k \geq 2$ and $i$ and $j$ are positive integers 
which are not equal to $k$.

\subsection{Generalized Clusters for fillings 
of $D^{0,j,k}_{kn+j}$}

We will start by considering fillings of $D^{0,j,k}_{kn+j}$. 
Fix a set of patterns 
$\Gamma \subseteq \mathcal{P}^{0,0,k}_{rk}$ where $r \geq 2$.
We let $\mathcal{MP}^{0,j,k}_{kn+j,\Gamma}$ denote the set of elements that 
arise by starting with an element $F$ of $\mathcal{P}^{0,j,k}_{kn+j}$ and 
marking some of the $\Gamma$-matches in $F$ by placing on $x$ on the column 
which starts the $\Gamma$-match. Given an element 
$F \in \mathcal{MP}^{0,j,k}_{kn+j,\Gamma}$, we let $m_{\Gamma}(F)$ denote 
the number of marked $\Gamma$-matches in $F$. 

To find an extension of Theorem 
\ref{thm:mainGamma} for these types of arrays, we need 
to define a special type of generalized cluster which we 
call a generalized end cluster. That is, suppose that we are given a binary relation $\mathscr R$ on columns of integers.  

\begin{definition}
We say that $Q \in \mathcal{MP}^{0,j,k}_{kn+j,\Gamma}$ is a 
{\bf generalized $\Gamma,\mathscr R$-end-cluster} if 
we can write $Q=B_1B_2\cdots B_m$ where $B_i$ are blocks 
of consecutive columns 
in $Q$ such that  
\begin{enumerate}
\item $B_m$ is a column of height $j$, 	
\item for $i < m$, 
either $B_i$ is a single column or $B_i$ consists of $r$-columns 
where $r \geq 2$, $\red(B_i)$ is a $\Gamma$-cluster in $\mathcal{MP}_{kr,\Gamma}$, 
and any pair of consecutive columns in $B_i$ are in $\mathscr R$ and
	\item for $1\leq i\leq m-1$, the pair $(\last(B_i),\first(B_{i+1})$ is 
not in $\mathscr R$.  
\end{enumerate}
\end{definition}
Let $\mathcal{GEC}^{0,j,k}_{kn+j,\Gamma,\mathscr R}$ denote the set of all generalized 
$\Gamma,\mathscr R$-end-clusters which have $n$ columns of height $k$ followed 
by a column of height $j$. 
Given $Q =B_1 B_2 \ldots B_m \in \mathcal{GEC}^{0,j,k}_{kn+j,\Gamma,\mathscr R}$, we 
define the weight of $B_i$, $w_{\Gamma,\mathscr R}(B_i)$, to be 1 if $B_i$ is a single column and 
$x^{m_{\Gamma}(\red(B_i))}$ if $B_i$ is order isomorphic to a $\Gamma$-cluster. Then we 
define the weight of $Q$, $w_{\Gamma,\mathscr R}(Q)$, to be 
$(-1)^{m-1}\prod_{i=1}^m w_{\Gamma,\mathscr R}(B_i)$. We 
let 
\begin{equation}\label{GEPRC}
GEC^{0,j,k}_{kn+j,\Gamma,\mathscr R}(x) = 
\sum_{Q \in \mathcal{GEC}^{0,j,k}_{kn+j,\Gamma,\mathscr R}} w_{\Gamma,\mathscr R}(Q).
\end{equation}
Let $\mathcal{P}^{0,j,k}_{kn+j,\mathscr R}$ denote the set of all elements 
$F \in \mathcal{P}^{0,j,k}_{kn+j}$ such that the relation $\mathscr R$ holds 
for any pair of consecutive columns in $F$. 
We let $\mathcal{MP}^{0,j,k}_{kn+j,\Gamma,\mathscr R}$ denote the set of elements that 
arise by starting with an element $F$ of $\mathcal{P}^{0,j,k}_{kn+j,\mathscr R}$ and 
marking some of the $\Gamma$-matches in $F$ by placing on $x$ on the column 
which starts the $\Gamma$-match.

Then we have the following theorem. 

\begin{theorem}\label{thm:jmainGamma}
Let $\mathscr R$ be a binary relation on pairs of columns $(C,D)$  
which are filled 
with integers which are increasing from bottom to top. 
 Let $\Gamma \subseteq \mathcal{P}^{0,0,k}_{tk}$ where $t \geq 2$. Then 
\begin{equation}\label{eq:endprcluster}
\sum_{n \geq 0} \frac{t^{kn+j}}{(kn+j)!} 
\sum_{F \in \mathcal{P}^{0,j,k}_{kn+j,\mathscr R}} x^{\Gmch(F)} =  
\frac{\sum_{n \geq 0} 
\frac{t^{kn+j}}{(kn+j)!} GEC^{0,j,k}_{kn+j,\Gamma,\mathscr R}(x-1)}{1-\sum_{n \geq 1} 
\frac{t^{kn}}{(kn)!} GC^{0,0,k}_{kn,\Gamma,\mathscr R}(x-1)}.
\end{equation}
\end{theorem}
\begin{proof}
Replace $x$ by $x+1$ in (\ref{eq:endprcluster}).  Then 
it easy to see that 
\begin{equation}\label{eq:endprcluster2} 
\sum_{n \geq 0} \frac{t^{kn+j}}{(kn+j)!} 
\sum_{F \in \mathcal{P}^{0,j,k}_{kn+j,\mathscr R}} (x+1)^{\Gmch(F)} = 
\sum_{n \geq 0} \frac{t^{kn+j}}{(kn+j)!} 
\sum_{F \in \mathcal{MP}^{0,j,k}_{kn+j,\Gamma,\mathscr R}} x^{m_\Gamma(F)}.
\end{equation}
Thus we must show that 
\begin{equation}\label{eq:endprcluster3}
 \sum_{n \geq 0} \frac{t^{kn+j}}{(kn+j)!} 
\sum_{F \in \mathcal{MP}^{0,0,k}_{kn,\Gamma,\mathscr R}} x^{m_\Gamma(F)} =
\frac{\sum_{n \geq 0} 
\frac{t^{kn+j}}{(kn+j)!} GEC^{0,j,k}_{kn+j,\Gamma,\mathscr R}(x)}{1-\sum_{n \geq 1} 
\frac{t^{kn}}{(kn)!} GC^{0,0,k}_{kn,\Gamma,\mathscr R}(x)}.
\end{equation}
Now 
\begin{equation}\label{eq:endprcluster33}
\frac{1}{1-\sum_{n \geq 1} \frac{t^{kn}}{(kn)!} GC^{0,0,k}_{kn,\Gamma,\mathscr R}(x)} = 
1+ \sum_{m \geq 1} \left(\sum_{n \geq 1} 
\frac{t^{kn}}{(kn)!} GC^{0,0,k}_{kn,\Gamma,\mathscr R}(x)\right)^m.
\end{equation} 
Taking the coefficient of $\frac{t^{ks+j}}{(ks+j)!}$ on both sides of 
(\ref{eq:endprcluster3}) where $s \geq 0$, we see that we must show that 
\begin{eqnarray}\label{eq:endprcluster4}
\sum_{F \in \mathcal{MP}^{0,j,k}_{ks+j,\mathscr R}} x^{m_\Gamma(F)} &=& 
\left(\left(\sum_{m=1}^\infty \left(\sum_{n \geq 1} 
\frac{t^{kn}}{(kn)!} GC^{0,0,k}_{kn,\Gamma,\mathscr R}(x)\right)^m\right) \times \right.
\nonumber \\
&& \ \ \ \ \ \ \ \left. \left(\sum_{n \geq 0} 
\frac{t^{kn+j}}{(kn+j)!} GEC^{0,j,k}_{kn+j,\Gamma,\mathscr R}(x)\right)\right)|_{\frac{t^{ks+j}}{(ks+j)!}} 
\nonumber\\
&=& \sum_{a+b =s, a,b \geq 0} \binom{ks+j}{ka,kb+j} \left(\sum_{m=1}^\infty \left(\sum_{n \geq 1} 
\frac{t^{kn}}{(kn)!} GC^{0,0,k}_{kn,\Gamma,\mathscr R}(x)\right)^m\right)|_{\frac{t^{ka}}{(ka)!}}  \times \nonumber \\
&& \ \ \ \ \ \ \ \left(\sum_{n \geq 0} 
\frac{t^{kn+j}}{(kn+j)!} GEC^{0,j,k}_{kn+j,\Gamma,\mathscr R}(x)\right)|_{\frac{t^{kb+j}}{(kb+j)!}} \nonumber \\
&=& \sum_{a + b =s, a,b \geq 0} \binom{ks+j}{ka,kb+j} \times \nonumber \\
&&\left( \sum_{m =1}^a \sum_{\overset{a_1+ a_2 + \cdots +a_m=a}{a_i \geq 1}} 
\binom{ka}{ka_1,\ldots,ka_m} \prod_{i=1}^m GC^{0,0,k}_{ka_i,\Gamma,r}(x)\right) GEC^{0,j,k}_{kb+j,\Gamma,\mathscr R}(x) \nonumber \\
&=& \sum_{a + b =s, a,b \geq 0} 
\sum_{m =1}^a \sum_{\overset{a_1+ a_2 + \cdots +a_m=a}{a_i \geq 1}} 
\binom{ks+j}{ka_1,\ldots,ka_m,kb+j} \times \nonumber \\
&&\ \ \ \ \ \ \ \ \ \ \ \ \ \ \ GEC^{0,j,k}_{kb+j,\Gamma,\mathscr R}(x)
\prod_{j=1}^m GC^{0,0,k}_{ka_j,\Gamma,r}(x). 
\end{eqnarray}
The right-hand side of (\ref{eq:endprcluster4}) is now easy to interpret. 
First we pick non-negative integers $a$ and $b$ such that 
$a+b=s$. Then we pick an $m$ such that $1 \leq m \leq a$. Next we 
pick $a_1, \ldots, a_m \geq 1$ such that $a_1+a_2+ \cdots +a_m = a$. 
Next the binomial coefficient $\binom{ks+j}{ka_1,\ldots,ka_m,kb+j}$ allows 
us to pick a sequence of sets $S_1, \ldots, S_m,S_{m+1}$ which partition 
$\{1, \ldots, ks+j\}$ such that $|S_i| =ka_i$ for $i=1, \ldots ,m$ and 
$|S_{m+1}| = kb+j$. The factor $GEC^{0,j,k}_{kb+j,\Gamma,\mathscr R}(x)$ allows 
us to pick a generalized $\Gamma,\mathscr R$-end-cluster $G_{m+1}$ of size $kb+j$ 
with weight $w_{\Gamma,\mathscr R}(G_{m+1})$. Note 
that in the cases where $b=0$, our definitions imply that  
$G_{m+1}$ is just a column of height $j$ filled with the numbers 
$1, \ldots, j$ 
which is increasing, reading from bottom to top. 
Finally the product $\prod_{j=1}^m GC^{0,0,k}_{ka_j,\Gamma,\mathscr R}(x)$ allows 
us to pick generalized $\Gamma,\mathscr R$-clusters $G_i \in \mathcal{GC}^{0,0,k}_{ka_i,\Gamma,\mathscr R}$ for 
$i=1,\ldots, m$ with weight $\prod_{i=1}^m w_{\Gamma,\mathscr R}(G_i)$. Note 
that in the cases where $a_i =1$, our definitions imply that  
$G_i$ is just a column of height $k$ filled with the numbers $1, \ldots, k$ 
which is increasing, reading from bottom to top.

\fig{1.20}{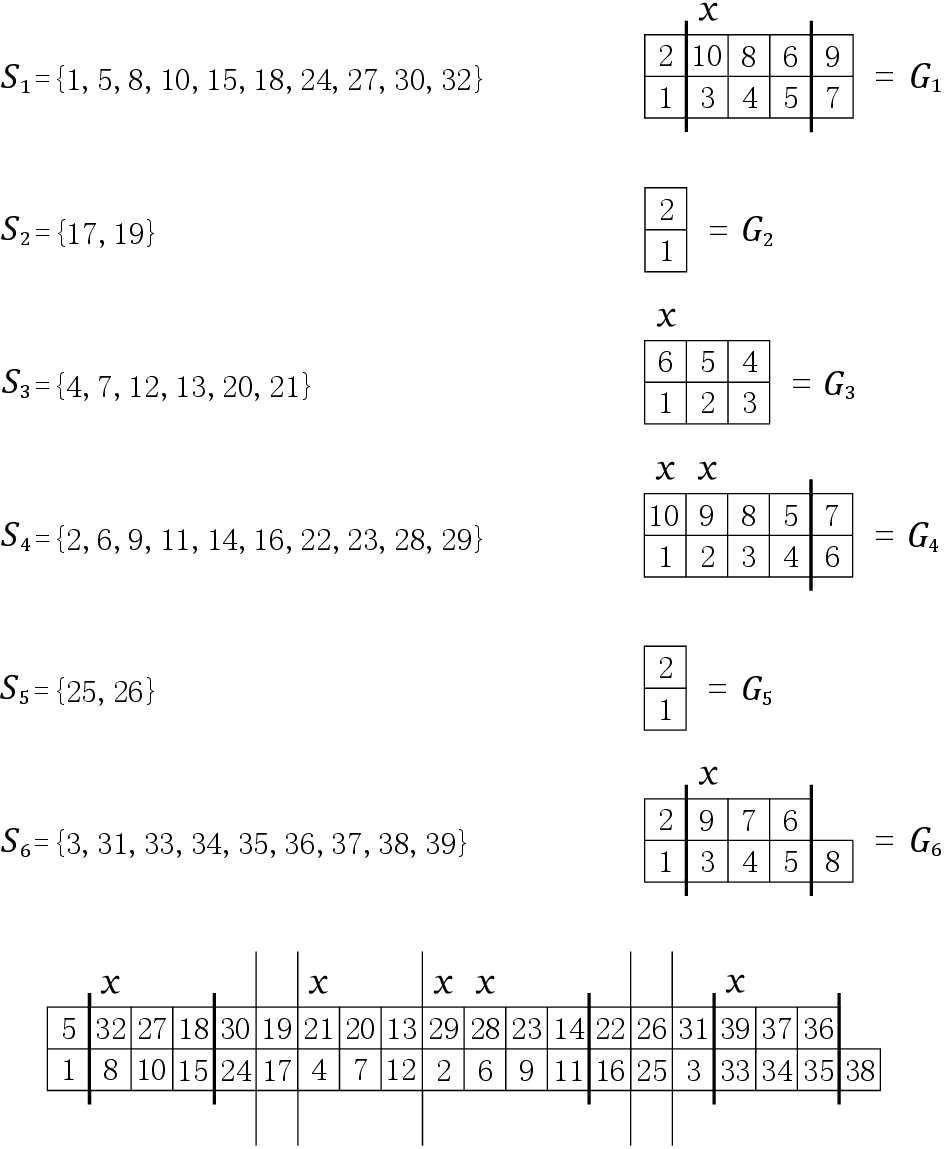}{Construction for the right-hand side of (\ref{eq:endprcluster4}).}

For example, suppose that $k=2$ and $j=1$ and $\Gamma = \{P\}$ where 
$P = \begin{array}{|c|c|c|}
\hline 
6 & 5 & 4 \\
\hline
1 & 2 & 3 \\
\hline
\end{array}$.
Suppose that $\mathscr R$ is the 
relation such that for any two increasing columns of 
integers $C$ and $D$, $(C,D) \in \mathscr R$ if and only if 
the top element of $C$ is greater than the bottom element of $D$. 
Then in Figure \ref{fig:EPRCluster2}, we have pictured 
$S_1, S_2, S_3, S_4, S_5,S_6$ which partition $\{1, \ldots, 39\}$ and 
corresponding generalized $\Gamma,\mathscr R$-clusters 
$G_1, \ldots, G_5$ and a general $\Gamma,\mathscr R$-end-cluster $G_6$. 
For each $i$, 
we have indicated the separation between the blocks of $G_i$ by dark black 
lines.  Then for each $i =1, 
\ldots, m+1$, we create a cluster $E_i$ which results by replacing 
each  $j$ in $G_i$ by the $j^{th}$ smallest element of $S_i$.  If we 
concatenate $E_1 \ldots E_6$ together, then we will obtain 
an element of $Q \in \mathcal{M}_{ks+j,\Gamma}^{0,j,k}$. The weight 
of $Q$ equals $\prod_{j=1}^6 w_{\Gamma,\mathscr R}(G_i)$ where 
\begin{eqnarray*}
w_{\Gamma,\mathscr R}(G_1) &=& (-1)^2 x, \\
w_{\Gamma,\mathscr R}(G_2) &=& 1, \\
w_{\Gamma,\mathscr R}(G_3) &=& x, \\
w_{\Gamma,\mathscr R}(G_4) &=& (-1)^1 x^2, \\ 
w_{\Gamma,\mathscr R}(G_5) &=& 1, \ \mbox{and} \\
w_{\Gamma,\mathscr R}(G_6) &=& x.   \\
\end{eqnarray*}
In Figure \ref{fig:EPRCluster2}, we have indicated the boundaries between the $_i$s by 
light lines. 

We let $\mathcal{HGEC}_{ks+j,\Gamma,\mathscr R}$ denote the set of all elements that can be constructed 
in this way. Thus $Q =E_1 \ldots E_mE_{m+1}$ is an element of 
$\mathcal{HGEC}_{ks+j,\Gamma,\mathscr R}$ if and only if for each $i =1, \ldots, m$, 
$\red(E_i)$ is a generalized $\Gamma,\mathscr R$-cluster and 
$\red(E_{m+1})$ is a generalized $\Gamma,\mathscr R$-end-cluster. Next we define a sign 
reversing involution $\theta:\mathcal{HGEC}_{ks+j,\Gamma,\mathscr R} \rightarrow \mathcal{HGEC}_{ks+j,\Gamma,\mathscr R}$. 
Given $Q =E_1 \ldots E_m E_{m+1} \in \mathcal{HGEC}_{ks+j,\Gamma,\mathscr R}$, look for the first $i$ such that either 
\begin{enumerate}
\item the block structure of $\red(E_i) = B^{(i)}_1 \ldots B^{(i)}_{k_i}$ consists 
of more than one block or 
\item $E_i$ consists of a single block $B^{(i)}_1$ and $(\last(B^{(i)}_1),\first(E_{i+1}))$ is 
not in $\mathscr R$.
\end{enumerate}
In case (1),  if $i \leq m$, we let $\theta(E_1 \ldots E_{m+1})$ 
be the result of  replacing $E_i$ by two generalized $\Gamma,\mathscr R$-clusters, $E_i^*$ and 
$E_i^{**}$, where $E_i*$ consists just of $B^{(i)}_1$ and $E_i^{**}$ consists 
of $B^{(i)}_2 \ldots B^{(i)}_{k_i}$. If $i = m+1$, we let $\theta(E_1 \ldots E_{m+1})$ 
be the result of  replacing $E_{m+1}$ by a generalized $\Gamma,\mathscr R$-cluster $E_{m+1}^*$ followed by  
a generalized $\Gamma,\mathscr R$-end-cluster 
$E_{m+1}^{**}$ where $E_{m+1}*$ consists just of $B^{(m+1)}_1$ and $E_{m+1}^{**}$ consists 
of $B^{(m+1)}_2 \ldots B^{(m+1)}_{k_{m+1}}$.  Note that in either case, 
it is easy to check that 
$w_{\Gamma,\mathscr R}(E_i) = -w_{\Gamma,\mathscr R}(E_i^*)w_{\Gamma,\mathscr R}(E_i)^{**}$. 
In case (2), if $i < m$, we let $\theta(E_1 \ldots E_{m+1})$ 
be the result of  replacing  $E_i$ and $E_{i+1}$ by the single generalized 
$\Gamma,\mathscr R$-cluster $E= B^{(i)}_1 B^{(i+1)}_1 \ldots B^{(i+1)}_{k_{i+1}}$. 
If $i =m$, we let $\theta(E_1 \ldots E_{m+1})$ 
be the result of  replacing  $E_m$ and $E_{m+1}$ by the single generalized 
$\Gamma,\mathscr R$-end-cluster $E= B^{(m)}_1 B^{(m+1)}_1 \ldots B^{(m+1)}_{k_{m+1}}$. 
In either 
case,  it is easy to check that $w_{\Gamma,\mathscr R}(E_i) w_{\Gamma,\mathscr R}(E_{i+1}) = -w_{\Gamma,\mathscr R}(E)$.
If neither case (1) or case (2) applies, then we let $\theta(E_1 \ldots E_{m+1}) =
E_1 \ldots E_{m+1}$.  For example, suppose that $\mathscr R$ is the binary 
relation where for any two increasing columns of integers $C$ and $D$, $(C,D) \in \mathscr R$ if and only if 
the top element of $C$ is greater than the bottom element of $D$ and $\Gamma =\{P\}$ where 
$P = \begin{array}{|c|c|c|}
\hline 
6 & 5 & 4 \\
\hline
1 & 2 & 3 \\
\hline
\end{array}$.
Then  if $Q=E_1 \ldots E_6$ is the 
generalized $\Gamma,\mathscr R$-cluster pictured in Figure \ref{fig:EPRCluster2}, 
then we are in case (1) with $i=1$ since $E_1$ consists of more than one block. 
Thus $\theta(Q)$ results by breaking that generalized $\Gamma,\mathscr R$-cluster into 
two clusters $E_1^*$ of size 1 and $E^{**}_1$ of size 4. $\theta(Q)$ is 
pictured in Figure \ref{fig:EPRCluster3}.

\fig{1.20}{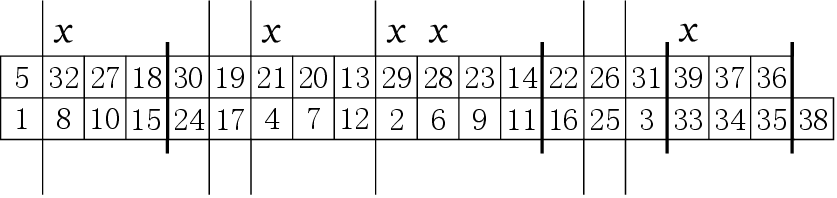}{The involution $\theta$.}

It is easy to see that $\theta$ is an involution. That is, if 
$Q=E_1 \ldots E_{m+1}$ is in case (1) using $E_i$, then $\theta(Q)$ will be 
in case (2) using $E_i^*$ and $E_i^{**}$.  Similarly 
if $Q=E_1 \ldots E_{m+1}$ is in case (1) using $E_i$ and $E_{i+1}$, then $\theta(Q)$ will be 
in case (2) using $E =E_iE_{i+1}$. It follows that if 
$\theta (E_1 \ldots E_{m+1}) \neq E_1 \ldots E_{m+1}$, then 
$w_{\Gamma,\mathscr R}(E_1 \ldots E_{m+1}) = - w_{\Gamma,\mathscr R}(\theta(E_1 \ldots E_{m+1}))$ so 
that the right-hand side of (\ref{eq:prcluster4}) equals 
$$\sum_{Q = E_1 \ldots E_{m+1} \in \mathcal{HGEC}_{ks+j,\Gamma,\mathscr R}, \theta(Q) =Q}
\prod_{i=1}^{m+1} w_{\Gamma,\mathscr R}(E_i).$$
Thus we must examine the fixed points of $\theta$.  

If $Q = E_1 \ldots E_{m+1} \in \mathcal{HGEC}_{ks+j,\Gamma,\mathscr R}$ and $\theta(Q) =Q$, 
then it must be the case that for each $i \leq m$, 
$E_i$ consists of single column of weight 1 or 
it reduces to generalized 
$\Gamma,\mathscr R$-cluster $\overline{E}_i$ consisting of a single block $B^{(i)}_1$
whose weight is the weight of $\red(B^{(i)}_1)$ as
a $\Gamma$-cluster. Moreover, it must be the case that for all $i=1, \ldots m-1$, 
$(\last(E_i),\first(E_{i+1})$ is in $\mathscr R$. Similarly, $E_{m+1}$ must consists of 
a single column of height $j$ and $(\last(E_m),E_{m+1})$ must be in $\mathscr R$. 
But this means for all $j =1, \ldots, s-1$, 
$(Q[j],Q[j+1])$ is in $\mathscr R$.  That is, either $Q[j]$ equals $\last(E_i)$ for some $i$ 
or column $j$ is contained in one of the $\Gamma$-clusters $E_i$ in which 
case $(Q[j],Q[j+1])$ is in $\mathscr R$ by our definition of generalized $\Gamma,\mathscr R$-clusters. Thus any 
fixed point $Q$ of $\theta$ is an element $\mathcal{MP}^{0,j,k}_{ks+j,\Gamma,\mathscr R}$. Then just like 
our proof Theorem \ref{thm:Gammacluster}, it follows that  $E_1, \ldots, E_m$ are just 
the maximal $\Gamma$-subclusters of an element in $\mathcal{P}^{0,j,k}_{kn,\mathscr R}$. 
Vice versa, if $T =F_1 \ldots F_rF_{r+1}$ is an element of $\mathcal{P}^{0,0,k}_{ks+j,\mathscr R}$ 
where $F_1, \ldots, F_r$ are the maximal $\Gamma$-subclusters of $T$ and $F_{r+1}$ is 
the last column of height $j$, then 
$T = F_1 \ldots F_r$ is a fixed point of $\theta$. Thus we have proved 
that the right-hand side of (\ref{eq:endprcluster4}) equals 
$$\sum_{F \in \mathcal{MP}^{0,j,k}_{ks+j,\Gamma,\mathscr R}} x^{m_\Gamma(F)}$$
which is what we wanted to prove. 
\end{proof}

\subsection{Generalized Clusters for fillings 
of $D^{i,0,k}_{i+kn}$.}

In this subsection, we shall extend the generalized cluster 
method to deal with various types of fillings of 
$D^{i,0,k}_{i+kn}$. Fix a set of patterns 
$\Gamma \subseteq \mathcal{P}^{0,0,k}_{tk}$ where $t \geq 2$.
We let $\mathcal{MP}^{i,0,k}_{i+kn,\Gamma}$ denote the set of elements that 
arise by starting with an element $F$ of $\mathcal{P}^{i,0,k}_{i+kn}$ and 
marking some of the $\Gamma$-matches in $F$ by placing on $x$ on the column 
which starts the $\Gamma$-match. Given an element 
$F \in \mathcal{MP}^{i,0,k}_{i+kn,\Gamma}$, we let $m_{\Gamma}(F)$ denote 
the number of marked $\Gamma$-matches in $F$. 

To find an extension of Theorem 
\ref{thm:mainGamma} for these types of arrays, we need 
to define a special type of generalized cluster which we 
call a generalized start cluster. That is, suppose that we are given a 
binary relation $\mathscr R$ 
on columns of integers. 

\begin{definition}
We say that $Q \in \mathcal{MP}^{i,0,k}_{i+kn,\Gamma}$ is a 
{\bf generalized $\Gamma,\mathscr R$-start-cluster} if 
we can write $Q=B_1B_2\cdots B_m$ where $B_i$ are blocks 
of consecutive columns 
in $Q$ such that  
\begin{enumerate}
\item $B_1$ is a single column of height $i$, 	
\item for $2\leq a \leq m$, 
either $B_a$ is a single column or $B_a$ consists of $r$-columns 
where $r \geq 2$, $\red(B_a)$ is a $\Gamma$-cluster in $\mathcal{MP}_{kr,\Gamma}$, 
and any pair of consecutive columns in $B_i$ are in $\mathscr R$ and
	\item for $1\leq i\leq m-1$, the pair $(\last(B_i),\first(B_{i+1})$ is 
not in $\mathscr R$.  
\end{enumerate}
\end{definition}
Let $\mathcal{GSC}^{i,0,k}_{i+kn,\Gamma,\mathscr R}$ denote the set of all generalized 
$\Gamma,\mathscr R$-start-clusters which start with a column of height $i$ and 
which is followed by $n$ columns of height $k$. 
Given $Q =B_1 B_2 \ldots B_m \in \mathcal{GEC}^{i,0,k}_{i+kn,\Gamma,\mathscr R}$, we 
define the weight of $B_i$, $w_{\Gamma,\mathscr R}(B_i)$, to be 1 if $B_i$ is a single column and 
$x^{m_{\Gamma}(\red(B_i))}$ if $B_i$ is order isomorphic to a $\Gamma$-cluster. Then we 
define the weight of $Q$, $w_{\Gamma,\mathscr R}(Q)$, to be 
$(-1)^{m-1}\prod_{i=1}^m w_{\Gamma,\mathscr R}(B_i)$. We 
let 
\begin{equation}\label{GSPRC}
GSC^{i,0,k}_{i+kn,\Gamma,\mathscr R}(x) = 
\sum_{Q \in \mathcal{GSC}^{i,0,,k}_{i+kn,\Gamma,\mathscr R}} w_{\Gamma,\mathscr R}(Q). 
\end{equation}
Let $\mathcal{P}^{i,0,k}_{i+kn,\mathscr R}$ denote the set of all elements 
$F \in \mathcal{P}^{i,0,k}_{i+kn}$ such that the relation $\mathscr R$ holds 
for any pair of consecutive columns in $F$. 
We let $\mathcal{MP}^{i,0,k}_{i+kn,\Gamma,\mathscr R}$ denote the set of elements that 
arise by starting with an element $F$ of $\mathcal{P}^{i,0,k}_{i+kn,\mathscr R}$ and 
marking some of the $\Gamma$-matches in $F$ by placing on $x$ on the column 
which starts the $\Gamma$-match.

Then we have the following theorem. 

\begin{theorem}\label{thm:imainGamma}
Let $\mathscr R$ be a binary relation on pairs of columns $(C,D)$  
which are filled 
with integers which are increasing from bottom to top. 
 Let $\Gamma \subseteq \mathcal{P}^{0,0,k}_{tk}$ where $t \geq 2$. Then 
\begin{equation}\label{eq:stprcluster}
\sum_{n \geq 0} \frac{t^{i+kn}}{(i+kn)!} 
\sum_{F \in \mathcal{P}^{i,0,k}_{i+kn,\mathscr R}} x^{\Gmch(F)} =  
\frac{\sum_{n \geq 0} 
\frac{t^{i+kn}}{(i+kn)!} GSC^{i,0,k}_{i+kn,\Gamma,\mathscr R}(x-1)}{1-\sum_{n \geq 1} 
\frac{t^{kn}}{(kn)!} GC^{0,0,k}_{kn,\Gamma,\mathscr R}(x-1)}.
\end{equation}
\end{theorem}

The proof of Theorem \ref{thm:imainGamma} is very similar to the proof of 
Theorem \ref{thm:jmainGamma} so that we shall not give the details.

\subsection{Generalized Clusters for fillings of $D^{i,j,k}_{i+kn+j}$.}

In this subsection, we shall combine the results of the 
previous two subsections 
to extend the generalized cluster 
method to deal with various types of fillings of 
$D^{i,j,k}_{i+kn+j}$. Fix a set of patterns 
$\Gamma \subseteq \mathcal{P}^{0,0,k}_{tk}$ where $t \geq 2$.
We let $\mathcal{MP}^{i,j,k}_{i+kn+j,\Gamma}$ denote the set of elements that 
arise by starting with an element $F$ of $\mathcal{P}^{i,j,k}_{i+kn+j}$ and 
marking some of the $\Gamma$-matches in $F$ by placing on $x$ on the column 
which starts the $\Gamma$-match. Given an element 
$F \in \mathcal{MP}^{i,j,k}_{i+kn+j,\Gamma}$, we let $m_{\Gamma}(F)$ denote 
the number of marked $\Gamma$-matches in $F$. Clearly, neither the first or the last column of $F$ is contained in a $\Gamma$-match.

To find an extension of Theorem 
\ref{thm:mainGamma} for fillings of $D_{kn+j}^{0,j,k}$, we defined generalized end clusters and  for fillings of $D_{kn+i}^{i,0,k}$, we defined generalized start clusters. However, they are not enough to extend Theorem 
\ref{thm:mainGamma} for fillings of $D_{i+kn+j}^{i,j,k}$. Next
we define a special type of generalized cluster which we 
call a generalized start-end cluster. That is, suppose that we are given a binary relation $\mathscr R$ on columns of integers.

\begin{definition}
	We say that $Q \in \mathcal{MP}^{i,j,k}_{i+kn+j,\Gamma}$ is a 
	{\bf generalized $\Gamma,\mathscr R$-start-end-cluster} if 
	we can write $Q=B_1B_2\cdots B_m$ where $B_i$ are blocks 
	of consecutive columns 
	in $Q$ such that  
	\begin{enumerate}
		\item $m\geq 2$,
		\item $B_1$ is a column of height $i$, 
		\item $B_m$ is a column of height $j$, 	
		\item for $1<i < m$, 
		either $B_i$ is a single column or $B_i$ consists of $r$-columns 
		where $r \geq 2$, $\red(B_i)$ is a $\Gamma$-cluster in $\mathcal{MP}_{kr,\Gamma}$, 
		and any pair of consecutive columns in $B_i$ are in $\mathscr R$ and
		\item for $1\leq i\leq m-1$, the pair $(\last(B_i),\first(B_{i+1}))$ is 
		not in $\mathscr R$.  
	\end{enumerate}
\end{definition}

Let $\mathcal{GSEC}^{i,j,k}_{i+kn+j,\Gamma,\mathscr R}$ denote the set of all generalized $\Gamma,\mathscr R$-start-end-clusters which have $n$ columns of height $k$ between a column of height $i$ and a column of height $j$. 
Given $Q =B_1 B_2 \ldots B_m \in \mathcal{GSEC}^{i,j,k}_{i+kn+j,\Gamma,\mathscr R}$, we 
define the weight of $B_i$, $\omega_{\Gamma,\mathscr R}(B_i)$, to be 1 if $B_i$ is a single column and 
$x^{m_{\Gamma}(\red(B_i))}$ if $B_i$ is order isomorphic to a $\Gamma$-cluster. Then we 
define the weight of $Q$, $\omega_{\Gamma,\mathscr R}(Q)$, to be 
$(-1)^{m-1}\prod_{i=1}^m \omega_{\Gamma,\mathscr R}(B_i)$. We 
let 
\begin{equation}
GSEC^{0,j,k}_{kn+j,\Gamma,\mathscr R}(x) = 
\sum_{Q \in \mathcal{GSEC}^{i,j,k}_{i+kn+j,\Gamma,\mathscr R}} \omega_{\Gamma,\mathscr R}(Q).
\end{equation}
Let $\mathcal{P}^{i,j,k}_{i+kn+j,\mathscr R}$ denote the set of all elements 
$F \in \mathcal{P}^{i,j,k}_{i+kn+j}$ such that the relation $\mathscr R$ holds 
for any pair of consecutive columns in $F$. 
We let $\mathcal{MP}^{i,j,k}_{i+kn+j,\Gamma,\mathscr R}$ denote the set of elements that 
arise by starting with an element $F$ of $\mathcal{P}^{i,j,k}_{i+kn+j,\mathscr R}$ and 
marking some of the $\Gamma$-matches in $F$ by placing an $x$ on the column 
which starts the $\Gamma$-match.

Then we have the following theorem. 

\begin{theorem}\label{thm:ijmainGamma}
	Let $\mathscr R$ be a binary relation on pairs of columns $(C,D)$  
	which are filled 
	with integers which are increasing from bottom to top. 
	Let $\Gamma \subseteq \mathcal{P}^{0,0,k}_{tk}$ where $t \geq 2$. Then 
	\begin{multline}\label{eqn:seprcluster}
	\sum_{n \geq 0} \frac{t^{i+kn+j}}{(i+kn+j)!} 
	\sum_{F \in \mathcal{P}^{i,j,k}_{i+kn+j,\mathscr R}} x^{\Gmch(F)} =  \\
	\frac{\left(\sum_{n \geq 0} 
		\frac{t^{i+kn}}{(i+kn)!} GSC^{i,0,k}_{i+kn,\Gamma,\mathscr R}(x-1)\right)\left(\sum_{n \geq 0} 
		\frac{t^{kn+j}}{(kn+j)!} GEC^{0,j,k}_{kn+j,\Gamma,\mathscr R}(x-1)\right)}{1-\sum_{n \geq 1} 
		\frac{t^{kn}}{(kn)!} GC^{0,0,k}_{kn,\Gamma,\mathscr R}(x-1)}\\
	+\sum_{n \geq 0} 
	\frac{t^{i+kn+j}}{(i+kn+j)!} GSEC^{i,j,k}_{i+kn+j,\Gamma,\mathscr R}(x-1).
	\end{multline}
\end{theorem}
\begin{proof}
 Replace $x$ by $x+1$ in (\ref{eqn:seprcluster}).  Then  
it easy to see that 
\begin{equation}\label{eqn:seprcluster2} 
\sum_{n \geq 0} \frac{t^{i+kn+j}}{(i+kn+j)!} 
\sum_{F \in \mathcal{P}^{i,j,k}_{i+kn+j,\Gamma,\mathscr R}} (x+1)^{\Gmch(F)} = 
\sum_{n \geq 0} \frac{t^{i+kn+j}}{(i+kn+j)!} 
\sum_{F \in \mathcal{MP}^{i,j,k}_{i+kn+j,\Gamma,\mathscr R}} x^{m_\Gamma(F)}.
\end{equation}

Thus we must show that 
\begin{multline}\label{eqn:seprcluster3}
\sum_{n \geq 0} \frac{t^{kn+j}}{(kn+j)!} 
\sum_{F \in \mathcal{MP}^{0,0,k}_{kn,\Gamma, \mathscr R}} x^{m_\Gamma(F)} =  \\
\frac{\left(\sum_{n \geq 0} 
	\frac{t^{i+kn}}{(i+kn)!} GSC^{i,0,k}_{i+kn,\Gamma,\mathscr R}(x)\right)\left(\sum_{n \geq 0} 
	\frac{t^{kn+j}}{(kn+j)!} GEC^{0,j,k}_{kn+j,\Gamma,\mathscr R}(x)\right)}{1-\sum_{n \geq 1} 
	\frac{t^{kn}}{(kn)!} GC^{0,0,k}_{kn,\Gamma,\mathscr R}(x)}\\
+\sum_{n \geq 0} 
\frac{t^{i+kn+j}}{(i+kn+j)!} GSEC^{i,j,k}_{i+kn+j,\Gamma,\mathscr R}(x).
\end{multline}

Now we rewrite the right-hand side of (\ref{eqn:seprcluster3}) as
\begin{multline}\label{eqn:seprcluster33}
\left(\left(\sum_{n \geq 0} 
	\frac{t^{kn+j}}{(kn+j)!} GSC^{i,0,k}_{i+kn,\Gamma,\mathscr R}(x)\right)\left(\sum_{n \geq 0} 
	\frac{t^{kn+j}}{(kn+j)!} GEC^{0,j,k}_{kn+j,\Gamma,\mathscr R}(x)\right)\right.\\
\left.\times \sum_{m\geq 0}\left(\sum_{n \geq 1} 
	\frac{t^{kn}}{(kn)!} GC^{0,0,k}_{kn,\Gamma,\mathscr R}(x)\right)^m\right)\\
+\sum_{n \geq 0} 
\frac{t^{i+kn+j}}{(i+kn+j)!} GSEC^{i,j,k}_{i+kn+j,\Gamma,\mathscr R}(x).
\end{multline}
Taking the coefficient of $\frac{t^{ks+j}}{(ks+j)!}$ on both sides of 
(\ref{eqn:seprcluster33}) where $s \geq 0$, we see that we must show that
\begin{eqnarray}\label{eqn:seprcluster4}
&&\sum_{F \in \mathcal{MP}^{i,j,k}_{i+ks+j,\Gamma, \mathscr R}} x^{m_\Gamma(F)} \nonumber\\
&=& 
\left(\left(\left(\sum_{n \geq 0} 
\frac{t^{i+kn}}{(i+kn)!} GSC^{i,0,k}_{i+kn,\Gamma,\Gamma, \mathscr R}(x)\right)\left(\sum_{n \geq 0} 
\frac{t^{kn+j}}{(kn+j)!} GEC^{0,j,k}_{kn+j,\Gamma,\Gamma, \mathscr R}(x)\right)\right.\right.\nonumber\\
&&\left.\times \sum_{m\geq 0}\left(\sum_{n \geq 1} 
\frac{t^{kn}}{(kn)!} GC^{0,0,k}_{kn,\Gamma,\mathscr R}(x)\right)^m\right)\nonumber\\
&&\left.\left.+\sum_{n \geq 0} 
\frac{t^{i+kn+j}}{(i+kn+j)!} GSEC^{i,j,k}_{i+kn+j,\Gamma,\mathscr R}(x)
\right)\right|_{\frac{t^{i+ks+j}}{(i+ks+j)!}} 
\nonumber\\
&=& \left(\sum_{a + b +c=s, a,b,c \geq 0} 
\sum_{m =1}^b \sum_{\overset{b_1+ b_1 + \cdots +b_m=b}{b_i \geq 1}} 
\binom{i+ks+j}{i+ka,kb_1,\ldots,kb_m,kc+j} \times\right. \nonumber \\
&& \ \ \ \ \ \ \ \left. GSC^{i,0,k}_{i+ka,\Gamma, \mathscr R}(x)~ GEC^{0,j,k}_{kc+j,\Gamma, \mathscr R}(x)~
\prod_{j=1}^m GC^{0,0,k}_{kb_j,\Gamma, \mathscr R}(x)\right)+GSEC^{i,j,k}_{i+ks+j,\Gamma, \mathscr R}(x). \nonumber \\
&& ~
\end{eqnarray}
Now we interpret the right-hand side of (\ref{eqn:seprcluster4}). 
Our situation is a bit different from our previous proofs 
in that (\ref{eqn:seprcluster4}) allows two ways to construct an array in $\mathcal{MP}_{i+ks+j,\Gamma, \mathscr R}^{i,j,k}$.

One way is via $\mathcal{GSC}$, $\mathcal{GEC}$ and $\mathcal{GC}$. First we pick non-negative integers $a$, $b$ and $c$ such that 
$a+b+c=s$. Then we pick an $m$ such that $1 \leq m \leq b$. Next we 
pick $b_1, \ldots, b_m \geq 1$ such that $b_1+b_2+ \cdots +b_m = b$. 
Next the multinomial coefficient $\binom{i+ks+j}{i+ka,kb_1,\ldots,kb_m,kc+j}$ allows 
us to pick a sequence of 
sets $S_1, \ldots, S_m,S_{m+1},S_{m+2}$ which partition 
$\{1,2, \ldots, i+ks+j\}$ such that $|S_1|=i+ka$, $|S_{h+1}| =kb_h$ for $h=1, \ldots ,m$ and 
$|S_{m+2}| = kc+j$. 
The factor $GSC^{i,0,k}_{i+ka,\Gamma, \mathscr R}(x)$ allows 
us to pick a generalized $\Gamma, \mathscr R$-start-cluster $G_{1}$ of size $i+ka$ 
with weight $\omega_{\Gamma, \mathscr R}(G_{1})$. Note 
that in the cases where $a=0$, our definitions imply that  
$G_{1}$ is just a column of height $i$ filled with the numbers 
$1, \ldots, i$ 
which is increasing, reading from bottom to top. The factor $GEC^{0,j,k}_{kc+j,\Gamma, \mathscr R}(x)$ allows 
us to pick a generalized $\Gamma, \mathscr R$-end-cluster $G_{m+2}$ of size $kc+j$ 
with weight $\omega_{\Gamma,\mathscr R}(G_{m+2})$. Note 
that in the cases where $c=0$, our definitions imply that  
$G_{m+2}$ is just a column of height $j$ filled with the numbers 
$1, \ldots, j$ 
which is increasing, reading from bottom to top. Finally, the product $\prod_{j=1}^m GC^{0,0,k}_{kb_h,\Gamma, \mathscr R}(x)$ allows 
us to pick generalized $\Gamma, \mathscr R$-clusters $G_h \in \mathcal{GC}^{0,0,k}_{kb_h,\Gamma, \mathscr R}$ for 
$h=2,\ldots, m+1$ with weight $\prod_{h=1}^m \omega_{\Gamma,\mathscr R}(G_h)$. Note 
that in the cases where $b_h =1$, our definitions imply that  
$G_h$ is just a column of height $k$ filled with the numbers $1, \ldots, k$ 
which is increasing, reading from bottom to top. 

\fig{1.2}{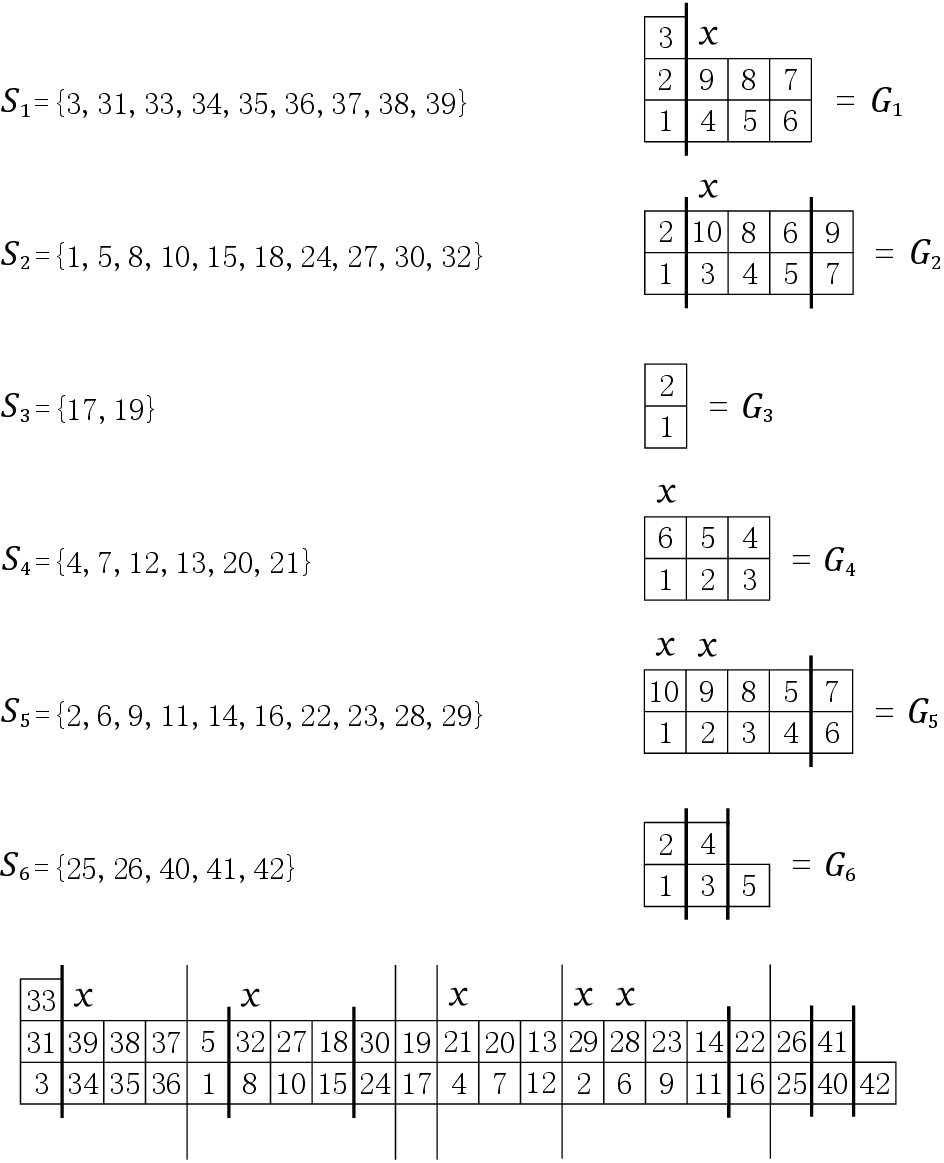}{Construction for the right-hand side of (\ref{eqn:seprcluster4}) via $\mathcal{GSC}$, $\mathcal{GEC}$ and $\mathcal{GC}$.}

Another way is via $\mathcal{GSEC}$. The term $GSEC^{i,j,k}_{i+ks+j,\Gamma, \mathscr R}(x)$ allows us to have a generalized $\Gamma, \mathscr R$-start-end-cluster $G_*$ of size $i+ks+j$. Note that in the case where $s=0$, our definitions implies that $G_*$ has two blocks and the first block is a single column of height $i$ and the second block is a single column of height $j$. 

\fig{1.2}{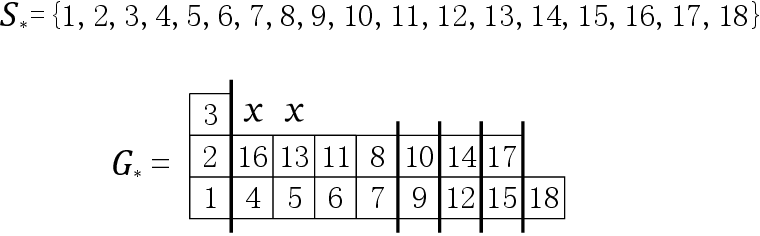}{Construction for the right-hand side of 
(\ref{eqn:seprcluster4}) via $\mathcal{GSEC}$.}

For example, suppose that $i=3$, $k=2$ and $j=1$ and $\Gamma = \{P\}$ where 
$P = \begin{array}{|c|c|c|}
\hline 
6 & 5 & 4 \\
\hline
1 & 2 & 3 \\
\hline
\end{array}$.
Suppose that $\mathscr R$ is relation where for any two increasing 
columns $C$ and $D$ of integers, $(C,D) \in \mathscr R$ if and only if the top element of $C$ is greater than the bottom element of $D$. 
Then in Figure \ref{fig:SEPRCluster2}, we have pictured 
$S_1, S_2, S_3, S_4, S_5,S_6$ which partition $\{1, \ldots, 42\}$ and 
corresponding a generalized $\Gamma, \mathscr R$- start-cluster $G_1$, generalized $\Gamma, \mathscr R$-clusters 
$G_2, \ldots, G_5$ and a generalized $\Gamma, \mathscr R$-end-cluster $G_6$. This is a construction for the right-hand side of (\ref{eqn:seprcluster4}) via $\mathcal{GSC}$, $\mathcal{GEC}$ and $\mathcal{GC}$. 
For each $h$, 
we have indicated the separation between the blocks of $G_i$ by dark black 
lines.  Then for each $h =1, 
\ldots, m+1$, we let $E_i$ be the result of replacing 
each $j$ in $G_i$ by the $j^{th}$ smallest element of $S_i$.  If we 
concatenate $E_1 \ldots E_6$ together, then we will obtain 
an element of $Q \in \mathcal{MP}_{i+ks+j,\Gamma}^{i,j,k}$. The weight 
of $Q$ equals $\prod_{h=1}^6 \omega_{\Gamma,\mathscr R}(G_h)$ where 
\begin{eqnarray*}
	\omega_{\Gamma,\mathscr R}(G_1) &=& (-1)^{2-1} x, \\
	\omega_{\Gamma,\mathscr R}(G_2) &=& (-1)^{3-1}x, \\
	\omega_{\Gamma,\mathscr R}(G_3) &=& (-1)^{1-1}, \\
	\omega_{\Gamma,\mathscr R}(G_4) &=& (-1)^{1-1} x, \\ 
	\omega_{\Gamma,\mathscr R}(G_5) &=& (-1)^{2-1}x^2, \ \mbox{and} \\
	\omega_{\Gamma,\mathscr R}(G_6) &=& (-1)^{3-1}.   \\
\end{eqnarray*}
In Figure \ref{fig:SEPRCluster2}, we have indicated the boundaries between the $E_h$'s by 
light lines. 

In Figure \ref{fig:SEPRCluster22}, we have pictured 
$S_*=\{1,2, \ldots, 18\}$ and 
corresponding a generalized $\Gamma, \mathscr R$-start-end-cluster $G_*$. This is a construction for the right-hand side of (\ref{eqn:seprcluster4}) via $\mathcal{GSEC}$.  
Clearly, $G_* \in \mathcal{MP}_{i+ks+j,\Gamma}^{i,j,k}$. The weight 
of $G_*$ is 
\begin{eqnarray*}
\omega_{\Gamma,\mathscr R}(G_*)=(-1)^{6-1}x^2.
\end{eqnarray*}

We let $\mathcal{HGC}_{i+ks+j,\Gamma,\mathscr R}$ denote the set of all elements that can be constructed via $\mathcal{GSC}$, $\mathcal{GEC}$ and $\mathcal{GC}$, and let $\mathcal{HGSEC}_{i+ks+j,\Gamma,\mathscr R}$ denote the set of all elements that can be constructed via $\mathcal{GSEC}$. We let $$
\mathcal H_{i+ks+j,\Gamma,\mathscr R}:=\mathcal{HGC}_{i+ks+j,\Gamma,\mathscr R}~\cup~ \mathcal{HGSEC}_{i+ks+j,\Gamma,\mathscr R}.
$$

Then $Q =E_1 \ldots E_{m+1}E_{m+2}$ is an element of 
$\mathcal{HGC}_{i+ks+j,\Gamma,\mathscr R}$ if and only if $m\geq 0$, $\red(E_{1})$ is a generalized $\Gamma,\mathscr R$-start-cluster, for each $h =1, \ldots, m$, 
$\red(E_{h+1})$ is a generalized $\Gamma,\mathscr R$-cluster and $\red(E_{m+2})$ is a generalized $\Gamma,\mathscr R$-end-cluster. On the other hand, 
$Q =E_1$ is an element of 
$\mathcal{HGSEC}_{i+ks+j}, \Gamma, \mathscr R$ if and only if $\red(E_{1})$ is a generalized 
$\Gamma,\mathscr R$-start-end-cluster. 

Next we define a sign reversing involution $\theta:\mathcal{H}_{i+ks+j} \rightarrow \mathcal{HGEC}_{i+ks+j}$. We would like to use the same 
involutions as in our previous proofs. However, there is 
one potential problem.  That is, suppose that 
$Q = E_1 E_2$ where $E_1$ consists of a single column of height 
$i$, $E_2 =B_1 \ldots B_{m+1}$ is a generalized 
 $\Gamma, \mathscr R$-end-cluster, and $(E_1,\first(B_1)$ is 
not in $\mathscr R$.  Then if we followed our previous involutions, 
we would like to replace $E_1E_2$ by a single cluster $E$.  If we 
did not have generalized $\Gamma,\mathscr R$-start-end-clusters, 
then we could not define $\theta(E_1E_2) =E$.  This is the 
reason, we added a term corresponding to generalized 
$\Gamma,\mathscr R$-start-end-clusters. Formally, 
$\theta$ is defined as follows. Let $Q\in\mathcal{HGC}_{i+ks+j,\Gamma, \mathscr R}$.  \\
\ \\
{\bf Case 1.} $Q =B_1 B_2 \ldots B_m B_{m+1}$ is a generalized 
 $\Gamma, \mathscr R$-start-end-cluster, then we let 
$\theta(Q) = E_1 E_2$ where $E_1 =B_1$ is a column of height $i$ and 
$E_2 =B_2 \ldots B_mB_{m+1}$ is a generalized 
$\Gamma, \mathscr R$-end-cluster. Note by our definitions, 
$(E_1,\first(B_2))$ is not in $\mathscr R$. \\
\ \\
{\bf Case 2.} $Q =E_1E_2$ where $E_1 =B_1$ is a column of height $i$,  
$E_2 =B_2 \ldots B_mB_{m+1}$ is a generalized 
$\Gamma, \mathscr R$-end-cluster, and 
$(E_1,\first(B_2))$ is not in $\mathscr R$, 
then we let $\theta(E_1E_2) =E$ where 
$E= B_1 B_2 \ldots B_m B_{m+1}$ is a generalized 
 $\Gamma, \mathscr R$-start-end-cluster.\\
\ \\
Otherwise assume we are not in Case 1 or Case 2 
and that $Q =E_1 \ldots E_{m+1} E_{m+2}$ where 
$E_j =B^{(j)}_1 \ldots B^{(j)}_{k_j}$ for $j =1, \ldots, m+2$. Then 
we look for the first $h$ such that either \\
\ \\ 
{\bf Case 3} the block structure of  $\red(E_h) = B^{(h)}_1 B^{(h)}_2\ldots B^{(h)}_{k_h}$ consists of more than one block or \\
\ \\
{\bf Case 4} $E_h$ consists of a single block $B^{(h)}_1$ and $(\last(B^{(h)}_{1}),\first(B^{(h+1)}_1))$ is not in $\mathscr R$.

If $Q$ is in Case 3, then we let $\theta(Q)$ be the result  
of replacing $E_h$ by $E_h^* =B_1^{(h)}$ and 
$E^{**} = B_2^{(h)} \ldots B^{(h)}_{k_h}$ in $Q$.  
If $Q$ is in Case 4, we let 
$\theta(Q)$ be the result of replacing $E^{(h)}E^{(h+1)}$ by 
$E =B_1^{(h)}B_1^{(h+1)} \ldots B^{(h+1)}_{k_{h+1}}$ in $Q$. 

If neither Case (1), Case (2),  Case (3), or Case (4) applies, then we let $\theta(Q) = Q$. It is straightforward to check that $\theta$ is an involution 
and that if $\theta(Q) \neq Q$, then $w_{\Gamma,\mathscr R}(Q) = - 
w_{\Gamma,\mathscr R}(\theta(Q))$. The fixed points of $\theta$ 
are those $Q \in \mathcal{HGC}_{i+ks+j,\Gamma,\mathscr R}$ of 
the form $E_1 \ldots E_{m+1}$ where each $E_i$ consists of a single 
block and 
$E_1$ consists a single column of height $i$, $E_{m+1}$ consists of 
a single column of height $j$, and for each $2 \leq r \leq m$, 
$E_r$ is either a single column of height $k$ or $E_r$ reduces 
to a $\Gamma$ cluster such any two consecutive columns in 
$E_r$ satisfy $\mathscr R$. In addition, it must be 
the case that $(\last(E_i),\first(E_{i+1}) \in \mathscr R$ for 
$i =1, \ldots, m$.  As in our previous proofs, 
it follows that any two consecutive columns in $Q$ must be in 
$\mathscr R$ so that $Q \in 
\mathcal{MP}_{i+ks+j,\Gamma,\mathscr R}^{i,j,k}$.  In addition, 
$E_1, \ldots, E_{m+1}$ are just the maximal $\Gamma$-subclusters in $R$. 
Hence the theorem follows as in the previous proofs.  
\end{proof}

\subsection{Examples}

In this subsection, we will illustrate our generalized cluster method with 
two examples. First we shall compute the generating 
function for the number of matches of 162534 in down-up permutations 
and then compute the number of 124356 matches in 3-generalized Euler 
permutations. 

\subsubsection{Up-down pattern 162534 in down-up permutations}\label{sec:ud_in_du}

We say a permutation $\sigma=\sigma_1\sigma_2\cdots\sigma_n\in\mathcal S_n$ is an up-down permutation if $\sigma_1<\sigma_2>\sigma_3<\sigma_4>\sigma_5<\cdots$. More precisely, $\sigma\in\mathcal S_n$ is up-down if and only if 
$$
Des(\sigma)=\left\{2k:~ \forall \text{ integer }k, 0<k<\frac{n}{2} \right\}.
$$
We let $\mathcal{UD}_n$ denote the set of all the up-down permutations in $\mathcal S_n$. We say a permutation $\sigma=\sigma_1\sigma_2\cdots\sigma_n\in\mathcal S_n$ is an down-up permutation if $\sigma_1>\sigma_2<\sigma_3>\sigma_4<\sigma_5>\cdots$. Similarly, a permutation $\sigma\in\mathcal S_n$ is down-up if and only if 
$$
Des(\sigma)=\left\{2k+1:~ \forall \text{ integer }k, 0<k<\frac{n+1}{2} \right\}.
$$
We let $\mathcal{DU}_n$ denote the set of all the up-down permutations in $\mathcal S_n$.

We note that 
down-up permutations can be represented by arrays. Consider a binary relation $\mathscr R$ such that holds for a pair of columns $(C,D)$ of 
height $\leq 2$ if and only if the top element of column $C$ is greater than the bottom element of column $D$. Then 
$\{w(F): F \in \mathcal{P}_{2n+2,\mathscr R}^{1,1,2}\} = \mathcal{DU}_{2n+2}$ and $\{w(F): F \in \mathcal{P}_{2n+2,\mathscr R}^{1,0,2}\} = \mathcal{DU}_{2n+1}$.


Let $\tau=1~6~2~5~3~4\in\mathcal{UD}_{6}$. Then the generating function 
$$
\sum_{n\geq 1}\frac{t^n}{n!}\sum_{\sigma\in\mathcal{DU}_{n}}x^{\tau\text{-mch}(\sigma)}= A_P(x,t) +B_P(x,t)
$$
where 
\begin{eqnarray}
&&A_{P}(x,t):=\sum_{n\geq 0}\frac{t^{2n+1}}{(2n+1)!}\sum_{F\in\mathcal{P}^{1,0,2}_{2n+1,\mathscr R}}x^{P\text{-mch}(F)}, \label{eqn:du_gf_odd}\\
&&B_{P}(x,t):=\sum_{n\geq 0}\frac{t^{2n+2}}{(2n+2)!}\sum_{F\in\mathcal{P}^{1,1,2}_{2n+2,\mathscr R}}x^{P\text{-mch}(F)},\label{eqn:du_gf_even}
\end{eqnarray}
and $P= \begin{array}{|c|c|c|}
\hline 
6 & 5 & 4 \\
\hline
1 & 2 & 3 \\
\hline
\end{array}\in\mathcal{P}_{0,0,2j}^{0,0,2}$. 
Note that (\ref{eqn:du_gf_odd}) and (\ref{eqn:du_gf_even}) are equivalent to generating functions for the distribution of $\tau$-matches  in down-up permutations of even length and odd length, respectively.

By Theorem \ref{thm:imainGamma},
\begin{equation}
A_{P}(x,t)=\frac{\sum_{n \geq 0} 
	\frac{t^{2n+1}}{(2n+1)!} GSC^{1,0,2}_{2n+1,P,\mathscr R}(x-1)}{1-\sum_{n \geq 1} 
	\frac{t^{2n}}{(2n)!} GC^{0,0,2}_{2n,P,\mathscr R}(x-1)},
\end{equation}
and, by Theorem \ref{thm:ijmainGamma},
\begin{multline}
B_{P}(x,t)=	\\
\frac{\left(\sum_{n \geq 0} 
		\frac{t^{2n+1}}{(2n+1)!} GSC^{1,0,2}_{2n+1,P,\mathscr R}(x-1)\right)\left(\sum_{n \geq 0} 
		\frac{t^{2n+1}}{(2n+1)!} GEC^{0,1,2}_{2n+1,P,\mathscr R}(x-1)\right)}{1-\sum_{n \geq 1} 
		\frac{t^{2n}}{(2n)!} GC^{0,0,2}_{2n,P,\mathscr R}(x-1)}\\
	+\sum_{n \geq 0} 
	\frac{t^{2n+2}}{(2n+2)!} GSEC^{1,1,2}_{2+2,P,\mathscr R}(x-1).
\end{multline}

Thus we need to compute $GC^{0,0,2}_{2n,P,\mathscr R}(x)$, 
$GSC^{0,0,2}_{2n,P,\mathscr R}(x)$, 
$GEC^{0,0,2}_{2n,P,\mathscr R}(x)$, and $GSEC^{0,0,2}_{2n,P,\mathscr R}(x)$.

We shall start by discussing structures of $P$-clusters. Note that 
$P$ is one of the patterns that we discussed in Section 2. In this 
case, the Hasse diagram of a $P$-clusters is of the form 
pictured in Figure \ref{fig:pcluster}.

\fig{1.40}{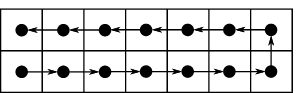}{The Hasse diagram of a $P$-cluster.}

If we let $\mathcal{C}_{2n,P}$ denote the set of $P$-clusters consisting of $n$ columns, then we proved that 
\begin{eqnarray}
&&C_2(x) =1\\	
&&C_4(x) =0\\
&&C_6(x)=x\label{eqn:du_C6}\\
	&&C_8(x)=x^2\label{eqn:du_C8}\\
	&&C_{2n}(x)=x\left(C_{2n-4}(x)+C_{2n-2}(x)\right),~\text{ for }n\geq 5.\label{eqn:du_C2n}
\end{eqnarray}

In section 2, we also gave a recursion to compute  
$GC_{2n,P,\mathscr R}^{0,0,2}(x)$.  However, in this section, we 
shall give an alternative way to compute $GC_{2n,P,\mathscr R}^{0,0,2}(x)$.
Suppose $Q\in\mathcal{GC}_{2n,P,\mathscr R}^{0,0,2}$ has $m$ blocks, i.e., $Q=B_1B_2\ldots B_m$. For an array $F$, we let $\col(F)$ denote the number of columns in $\col(F)$. Each $B_i$ is either order isomorphic to a $P$-cluster or a single column.  We let $\mathcal{GC}_{\col = (b_1,b_2,\ldots,b_m)}$ to denote the set of generalized clusters such that $\col(B_i)=b_i$. Given $\col(B_i)$ for each $1\leq i\leq m$, the set of words of generalized clusters 
$\mathcal{GC}_{\col = (b_1,b_2,\ldots,b_m)}$ is equal to the set 
of linear extensions of the poset whose Hasse diagram 
corresponds to $(b_1, b_2, \ldots, b_m)$. That is, we are 
assuming that for any block $(\last(B_i),\first(B_{i+1})$ 
is not in $\mathscr R$ so that 
there must be an arrow directed from the top of the last column of $B_i$ to the 
bottom of the first column of $B_{i+1}$. We let $\Gamma(b_1,b_2,\ldots,b_m)$ denote the Hasse diagram corresponding to $\mathcal{GC}_{\col = (b_1,b_2,\ldots,b_m)}$. For example, $\Gamma (3,1,1,5,1)$ is pictured in Figure \ref{fig:downup_gc_ex}.

\fig{1.40}{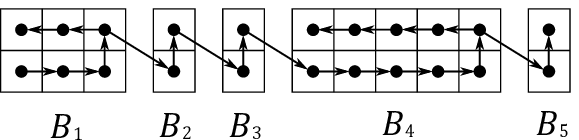}{$\Gamma (3,1,1,5,1)$.}

It is clear that
\begin{equation}\label{eqn:du_gc_compute}
	\sum_{Q\in\mathcal{GC}_{\col = (b_1,b_2,\ldots,b_m)}}\omega_{P,\mathscr R}(Q)=(-1)^{m-1}\text{LE}(\Gamma(b_1,b_2,\ldots,b_m))\prod_{i=1}^m C_{2b_i}(x),
\end{equation}
where $\text{LE}(\Gamma(b_1,b_2,\ldots,b_m))$ is the number of linear extensions of $\Gamma(b_1,b_2,\ldots,b_m)$ and by convention, we let $C_{2}(x)=1$. Therefore, to compute (\ref{eqn:du_gc_compute}), we only need to count linear extensions of $\Gamma(b_1,b_2,\ldots,b_m)$. Continue using $\Gamma (3,1,1,5,1)$ as example, the Hasse diagram in  Figure \ref{fig:downup_gc_ex} is actually a tree-like diagram, as drawn in Figure \ref{fig:downup_gc_ex_str}. Then 
we can easily compute that
$$
\text{LE}(\Gamma (3,1,1,5,1))=\binom{6}{4}\binom{18}{2}.
$$
That is, it is easy to see that the first four elements in the bottom 
row must be labeled with $1, \ldots, 4$, reading from left to right, since 
there is a directed path from 
these elements to any other elements in the poset. 
Once we remove these four  elements, the Hasse diagram becomes disconnected 
so that $\binom{18}{2}$ ways to pick the labels $a < b$ 
for the two elements above 
the element labeled with 4 and  only one way to assign $a$ and $b$ 
to those two elements. Once 
we have picked those two elements $a$ and $b$, the next 10 elements must be the smallest 
elements of $\{1,\ldots, 22\} -\{1,2,3,4,,a,b\}$. Once we remove 
those 10 elements, the Hasse diagram again becomes disconnected so that 
there are $\binom{6}{4}$ ways to pick the labels of the vertical segment and only one 
way to order them.

\fig{1.40}{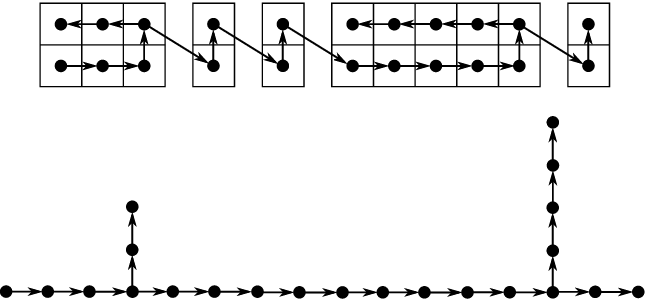}{$\Gamma (3,1,1,5,1)$.}

In general, it is not difficult to see that 
$$
\text{LE}(\Gamma (b_1,b_2,\ldots,b_m))=\prod_{i=1}^{m-1}\binom{b_{m-i}-1+2\sum\limits_{j=m-i+1}^nb_{i}}{b_{m-i}-1}
$$
and then hence
\begin{multline}
\sum_{Q\in\mathcal{GC}_{\col = (b_1,b_2,\ldots,b_m)}}\omega_{P,\mathscr R}(Q)=\\(-1)^{m-1}C_{2b_m}(x)\prod_{i=1}^{m-1}\binom{b_{m-i}-1+2\sum\limits_{j=m-i+1}^nb_{i}}{b_{m-i}-1} C_{2b_i}(x).
\end{multline}

Since a generalized $P$-cluster can not have a block of 2 columns,
\begin{multline*}
GC_{2n,P,\mathscr R}^{0,0,2}(x)=
\sum_{m=1}^n~\sum_{\overset{b_1+ \cdots +b_m=n}{b_i = 1\text{ or }b_i\geq 3}}~ \sum_{Q\in\mathcal{GC}_{\col = (b_1,\ldots,b_m)}}\omega_{P,\mathscr R}(Q)=\\
\sum_{m=1}^n~\sum_{\overset{b_1+ \cdots +b_m=n}{b_i = 1\text{ or }b_i\geq 3}}~ (-1)^{m-1}C_{2b_m}(x)\prod_{i=1}^{m-1}\binom{b_{m-i}-1+2\sum\limits_{j=m-i+1}^mb_{i}}{b_{m-i}-1} C_{2b_i}(x).
\end{multline*}

Using a computer program, we computed that 
\begin{eqnarray}\label{eqn:du_gc_list}
GC_{2,P,\mathscr R}^{0,0,2}(x)&=&1\nonumber\\
GC_{4,P,\mathscr R}^{0,0,2}(x)&=&-1\nonumber\\
GC_{6,P,\mathscr R}^{0,0,2}(x)&=&1+x\nonumber\\
GC_{8,P,\mathscr R}^{0,0,2}(x)&=&-1-7x+x^2\nonumber\\
GC_{10,P,\mathscr R}^{0,0,2}(x)&=&1+22x-10x^2+x^3\nonumber\\
GC_{12,P,\mathscr R}^{0,0,2}(x)&=&-1-50x+   2x^2 -14x^3+   x^4\nonumber\\
GC_{14,P,\mathscr R}^{0,0,2}(x)&=&1+ 95x  +299x^2  -86x^3  -19x^4    +x^5\nonumber\\
GC_{16,P,\mathscr R}^{0,0,2}(x)&=&-1-161x -1796x^2+  1705x^3  -377x^4   -25x^5 +  x^6\nonumber\\
&\cdots&
\end{eqnarray}

Next we compute $GSC_{2n+1,P,\mathscr R}^{1,0,2}(x)$. Suppose $Q\in\mathcal{GSC}_{2n+1,P,\mathscr R}^{1,0,2}$ has $m$ blocks, i.e., $Q=B_1B_2\ldots B_m$. $B_1$ has to be a block consisting of one element, and for $2\leq i\leq m$, each $B_i$ is either order isomorphic to a $P$-cluster or a single column.  We let $\mathcal{GSC}_{\col = (1,b_2,\ldots,b_m)}$ to denote the set of generalized start-clusters such that $\col(B_i)=b_i$. Given $\col(B_i)$ for each $2\leq i\leq m$, 
the set of words of generalized start-clusters 
$\mathcal{GSC}_{\col = (b_1,b_2,\ldots,b_m)}$ is equal to the set 
of linear extensions of the poset whose 
Hasse diagram corresponds is equal to $\Gamma_S(1, b_2, \ldots, b_m)$ 
where $\Gamma_S(1, b_2, \ldots, b_m)$ is the Hasse diagram 
obtained from $\Gamma(b_2, \ldots, b_m)$ by adding 
an arrow from the element in $B_1$ to the left-most element in 
the bottom row of $\Gamma(b_2, \ldots, b_m)$.
 For example, $\Gamma_S (1,3,1,1,5,1)$ is pictured in Figure \ref{fig:downup_gsc_ex}.

\fig{1.40}{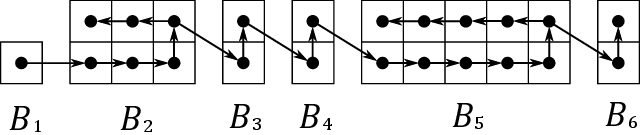}{$\Gamma_S (1,3,1,1,5,1)$.}

It is clear that
\begin{equation}\label{eqn:du_gsc_compute}
\sum_{Q\in\mathcal{GSC}_{\col = (1,b_2,\ldots,b_m)}}\omega_{P,\mathscr R}(Q)=(-1)^{m-1}\text{LE}(\Gamma_S(1,b_2,\ldots,b_m))\prod_{i=1}^m C_{2b_i}(x).
\end{equation}
By comparing Figures \ref{fig:downup_gc_ex} and \ref{fig:downup_gsc_ex}, 
it is easy to see that 
$$
\text{LE}(\Gamma_S(1,b_2,\ldots,b_m))=\text{LE}(\Gamma(b_2,\ldots,b_m)).
$$
Therefore, for $n\geq 1$, 
\begin{eqnarray*}
GSC_{2n+1,P,\mathscr R}^{1,0,2}(x)&=&
\sum_{m=1}^{n+1}~\sum_{\overset{1+b_2+ \cdots +b_m=n+1}{b_i = 1\text{ or }b_i\geq 3}}~ \sum_{Q\in\mathcal{GSC}_{\col = (1,b_2,\ldots,b_m)}}\omega_{P,\mathscr R}(Q)\\
&=&-GC_{2n,P,\mathscr R}^{0,0,2}(x),
\end{eqnarray*}
and $GSC_{1,P,\mathscr R}^{1,0,2}(x)=1$. Then by (\ref{eqn:du_gc_list}), we have
\begin{eqnarray}\label{eqn:du_gsc_list}
GSC_{1,P,\mathscr R}^{1,0,2}(x)&=&1\nonumber\\
GSC_{3,P,\mathscr R}^{1,0,2}(x)&=&-1\nonumber\\
GSC_{5,P,\mathscr R}^{1,0,2}(x)&=&1\nonumber\\
GSC_{7,P,\mathscr R}^{1,0,2}(x)&=&-1-x\nonumber\\
GSC_{9,P,\mathscr R}^{1,0,2}(x)&=&1+7x-x^2\nonumber\\
GSC_{11,P,\mathscr R}^{1,0,2}(x)&=&-1-22x+10x^2-x^3\nonumber\\
GSC_{13,P,\mathscr R}^{1,0,2}(x)&=&1+50x-   2x^2+14x^3-x^4\nonumber\\
GSC_{15,P,\mathscr R}^{1,0,2}(x)&=&-1-95x -299x^2 +86x^3+19x^4-x^5\nonumber\\
GSC_{17,P,\mathscr R}^{1,0,2}(x)&=&1+161x +1796x^2-1705x^3+377x^4+25x^5-x^6\nonumber\\
&\cdots&
\end{eqnarray}

Using these initial values of $GC_{2n,P,\mathscr R}^{0,0,2}(x-1)$ and $GSC_{2n+1,P,\mathscr R}^{1,0,2}(x-1)$ in (\ref{eqn:du_gf_odd}), we computed 
the following initial terms of $A_P(x,t)$: 
\begin{eqnarray*}
A_P(x,t)&=&t+\frac{2}{3!}t^3+\frac{16}{5!}t^5+\frac{266 + 6 x}{7!} t^7+\frac{7623 + 303 x + 9 x^2}{9!} t^9\\
&&+\frac{333475 + 19557 x + 695 x^2 + 10 x^3}{11!}t^{11}+\cdots.
\end{eqnarray*}

Next we compute $GEC_{2n+1,P,\mathscr R}^{0,1,2}(x)$. Suppose $Q\in\mathcal{GEC}_{2n+1,P,\mathscr R}^{0,1,2}$ has $m$ blocks, i.e., $Q=B_1B_2\ldots B_m$. $B_m$ has to be a block consisting of one element, and for $1\leq i\leq m-1$, each $B_i$ is either order isomorphic to a $P$-cluster or a single column.  We let $\mathcal{GEC}_{\col = (b_1,\ldots,b_{m-1},1)}$ to denote the set of generalized end-clusters such that $\col(B_i)=b_i$. Given $\col(B_i)$ for each $1\leq i\leq m-1$, we can represent the set of words of such  generalized end-clusters uniquely by as the set of linear extensions of the poset whose Hasse diagram is 
$\Gamma_E(b_1,\ldots,b_{m-1},1)$ where 
$\Gamma_E(b_1,\ldots,b_{m-1},1)$ is the Hasse diagram obtained 
from $\Gamma(b_1,\ldots,b_{m-1})$ by adding an arrow from the right-most 
element in the top row of $\Gamma(b_1,\ldots,b_{m-1})$ to the element 
in $B_m$. For example, $\Gamma_E (3,1,1,5,1,1)$ is pictured in Figure \ref{fig:downup_gec_ex}.

\fig{1.40}{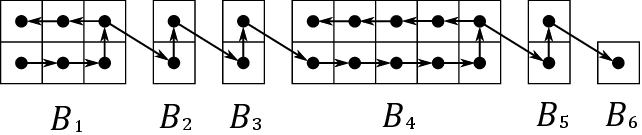}{ $\Gamma_E (3,1,1,5,1,1)$.}

In this case, one can show that 

$$
\text{LE}(\Gamma (b_1,b_2,\ldots,1))=\prod_{i=1}^{m-1}\binom{b_{m-i}+2\sum_{j=m-i+1}^{m-1}b_{i}}{b_{m-i}-1}.
$$
Hence 
\begin{multline*}
GEC_{2n+1,P,\mathscr R}^{0,1,2}(x)=
\sum_{m=1}^n~\sum_{\overset{2b_1+ \cdots +2b_{m-1}+1=2n+1}{b_i = 1\text{ or }b_i\geq 3}}~ \sum_{Q\in\mathcal{GEC}_{\col = (b_1,\ldots,b_{m-1},1)}}\omega_{P,\mathscr R}(Q)=\\
\sum_{m=1}^n~\sum_{\overset{2b_1+ \cdots +2b_{m-1}+1=2n+1}{b_i = 1\text{ or }b_i\geq 3}}~ (-1)^{m-1}C_{2b_m}(x)\prod_{i=1}^{m-1}\binom{b_{m-i}+2\sum\limits_{j=m-i+1}^{m-1}b_{i}}{b_{m-i}-1} C_{2b_i}(x).
\end{multline*}
Using a computer program, we computed that 
\begin{eqnarray}\label{eqn:du_gec_list}
GEC_{1,P,\mathscr R}^{0,1,2}(x)&=&1\nonumber\\
GEC_{3,P,\mathscr R}^{0,1,2}(x)&=&-1\nonumber\\
GEC_{5,P,\mathscr R}^{0,1,2}(x)&=&1\nonumber\\
GEC_{7,P,\mathscr R}^{0,1,2}(x)&=&-1-3x\nonumber\\
GEC_{9,P,\mathscr R}^{0,1,2}(x)&=&1+13x-4x^2\nonumber\\
GEC_{11,P,\mathscr R}^{0,1,2}(x)&=&-1-34x+19x^2-5x^3\nonumber\\
GEC_{13,P,\mathscr R}^{0,1,2}(x)&=&1+70x68x^2+28x^3-6x^4\nonumber\\
GEC_{15,P,\mathscr R}^{0,1,2}(x)&=&
-1 -125x -789x^2+531x^3 +  41x^4-7x^5 \nonumber\\
GEC_{17,P,\mathscr R}^{0,1,2}(x)&=& 1+203x+3551x^2-3973x^3+  2195x^4 +  59x^5    -8x^6\nonumber\\
&\cdots&
\end{eqnarray}

Finally we shall compute  $GSEC_{2n+2,P,\mathscr R}^{1,1,2}(x)$. Suppose $Q\in\mathcal{GSEC}_{2n+2,P,\mathscr R}^{1,1,2}$ has $m$ blocks, i.e., $Q=B_1B_2\ldots B_m$. Because of the definition of generalized start-end clusters, $m\geq 2$. $B_1$ as well as $B_m$ has to be a block consisting of one element, and for $2\leq i\leq m-1$, each $B_i$ is either order isomorphic to a $P$-cluster or a single column.  We let $\mathcal{GEC}_{\col = (1,b_2,\ldots,b_{m-1},1)}$ to denote the set of generalized start-end-clusters such that $\col(B_i)=b_i$. Given $\col(B_i)$ for each $2\leq i\leq m-1$, we can represent the set of words 
of such generalized start-end-clusters as the linear extensions of 
the poset whose Hasse diagram $\Gamma_{SE}(1,b_2,\ldots,b_{m-1},1)$ is obtained from the Hasse diagram $\Gamma(b_2,\ldots,b_{m-1})$ by adding an 
arrow connecting the point corresponding to $B_1$ to the left-most 
element in the bottom row of $\Gamma(b_2,\ldots,b_{m-1})$ and 
add an arrow from the right-most element in the top row of $\Gamma(b_2,\ldots,b_{m-1})$ to a point corresponding to $B_m$. 
For example, $\Gamma_{SE} (1,3,1,1,5,1,1)$ is pictured in Figure \ref{fig:downup_gsec_ex}.

\fig{1.40}{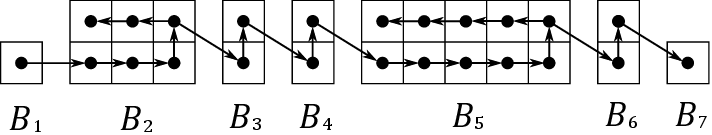}{ $\Gamma_{SE} (1,3,1,1,5,1,1)$.}

Comparing Figure \ref{fig:downup_gec_ex} and \ref{fig:downup_gsec_ex}, it is easy to see that
$$
\text{LE}(\Gamma_{SE}(1,b_2,\ldots,b_{m-1},1))=\text{LE}(\Gamma_E(b_2,\ldots,b_{m-1},1)),
$$
and then hence
$$
GSEC_{2n+2,P,\mathscr R}^{1,1,2}(x)=-GEC_{2n+1,P,\mathscr R}^{0,1,2}(x).
$$
Based on (\ref{eqn:du_gec_list}),
\begin{eqnarray}\label{eqn:du_gsec_list}
GSEC_{2,P,\mathscr R}^{1,1,2}(x)&=&-1\nonumber\\
GSEC_{4,P,\mathscr R}^{1,1,2}(x)&=&1\nonumber\\
GSEC_{6,P,\mathscr R}^{1,1,2}(x)&=&-1\nonumber\\
GSEC_{8,P,\mathscr R}^{1,1,2}(x)&=&1+3x\nonumber\\
GSEC_{10,P,\mathscr R}^{1,1,2}(x)&=& -1-13x+4x^2\nonumber\\
GSEC_{12,P,\mathscr R}^{1,1,2}(x)&=& 1+34x-19x^2+5x^3\nonumber\\
GSEC_{14,P,\mathscr R}^{1,1,2}(x)&=& -1-70x-68x^2-28x^3+6x^4\nonumber\\
GSEC_{16,P,\mathscr R}^{1,1,2}(x)&=&1 +125x +789x^2-531x^3 -  41x^4+7x^5 \nonumber\\
GSEC_{18,P,\mathscr R}^{1,1,2}(x)&=& -1-203x-3551x^2+3973x^3-  2195x^4 -  59x^5    +8x^6\nonumber\\
&\cdots&
\end{eqnarray}

Then substituting the initial values of $GC_{2n,P,\mathscr R}^{0,0,2}(x-1)$, $GSC_{2n+1,P,\mathscr R}^{1,0,2}(x-1)$, $GEC_{2n+1,P,\mathscr R}^{0,1,2}(x-1)$ and $GSEC_{2n+2,P,\mathscr R}^{1,1,2}(x-1)$ into (\ref{eqn:du_gf_odd}) by expressions obtained in (\ref{eqn:du_gc_list}), (\ref{eqn:du_gsc_list}),  (\ref{eqn:du_gec_list}) and  (\ref{eqn:du_gsec_list}), we have computed that 
\begin{eqnarray*}
	B_P(x,t)&=&\frac{1}{2!}t^2+\frac{5}{4!}t^4+\frac{61}{6!} t^6+\frac{1358 + 27 x}{8!} t^8+\frac{48806 + 1611 x + 86 x^2}{10!}t^{10}\\
	&&+\frac{2561283 + 133803 x + 6734 x^2 + 65 x^3}{12!}t^{12}+\cdots.
\end{eqnarray*}

Adding $A_P(x,t)$ and $B_P(x,t)$, we obtain the following 
initial terms of the generating 
function for the distribution of the number of $123654$-matches 
 in down-up permutations. The initial terms of this generating functions are 
\begin{eqnarray*}
	&& t+\frac{1}{2!}t^2+\frac{2}{3!}t^3+\frac{5}{4!}t^4+\frac{16}{5!}t^5+\frac{61}{6!} t^6+\frac{266 + 6 x}{7!} t^7+ \frac{1358+27 x}{8!} t^8\\
	&&+\frac{7623 + 303 x + 9 x^2}{9!} t^9+\frac{48806 + 1611 x + 86 x^2}{10!}t^{10}+\cdots .
\end{eqnarray*} 

\subsubsection{Pattern 124356 in 3-generalized Euler permutations}

In this section, we shall study the generating 
function of $\tau$-matches where $\tau = 124356$ in generalized 
Euler permutations $\mathcal{E}_{3n+j}^{0,j,3}$ for 
$j = 0, 1, 2$. We have pictured $\tau$ as a $3 \times 2$ array 
$\Gamma_\tau$ at 
the top-left of Figure \ref{fig:P124356} and its corresponding Hasse diagram 
at the top-right of Figure \ref{fig:P124356}.  We picked $\tau$ because 
it has two key properties. The first key property is that 
there is unique $\Gamma_\tau$-cluster for 
in $\mathcal{CM}_{3n,\Gamma_\tau}^{0,0,3}$ for every $n \geq 1$. 
That is, a $\Gamma_\tau$ cluster with $n$ columns must be a linear 
extension of the Hasse diagram which is obtained from the Hasse diagram $D$ 
of $\Gamma_\tau$ by adding the arrows of the Hasse diagram for 
$\Gamma_\tau$ in any pair of consecutive columns. 
This process is pictured at the bottom left of Figure \ref{fig:P124356}
in the case where $n =6$. 
We claim that vertical arrows in each column other than the first 
and last column are redundant. That is, for each vertical arrow 
in columns $2, \ldots , n=1$, there is another directed path to its end points  that does not use vertical arrows. If we eliminate such vertical 
arrows as pictured at the bottom-right of Figure \ref{fig:P124356}, it then 
easy to see that remaining arrows linear order the points with 
the smallest element at the left-most cell in the bottom row 
of the diagram and the largest element in the top row of the diagram. Moreover, in the cluster, we must mark column except the last one with an $x$. 
It follows that $C_{3n,\Gamma_\tau}^{0,0,3}(x) = x^{n-1}$ for 
all $n \geq 2$. 

This fact is actually a special case of a more general theorem 
proved by Harmse and Remmel \cite{HR}. That is, Harmse and Remmel 
completely classified those elements 
$P \in \mathcal{P}_{2k}^{0,0,k}$ such that there is 
a unique $P$-cluster with $n$ columns for every $n$ greater than 
or equal to 2 as those $P$ such that in every pair of consecutive 
rows $j$ and $j+1$, either left column contains a pair of consecutive elements 
or the right column contains a pair of consecutive elements. For example, 
in $\Gamma_\tau$, 1 and 2 are consecutive elements that appear in 
the left column of rows 1 and 2 and 5 and 6 are consecutive elements 
that appear in the left column of rows 2 and 3.

\fig{1.20}{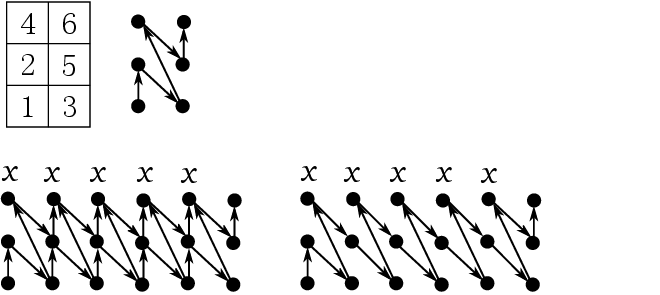}{$\Gamma_\tau$ and $\Gamma_\tau$-clusters.}

Let $\mathscr R$ be the relation that holds between any pair of increasing 
columns of integers $(C,D)$ if and only if the top element of 
$C$ is greater than the bottom element of $D$. Then for all $n \geq 0$, 
\begin{eqnarray*}
\mathcal{E}_{3n}^{0,0,3} &=&  \mathcal{P}_{3n, \mathscr R}^{0,0,3}, \\
\mathcal{E}_{3n+1}^{0,0,3} &=&  \mathcal{P}_{3n+1, \mathscr R}^{0,1,3}, \ \mbox{and}\\
\mathcal{E}_{3n+2}^{0,0,3} &=&  \mathcal{P}_{3n+2, \mathscr R}^{0,2,3}.
\end{eqnarray*}

It follows from Theorems \ref{thm:mainGamma} and \ref{thm:jmainGamma} that 
\begin{eqnarray*}
&&1+ \sum_{j=0}^2 \sum_{n \geq 1} \frac{t^{3n+j}}{(3n+j)!} 
\sum_{\sg \in \mathcal{E}^{0,j,3}_{3n+j}} x^{\tmch(\sg)} =  \\
&&\frac{1+\sum_{n\geq 0}\left(\frac{t^{3n+1}}{(3n+1)!}GEC^{0,0,3}_{3n+1,\Gamma_\tau,\mathscr R}(x-1)+\frac{t^{3n+2}}{(3n+2)!}GEC^{0,0,3}_{3n+2,\Gamma_\tau,\mathscr R}(x-1)\right)}{1-\sum\limits_{n\geq 1}\frac{t^{3n}}{(3n)!}GC^{0,0,3}_{3n,\Gamma_\tau,\mathscr R}(x-1)}.
\end{eqnarray*}

The second key property of $\Gamma_\tau$ is that in generalized 
clusters and generalized end-clusters $(B_1, \ldots, B_m)$, 
the order of the elements is completely determined. That is, 
the arrows between consecutive blocks $B_i,B_{i+1}$ are directed 
from the largest element in block $B_i$ to the smallest 
element in block $B_{i+1}$. This is illustrated in Figure 
\ref{fig:P124356A} where we have pictured 
a generalized $\Gamma_\tau$-cluster at the top and a 
generalized $\Gamma_\tau$-end-cluster at the bottom.

\fig{1.20}{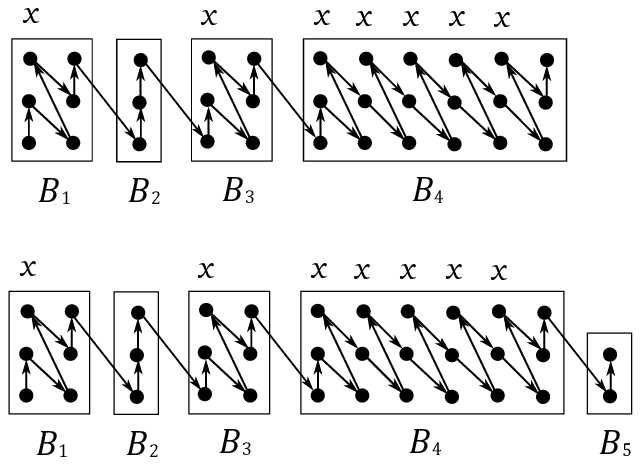}{The structure of generalized 
$\Gamma_\tau, \mathscr R$-clusters.}

Next we shall compute $GC^{0,0,3}_{3n,\Gamma_\tau,\mathscr R}(x)$, 
$GEC^{0,0,3}_{3n+1,\Gamma_\tau,\mathscr R}(x)$, and 
$GEC^{0,0,3}_{3n+2,\Gamma_\tau,\mathscr R}(x)$. We will start 
with computing $GC^{0,0,3}_{3n,\Gamma_\tau,\mathscr R}(x)$. 
It is easy check 
that  $GC^{0,0,3}_{3,\Gamma_\tau,\mathscr R}(x) =1$ and 
$GC^{0,0,3}_{6,\Gamma_\tau,\mathscr R}(x) = x-1$. We claim 
that for $n \geq 2$, 
$$GC^{0,0,3}_{3n,\Gamma_\tau,\mathscr R}(x) = (x-1)GC^{0,0,3}_{3n,\Gamma_\tau,\mathscr R}(x).
$$
That is, if $C = (B_1, \ldots, B_m)$ is a generalized 
$\Gamma_\tau,\mathscr R$-cluster, then either 
$B_1$ consists of a single column in which case 
$$w_{\Gamma_\tau,\mathscr R}(C) = 
- w_{\Gamma_\tau,\mathscr R}((B_2, \ldots, B_m)), $$ 
or $B_1$ has $k \geq 2$ columns and weight $x^{k-1}$ in which 
case we can replace $B_1$ by the unique cluster 
$B_1^\prime$ with $k-1$ columns and  
$$w_{\Gamma_\tau,\mathscr R}(C) = 
x w_{\Gamma_\tau,\mathscr R}((B_1^\prime, B_2, \ldots, B_m)). $$

We let
$$
GC(x,t):=\sum_{n\geq 1}\frac{t^{3n}}{(3n)!}
	GC^{0,0,3}_{3n,\Gamma_\tau,\mathscr R}(x).
$$
Based on the recursion above, we have
$$
\frac{\partial^3 GC(x,t)}{\partial t^3}-1=(x-1)~
GC(x,t),
$$
with boundary conditions
$$
GC(0,t)=0,~~\left.\frac{\partial GC_(x,t)}{\partial t}\right|_{t=0}=0,~~\left.\frac{\partial^2 GC(x,t)}{\partial t^2}\right|_{t=0}=0.
$$
We solved this differential equation using Mathematica and obtained 
that 
\begin{equation*}
	GC(x,t)=-\frac{e^{t \sqrt[3]{x-1}}+e^{-\frac{1}{2} i \left(\sqrt{3}-i\right) t \sqrt[3]{x-1}}+e^{\frac{1}{2} i \left(\sqrt{3}+i\right) t \sqrt[3]{x-1}}-3}{3-3 x}.
\end{equation*}

Next we compute $GEC^{0,0,3}_{3n+1,\Gamma_\tau,\mathscr R}(x)$. 
One can check that $GEC^{0,0,3}_{1,\Gamma_\tau,\mathscr R}(x)=1$ and 
$GEC^{0,0,3}_{4,\Gamma_\tau,\mathscr R}(x)=-1$. Then we can 
argue as we did in the computation of 
$GEC^{0,0,3}_{3n+1,\Gamma_\tau,\mathscr R}(x)$ that for $n \geq 2$, 
$$ GEC^{0,0,3}_{3n+1,\Gamma_\tau,\mathscr R}(x)=(x-1)GEC^{0,0,3}_{3n-2,\Gamma_\tau,\mathscr R}(x).
$$
We let
$$
GE_1(x,t):=\sum_{n\geq 1}\frac{t^{3n+1}}{(3n+1)!}
GEC^{0,0,3}_{3n+1,\Gamma_\tau,\mathscr R}(x)
$$
Based on the recursion above, we have
$$
\frac{\partial^3 GE_1(x,t)}{\partial t^3}+t=(x-1)~GE_1(x,t),
$$
with boundary conditions
$$
GE_1(x,0)=0,~~\left.\frac{\partial GE_1(x,t)}{\partial t}\right|_{t=0}=0,~~\left.\frac{\partial^2 GE_1(x,t)}{\partial t^2}\right|_{t=0}=0.
$$
Again, we used Mathematica to compute that 
\begin{equation*}
	GE_1(x,t)=\frac{e^{-\sqrt[3]{-1} t z} \left(3 t z e^{\sqrt[3]{-1} t z}+\left(1+(-1)^{2/3}\right) e^{i \sqrt{3} t z}-e^{\left(1+\sqrt[3]{-1}\right) t z}-(-1)^{2/3}\right)}{3 z^4},
\end{equation*}
where $z=\sqrt[3]{x-1}$.

The computation of $GEC^{0,0,3}_{3n+2,\Gamma_\tau,\mathscr R}(x)$ 
is similar to the computation 
of $GEC^{0,0,3}_{3n+1,\Gamma_\tau,\mathscr R}(x)$.
That is, we have  
\begin{eqnarray*}
	GEC^{0,0,3}_{2,\Gamma_\tau,\mathscr R}(x)&=&1, \\
GEC^{0,0,3}_{5,\Gamma_\tau,\mathscr R}(x)&=&-1, \ \mbox{and}\\
	GEC^{0,0,3}_{3n+2,\Gamma_\tau,\mathscr R}(x)&=&(x-1)
GEC^{0,0,3}_{3n-1,\Gamma_\tau,\mathscr R}(x) \ \mbox{for} \ n \geq 2.
\end{eqnarray*}
We let
$$
GE_2(x,t):=\sum\limits_{n\geq 1}\frac{t^{3n+2}}{(3n+2)!}
GEC^{0,0,3}_{3n+2,\Gamma_\tau,\mathscr R}(x).
$$
Based on the recursion above, we have the following differential equation
$$
\frac{\partial^3 GE_2(x,t)}{\partial t^3}+\frac{t^2}{2!}=(x-1)~GE_2(x,t),
$$
with boundary conditions
$$
GE_2(x,0)=0,~~\left.\frac{\partial GE_2(x,t)}{\partial t}\right|_{t=0}=0,~~\left.\frac{\partial^2 GE_2(x,t)}{\partial t^2}\right|_{t=0}=0.
$$
Then we used Mathematica to compute that 
$$
GE_2(x,t)=\frac{e^{-\sqrt[3]{-1} t z} \left(3 t^2 (x-1)^{2/3} e^{\sqrt[3]{-1} t z}-2 (-1)^{2/3} e^{i \sqrt{3} t z}-2 e^{\left(1+\sqrt[3]{-1}\right) t z}+2 (-1)^{2/3}+2\right)}{6 z^5},
$$
where $z=\sqrt[3]{x-1}$.

Finally we have
\begin{multline*}
	1+ \sum{j=0}^2 \sum_{n \geq 1} \frac{t^{3n+j}}{(3n+j)!} 
\sum_{\sg \in \mathcal{E}^{0,j,3}_{3n+j}} x^{\tmch(\sg)} = 
	\frac{1+\left(GE_1(x-1,t)+t\right)+\left(GE_2(x-1,t)+\frac{t^2}{2}\right)}{1-GC(x-1,t)}\\
	=
	\frac{-e^{-\sqrt[3]{-1} t \sqrt[3]{x-2}}}{2 (x-2)^{2/3} \left(e^{t \sqrt[3]{x-2}}+e^{-\frac{1}{2} i \left(\sqrt{3}-i\right) t \sqrt[3]{x-2}}+e^{\frac{1}{2} i \left(\sqrt{3}+i\right) t \sqrt[3]{x-2}}-3 x+3\right)}\\
	\times \left( 3 (x-2)^{2/3} e^{\sqrt[3]{-1} t \sqrt[3]{x-2}} (t (t+2) (x-1)+2 (x-2))-2 \left(\sqrt[3]{x-2}+1\right) e^{\left(1+\sqrt[3]{-1}\right) t \sqrt[3]{x-2}}\right.\\
	\left.+2 \left((-1)^{2/3} \sqrt[3]{x-2}+\sqrt[3]{x-2}-(-1)^{2/3}\right) e^{i \sqrt{3} t \sqrt[3]{x-2}}-2 (-1)^{2/3} \sqrt[3]{x-2}+2 (-1)^{2/3}+2\right).
\end{multline*}

A few initial terms of expansion of 
$1+ \sum{j=0}^2 \sum_{n \geq 1} \frac{t^{3n+j}}{(3n+j)!} 
\sum_{\sg \in \mathcal{E}^{0,j,3}_{3n+j}} x^{\tmch(\sg)}$ are
\begin{eqnarray*}
	&&1 + t + \frac{t^2}{2!} + \frac{t^3}{3!}+3\frac{t^4}{4!}+9\frac{t^5}{5!}+(18 + x)\frac{t^6}{6!}+ (93 + 6 x)\frac{t^7}{7!}+ (450 + 27 x)\frac{t^8}{8!}\\
	&&+( 1348 + 
	164 x + x^2)\frac{t^9}{9!}+(9936 + 1314 x + 9 x^2)\frac{t^{10}}{10!}+(66150 + 8397 x + 54 x^2)\frac{t^{11}}{11!}\\
	&&+\cdots.
\end{eqnarray*}

The key properties that made this example work is that 
$\Gamma_\tau$ had unique $\Gamma_\tau$-clusters, i.e. Harmse 
and Remmel's characterization applies,  and that the 
bottom-left element of $\Gamma_\tau$ is equal to 1 and the 
top-right element of $\Gamma_\tau$ is the largest element. 
There are many such examples $k$-Euler permutations where 
$k > 3$.  For example, when $k=4$, $\alpha = 12354678$ is such a permutation. 
Thus we can carry out a similar analysis for the distribution 
of $\alpha$-matches in 4-generalized Euler permutations.

\section{Joint Clusters and Generalized Joint Clusters}\label{chap:joint}

In this section, we extend the notions $\Gamma$-clusters and generalized 
$\Gamma$-clusters for a single set of patterns 
$\Gamma$ to joint clusters and generalized joint clusters for a sequence 
of sets of patterns $\Gamma_1, \ldots, \Gamma_s$ such that  
$\Gamma_i \subseteq \mathcal{P}^{0,0,k}_{r_ik}$ where $r_i \geq 2$ 
for $i =1, \ldots, s$. Our goal here is to 
compute generating functions of the form 
$$G^{i,j,k}_{\Gamma_1, \ldots, \Gamma_s}(x_1, 
\ldots x_m,t) = 
1+\sum_{n\geq1}\frac{t^{kn}}{(kn)!}\sum_{F\in\mathcal{P}^{i,j,k}_{kn}}\prod_{i=1}^s x_i^{\Gamma_i\text{-mch}(F)}.
$$
Similarly, if we are given some binary relation $\mathscr R$ on a pairs of columns, then we want to the compute multivariate generating function 
$$ G^{i,j,k}_{\Gamma_1, \ldots, \Gamma_s,\mathscr R}(x_1, 
\ldots x_m,t) =
1+\sum_{n\geq1}\frac{t^{kn}}{(kn)!}\sum_{F\in\mathcal{P}^{0,0,k}_{kn,\mathscr R}}\prod_{i=1}^s x_i^{\Gamma_i\text{-mch}(F)}.
$$
For each  the theorems in the previous sections, we will state 
multivariate analogues of that theorem.  In each case, the proofs 
of the multivariate analogues are essentially the same of the corresponding 
proofs in the previous sections so that we shall not supply the details.

\subsection{Joint clusters}

Let $\Gamma_i \subseteq \mathcal{P}^{0,0,k}_{r_ik}$ where $r_i \geq 2$ 
for $i = 1, \ldots, s$. 
For any $i,j \geq 0$, $k \geq 2$, and $n \geq 1$, we let 
$\mathcal{MP}^{i,j,k}_{kn,\Gamma_1, \ldots, \Gamma_s}$ 
denote the set of all fillings $F \in \mathcal{P}^{i,j,k}_{i+nk+j}$ 
where we have 
marked some of the $\Gamma_i$-matches in $F$ by placing an 
$x_i$ on top of the column that starts a $\Gamma_i$-match in $F$ for 
$i = 1, \ldots, s$. We let $m_{\Gamma_i}(F)$ equal the 
number of columns in $F$ which are marked with $x_i$. 
Note that our definitions allow a given column to be marked 
with some subset of $\{x_1, \ldots, x_s\}$. 
For example, suppose that $\Gamma_1 = \{P\}$ and $\Gamma_2 = \{Q\}$ where 
$Q=\begin{array}{|c|c|}
 \hline 3 & 4 \\
 \hline 1 & 2 \\
 \hline
 \end{array}$  and $P=\begin{array}{|c|c|c|}
 \hline 4&5&6 \\
 \hline 1 & 2 &3\\
 \hline
 \end{array}$.
Then Figure \ref{fig:Gamma12} pictures an element of  
$\mathcal{MP}^{0,0,k}_{12,\Gamma_1,\Gamma_2}$ where there are 
$\Gamma_1$-matches starting at columns 2 and 3 and $\Gamma_2$-matches 
starting at columns 2, 3, and 4. Thus we can have columns which have more than one label since a columns 2 and 3  starts a $\Gamma_1$-match and 
a $\Gamma_2$-match.  For the $F$ pictured in 
Figure \ref{fig:Gamma12}, $m_{\Gamma_1}(F) =2 = m_{\Gamma_2}(F)$. 

\fig{1.20}{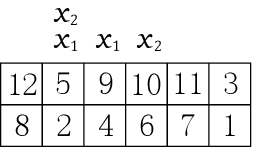}{An element of $\mathcal{MP}^{0,0,k}_{kn,\Gamma_1,\Gamma_2}$.}

A {\bf joint $\Gamma_1, \ldots, \Gamma_s$-cluster} is a filling  of $F \in \mathcal{MP}^{0,0,k}_{kn,\Gamma_1, \ldots, \Gamma_s}$ such that 
\begin{enumerate}
 \item every column of $F$ is contained in a marked 
$\Gamma_i$-match of $F$ for some $i=1, \ldots, s$ and  
\item for all $1\leq a <b \leq n$, if columns $a$ and $b$ are marked and 
there is no column $c$ such that $a < c < b$  and column 
$c$ is marked,  
then there is an $i$ such that there is a $\Gamma_i$-match 
starting at column $a$ which is marked with $x_i$ and that $\Gamma_i$-match 
contains column $b$.   
\end{enumerate}
We let $\mathcal{CM}^{0,0,k}_{kn,\Gamma_1, \ldots, \Gamma_s}$ denote the set of all joint $\Gamma_1, \ldots, \Gamma_s$-clusters 
in $\mathcal{MP}^{0,0,k}_{kn,\Gamma_1, \ldots, \Gamma_s}$. 
For each $n \geq 2$, we define the cluster polynomial 
$$C^{0,0,k}_{kn,\Gamma_1, \ldots, \Gamma_s}(x_1, \ldots, x_s) = 
\sum_{F \in \mathcal{CM}^{0,0,k}_{kn,\Gamma_1, \ldots, \Gamma_s}} 
\prod_{i=1}^s x_i^{m_{\Gamma_i}(F)}.$$  
By convention, we let $C^{0,0,k}_{k,\Gamma}(x_1, \ldots, x_s) =1$.

Then we have the following 
multivariate version of Theorem \ref{thm:Gammacluster}.
\begin{theorem}\label{thm:jointcluster}
	 Let $\Gamma_i \subseteq \mathcal{P}^{0,0,k}_{r_ik}$ where $r_i \geq 2$ for $i=1, \ldots, s$. Then 
	 \begin{multline}\label{eqn:thm_joint}
	 1+\sum_{n \geq 1} \frac{t^{kn}}{kn!} 
	 \sum_{F \in \mathcal{P}^{0,0,k}_{kn}} 
\prod_{i=1}^mx_i^{\Gamma_i\text{-}\mathrm{mch}(F)} =  \\
	 \frac{1}{1-\sum_{n \geq 1} 
	 	\frac{t^{kn}}{(kn)!} C^{0,0,k}_{kn,\Gamma_1, \ldots, \Gamma_s}(x_1-1,x_2-1,\cdots,x_s-1)}.
	 \end{multline}
\end{theorem}

\subsubsection{Example}

For any fixed $k > 0$, we let $P_{k,m}$ be the element of 
$\mathcal{P}^{0,0,k}_{km}$ whose word is the identity permutation. 
Thus the $i^{\mathrm{th}}$ column of $P_{k,m}$ contains the numbers 
$k(i-1)+1, \ldots, ki$, reading from bottom to top. For example, 
$P_{3,4}$ and its corresponding Hasse diagram 
is pictured in Figure \ref{fig:P34}.

\fig{1.20}{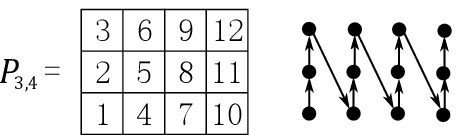}{The pattern $P_{3,4}$.}

Our goal in this subsection is to compute the generating function 
\begin{equation}\label{gfmm+1}
1+\sum_{n \geq 1} \frac{t^{kn}}{kn!} 
	 \sum_{F \in \mathcal{P}^{0,0,k}_{kn}} 
x^{P_{k,m}\text{-}\mathrm{mch}(F)}y^{P_{k,m+1}\text{-}\mathrm{mch}(F)}
\end{equation}
for any $m \geq 2$ and $k \geq 1$.

By Theorem \ref{thm:Gammacluster}, we can compute (\ref{gfmm+1}) if 
we can compute the polynomials $G^{0,0,k}_{kn,P_{k,m},P_{k,m+1}}(x,y)$. 
Our definitions ensures that if $C$ is a $P_{k,m},P_{k,(m+1)}$ cluster, 
then the underlying array of $C$ is just $P_{k,t}$ for some $t$. 

It is easy to check that 
\begin{align*}
&C^{0,0,k}_{k,P_{k,m},P_{k,m+1}}(x,y) =1, \\
&C^{0,0,2k}_{k,P_{k,m},P_{k,m+1}}(x,y) = \cdots = C^{0,0,k}_{(m-1)k,P_{k,m},P_{k,m+1}}(x,y) =0, \ \mbox{and} \\
&C^{0,0,k}_{km,P_{k,m},P_{k,m+1}}(x,y) =x.
\end{align*}

For $n \geq m+1$, we can classify the $P_{k,m},P_{k,(m+1)}$ clusters $C$ 
with $n$ columns by 
the label of the first column. We have to cases.\\
\ \\
{\bf Case 1.} $y$ is a label on the first column $C$. 
In this case, we let $j\geq 1$ be the least number such that 
column $j+1$ is also marked. Clearly, we must have that 
$1 \leq j \leq m$ because if columns 
$2, \ldots, m$ are not marked, then column $m+1$ must be marked. 
In this case, we can remove, 
the first $j$ columns of $C$ and we will be left with 
a $P_{k,m},P_{k,(m+1)}$ cluster with $n-j$ columns. Note that 
that the first columns can be marked with $x$ or not. It follows 
that such clusters contribute 
$$(1+x)y\sum_{j=1}^m  C^{0,0,k}_{k(n-j),P_{k,m},P_{k,m+1}}(x,y)$$ to 
$C^{0,0,k}_{kn,P_{k,m},P_{k,m+1}}(x,y)$. \\
\ \\
{\bf Case 2.} $y$ is not a label on the first column $C$. 
In this case, the first column of $C$ is marked with an $x$. Again,
we let $j\geq 1$ be the least number such that 
column $j+1$ is also marked. In this case, we must have that 
$1 \leq j \leq m-1$ because if columns 
$2, \ldots, m-1$ are not marked, then column $m$ must be marked. 
In this case, we can remove, 
the first $j$ columns of $C$ and we will be left with 
a $P_{k,m},P_{k,(m+1)}$ cluster with $n-j$ columns. Note that 
that the first columns can be marked with $x$ or not. It follows 
that such clusters contribute 
$$x\sum_{j=1}^{k-1}  C^{0,0,k}_{k(n-j),P_{k,m},P_{k,m+1}}(x,y)$$ to 
$C^{0,0,k}_{kn,P_{k,m},P_{k,m+1}}(x,y)$. \\
\ \\

It follows that for $n \geq m+1$, 
\begin{eqnarray}\label{mm+1rec}
C^{0,0,k}_{kn,P_{k,m},P_{k,m+1}}(x,y) &=&  (xy+y)
C^{0,0,k}_{k(n-m),P_{k,m},P_{k,m+1}}(x,y)+ \nonumber \\
&&(xy+x+y)\sum_{j=1}^{m-1}  C^{0,0,k}_{k(n-j),P_{k,m},P_{k,m+1}}(x,y).
\end{eqnarray}

Let 
\begin{equation}
C^{0,0,k}_{P_{k,m},P_{k,m+1}}(x,y,t) = 
\sum_{n \geq 1} \frac{t^{kn}}{(kn)!}  
C^{0,0,k}_{kn,P_{k,m},P_{k,m+1}}(x,y).
\end{equation}

Then it follows that 
\begin{eqnarray}\label{eq1.recmm+1}
C^{0,0,k}_{P_{k,m},P_{k,m+1}}(x,y,t)&=& \nonumber
\frac{t^k}{k!} + \frac{xt^{mk}}{(mk)!}+ \\
&& (xy+x_y)\sum_{j=1}^{m-1} \sum_{n \geq 2} 
C^{0,0,k}_{k(n-j),P_{k,m},P_{k,m+1}}(x,y)+ \nonumber\\
&&(xy+y) \sum_{n \geq 2} C^{0,0,k}_{k(n-k),P_{k,m},P_{k,m+1}}(x,y).
\end{eqnarray}
Taking the partial derivative $\frac{\partial^{mk}}{\partial t^{mk}}$ of 
both sides of this equation leads to the following partial differential 
equation:
\begin{eqnarray}\label{eq2.recmm+1}
\frac{\partial^{mk}}{\partial t^{mk}}C^{0,0,k}_{P_{k,m},P_{k,m+1}}(x,y,t)&=& x+(xy+x_y)\sum_{j=1}^{m-1} \frac{\partial^{(m-j)k}}{\partial t^{(m-j)k}}
\left( C^{0,0,k}_{P_{k,m},P_{k,m+1}}(x,y,t)-\frac{t^k}{k!}\right)+ \nonumber\\
&&(xy+y) C^{0,0,k}_{P_{k,m},P_{k,m+1}}(x,y,t).
\end{eqnarray}

In the case where $m=2$ and $k=1$ so that $P_{m,k}=12$ and $P_{k,m+1} = 123$, 
then we obtain the differential equation 
$$\frac{\partial^2}{\partial t^2}C^{0,0,1}_{12,123}(x,y,t)= 
(xy+x+y)\frac{\partial}{\partial t}C^{0,0,1}_{12,123}(x,y,t)+(xy+y)(C^{0,0,1}_{12,123}(x,y,t)-1)$$
with boundary conditions
$$
C^{0,0,1}_{12,123}(x,y,0)=0~~~\text{and}~~~\left.\frac{\partial C^{0,0,1}_{12,123}(x,y,t)}{\partial t}\right|_{t=0}=1.
$$

Solving the differential equation above, we get
\begin{eqnarray*}
	C^{0,0,1}_{12,123}(x,y,t)&=&1-e^{\frac{1}{2} t (x y+x+y)} \cosh \left(\frac{1}{2} t \sqrt{(x y+x+y)^2+4 (x+1) y}\right)\\&&+\frac{(x y+x+y+2) e^{\frac{1}{2} t (x y+x+y)} \sinh \left(\frac{1}{2} t \sqrt{(x y+x+y)^2+4 (x+1) y}\right)}{\sqrt{(x y+x+y)^2+4 (x+1) y}}.
\end{eqnarray*}

Then
\begin{eqnarray*}
1+ \sum_{n \geq 1} \frac{t^n}{n!} \sum_{\sg \in S_n} x^{12\mathrm{-mch}(\sg)}
y^{123\mathrm{-mch}(\sg)} &=&\frac{1}{1-C^{0,0,1}_{12,123}(x-1,y-1,t)}\\
	&=&\frac{2 z(x,y)e^{ \frac{1}{2} t \left(-x y+z(x,y)+1\right)}}{\left(-x y+z(x,y)-1\right) e^{t z(x,y)}+x y+z(x,y)+1},
\end{eqnarray*}
where $z(x,y)=\sqrt{x (y (x y+2)-4)+1}$.

A few initial terms of the expansion of $1+ \sum_{n \geq 1} \frac{t^n}{n!} \sum_{\sg \in S_n} x^{12\mathrm{-mch}(\sg)}
y^{123\mathrm{-mch}(\sg)}$ are
\begin{eqnarray*}
	&&1+t+ (1+x) \frac{t^2}{2!}+(1+4 x+x^2 y) \frac{t^3}{3!}+\ (1+11 x+5 x^2+6 x^2 y+x^3 y^2) \frac{t^4}{4!}\\
	&+&\left(1+26 x+43 x^2+23 x^2 y+18 x^3 y+8 x^3 y^2+x^4 y^3\right)\frac{t^5}{5!}\\
	&+& \left(1+57 x+230 x^2+61 x^3+72 x^2 y+202 x^3 y+39 x^3 y^2+47 x^4 y^2+10 x^4 y^3+x^5 y^4\right) \frac{t^6}{6!}\\
	&+&\cdots.
\end{eqnarray*}

In the case where $k=m=2$, the we obtain the differential equation
$$\frac{\partial^4}{\partial t^4} C^{0,0,2}_{P_{2,2},P_{2,3}}(x,y,t)= 
(xy+x+y)\frac{\partial^2}{\partial t^2}C^{0,0,2}_{P_{2,2},P_{2,3}}(x,y,t)+(xy+y)(C^{0,0,2}_{P_{2,2}.P_{2,3}}(x,y,t)-1)$$
with boundary conditions
\begin{eqnarray*}
&&C^{0,0,2}_{P_{2,2},P_{2,3}}(x,y,0)=0, \\
&&\left. \frac{\partial C^{0,0,2}_{P_{2,2},P_{2,3}}(x,y,t)}{\partial t}\right|_{t=0}=0,\\
&&\left. \frac{\partial^2 C^{0,0,2}_{P_{2,2},P_{2,3}}(x,y,t)}{\partial t^2}\right|_{t=0}=1, \ \mbox{and} \\
&&\left. \frac{\partial^3 C^{0,0,2}_{P_{2,2},P_{2,3}}(x,y,t)}{\partial t^3}\right|_{t=0}=0.
\end{eqnarray*}

We used Mathematica to solve this differential equation and 
obtained that 
\begin{eqnarray*}
C^{0,0,2}_{P_{2,2},P_{2,3}}(x,y,t)&=& \frac{1}{4} 
\left( 4+2\cosh(\frac{t}{\sqrt{2}}B(x,y,t))
\left(\frac{1+(1+x)(1+y)}{A(x,y,t)}-1\right) - \right. \\
&&\left. 2\cosh(\frac{t}{\sqrt{2}}C(x,y,t))
\left(\frac{1+(1+x)(1+y)}{A(x,y,t)}+1\right)\right)
\end{eqnarray*}
where 
\begin{eqnarray*}
A(x,y,t) &=& \sqrt{(x+y+xy)^2+4(1+x)y}, \\
B(x,y,t)&=& \sqrt{xy+x+y+A(x,y,t)}, \ \mbox{and} \\
C(x,y,t)&=& \sqrt{xy+x+y-A(x,y,t)}.
\end{eqnarray*}
Thus 
\begin{equation}
1+ \sum_{n \geq 1} \frac{t^{2n}}{(2n)!} 
\sum_{F \in P^{0,0,2}_{2n}} x^{P_{2,2}\mathrm{-mch}(F)}
y^{P_{2,3}\mathrm{-mch}(F)} = 
\frac{1}{1-C^{0,0,2}_{P_{2,2},P_{2,3}}(x-1,y-1,t)}.
\end{equation}
The first few terms of this series are 
\begin{align*}
&1+\frac{t^2}{2}+\frac{t^4}{4!}(5+x)+\frac{t^6}{6!}(61+28x+x^2y)\\
&+\frac{t^8}{8!}(1385+1011x+69x^2+54x^2y+x^3y^2)+\\
&+\frac{t^{10}}{10!}(50521+50666x+8523x^2+3183x^2y+418x^3y+130x^3y^2+x^4y^3) + \cdots .
\end{align*}

\subsection{Generalized joint clusters}

Next suppose that we are given a binary relation $\mathscr R$ between $k \times 1$ arrays of integers and sequences of 
sets of patterns  where $\Gamma_i \subseteq \mathcal{P}^{0,0,k}_{r_ik}$ 
where $r_i \geq 2$ for , $1\leq i\leq m$.  
Given our definition of joint $\Gamma_1, \ldots, \Gamma_s$-clusters, we can easily modify the 
definition of generalized $\Gamma$-clusters given in 
Definition \ref{def:gc} to 
generalized joint $\Gamma_1, \ldots, \Gamma_s$-clusters. 
\begin{definition}
We say that $Q \in \mathcal{MP}^{0,0,k}_{kn,\Gamma_1, \ldots, \Gamma_s}$ is a 
	\textbf{ generalized joint $\Gamma_1, \ldots, \Gamma_s,\mathscr R$-cluster} if 
	we can write $Q=B_1B_2\cdots B_h$ where $B_i$ are blocks of consecutive columns in $Q$ such that  
	\begin{enumerate}
		\item either $B_i$ is a single column or $B_i$ consists of $r$-columns 
		where $r \geq 2$, $\red(B_i)$ is a joint $\Gamma_1, \ldots, \Gamma_s$-cluster in $\mathcal{MP}^{0,0,k}_{kn,\Gamma_1, \ldots, \Gamma_s}$, and any pair of consecutive columns in $B_i$ are in $\mathscr R$ and
		\item for $1\leq i\leq h-1$, 
the pair $(\last(B_i),\first(B_{i+1}))$ is not in $\mathscr R$.  
	\end{enumerate}
\end{definition}

Let $\mathcal{GC}^{0,0,k}_{kn,\Gamma_1, \ldots, \Gamma_s,\mathscr R}$ denote the set of all generalized joint 
$\Gamma_1, \ldots, \Gamma_s,\mathscr R$-clusters which have $n$ columns of height $k$. For example, 
suppose that $\mathscr R$ is the relation that holds for a pair of columns $(C,D)$ if 
and only if the top element of column $C$ is greater than the bottom element 
of column $D$ and $\Gamma_1 = \{P_1\}$ and $\Gamma_2=\{P_2\}$ where 
$P_1 = \begin{array}{|c|c|c|}
\hline 
6 & 5 & 4 \\
\hline
1 & 2 & 3 \\
\hline
\end{array}$ and $P_2 = \begin{array}{|c|c|c|c|}
\hline 
8 & 7 & 6 & 5 \\
\hline
1 & 2 & 3 & 4\\
\hline
\end{array}$.
Then in Figure \ref{fig:PRjcluster}, we have pictured a generalized joint 
$\Gamma_1,\Gamma_2$-cluster with 
5 blocks $B_1,B_2,B_3,B_4,B_5$. 

\fig{1.20}{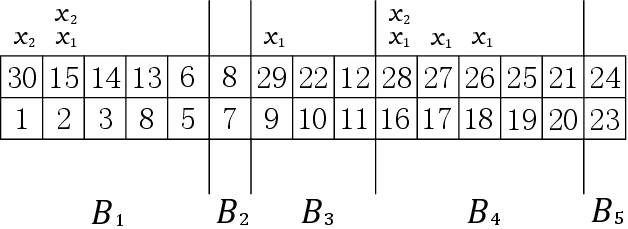}{A generalized $\Gamma_1,\Gamma_2$-cluster.}

Given $Q =B_1 B_2 \ldots B_h \in \mathcal{GC}^{0,0,k}_{kn,\Gamma_1, \ldots, \Gamma_s,\mathscr R}$, we 
define the weight of $B_i$, 
$$
\omega_{\Gamma_1, \ldots, \Gamma_s,\mathscr R}(B_i):=
\begin{cases}
1,~~~~~~~~~~~~~~~~~~~~~~~\text{ if }\col(B_i)=1,\\
\prod_{j=1}^s x_j^{m_{\Gamma_j}(\red(B_i))}~~\text{ if }\col(B_i)\geq 2,
\end{cases}
$$
where $\col(B_i)$ is the number of columns in $B_i$. We 
define the weight of $Q$, $\omega_{\Gamma_1, \ldots, \Gamma_s,\mathscr R}(Q)$, to be $(-1)^{h-1}\prod_{i=1}^h \omega_{\Gamma_1, \ldots, \Gamma_s,\mathscr R}(B_i)$. We then define the generalized 
joint $\Gamma_1, \ldots, \Gamma_s,\mathscr R$-cluster polynomial to be 
\begin{equation}\label{eqn:gjcluster_poly}
GC^{0,0,k}_{kn,\Gamma_1, \ldots, \Gamma_s,\mathscr R}(x_1,\cdots,x_m) := \sum_{Q \in \mathcal{GC}^{0,0,k}_{kn,\Gamma_1, \ldots, \Gamma_s,\mathscr R}} \omega_{\Gamma_1, \ldots, \Gamma_s,\mathscr R}(Q).
\end{equation}

Then we have the following theorem which is the multivariate version of Theorem \ref{thm:mainGamma}.

\begin{theorem}\label{thm:gjointcluster}
	Let $\mathscr R$ be a binary relation on pairs of 
columns $(C,D)$ of height $k$ 
	which are filled 
	with integers which are increasing from bottom to top. 
Let $\Gamma_1, \ldots, \Gamma_s$ be a  sequence of sets of patterns 
such that  $\Gamma_i \subseteq  \mathcal{P}^{0,0,k}_{r_ik}$
where $r_i \geq $ for $1\leq i\leq s$. Then 
	\begin{multline*}
	1+\sum_{n \geq 1} \frac{t^{kn}}{(kn)!} 
	\sum_{F \in \mathcal{P}^{0,0,k}_{kn,\mathscr R}}\prod_{i=1}^s 
x_i^{\Gamma_i\text{-mch}(F)} =\\  
	\frac{1}{1-\sum_{n \geq 1} \frac{t^{kn}}{(kn)!} 
GC^{0,0,k}_{kn,\Gamma_1, \ldots, \Gamma_s,\mathscr R}(x_1-1,\cdots,x_m-1)}.
	\end{multline*}
\end{theorem}

There are obvious analogues of start, end, and start-end general 
clusters in the joint cluster setting. Throughout the rest of 
this section, we shall assume 
that we are given a binary relation $\mathscr R$ 
on columns of integers and a sequence $\Gamma_1, \ldots, \Gamma_s$ 
of a sets of patterns such that 
$\Gamma_i \subseteq \mathcal{P}^{0,0,k}_{r_ik}$ 
where $r_i \geq 2$ for $i=1, \ldots, s$.  

\begin{definition}
We say that $Q \in \mathcal{MP}^{0,j,k}_{kn+j,\Gamma_1, \ldots, \Gamma_s}$ 
is a 
{\bf generalized joint $\Gamma_1, \ldots, \Gamma_s,\mathscr R$-end-cluster} if 
we can write $Q=B_1B_2\cdots B_m$ where $B_i$ are blocks 
of consecutive columns 
in $Q$ such that  
\begin{enumerate}
\item $B_m$ is a column of height $j$, 	
\item for $i < m$, 
either $B_i$ is a single column or $B_i$ consists of $r$-columns 
where $r \geq 2$, $\red(B_i)$ is a joint $\Gamma_1, \ldots, \Gamma_s$-cluster in $\mathcal{MP}_{kr,\Gamma_1, \ldots, \Gamma_s}$, 
and any pair of consecutive columns in $B_i$ are in $\mathscr R$ and
	\item for $1\leq i\leq m-1$, the pair $(\last(B_i),\first(B_{i+1})$ is 
not in $\mathscr R$.  
\end{enumerate}
\end{definition}
Let $\mathcal{GEC}^{0,j,k}_{kn+j,\Gamma_1, \ldots, \Gamma_s,\mathscr R}$ denote the set of all generalized 
$\Gamma_1, \ldots, \Gamma_s,\mathscr R$-end-clusters which have $n$ columns of height $k$ followed 
by a column of height $j$. 
Given $Q =B_1 B_2 \ldots B_m \in \mathcal{GEC}^{0,j,k}_{kn+j,\Gamma_1, \ldots, \Gamma_s,\mathscr R}$, we 
define the weight of $B_i$, $w_{\Gamma_1, \ldots, \Gamma_s,\mathscr R}(B_i)$, to be 1 if $B_i$ is a single column and 
$\prod_{i=1}^s x_i^{m_{\Gamma_i}(\red(B_i))}$ if $B_i$ is order isomorphic to a $\Gamma_1, \ldots, \Gamma_s$-cluster. We 
define the weight of $Q$, $w_{\Gamma_1, \ldots, \Gamma_s,\mathscr R}(Q)$, 
to be 
$(-1)^{m-1}\prod_{i=1}^s w_{\Gamma_1, \ldots, \Gamma_s,\mathscr R}(B_i)$. We 
then let 
\begin{equation}\label{GEPRC2}
GEC^{0,j,k}_{kn+j,\Gamma_1, \ldots, \Gamma_s,\mathscr R}(x_1,\ldots,x_s) = 
\sum_{Q \in \mathcal{GC}^{0,j,k}_{kn+j,\Gamma_1, \ldots, \Gamma_s,\mathscr R}} w_{\Gamma_1, \ldots, \Gamma_s,\mathscr R}(Q).
\end{equation}

Then we have the following multivariate analogue of Theorem 
\ref{thm:jmainGamma}. 

\begin{theorem}\label{thm:jmainGammajoint}
Let $\mathscr R$ be a binary relation on pairs of columns $(C,D)$  
which are filled 
with integers which are increasing from bottom to top. 
 Let $\Gamma_1, \ldots, \Gamma_s$ be a sequence of a sets of patterns 
such that $\Gamma_i \subseteq \mathcal{P}^{0,0,k}_{r_ik}$ 
where $r_i \geq 2$ for $i=1, \ldots, s$. Then 
\begin{multline}\label{eq:endprclusterjoint}
\sum_{n \geq 0} \frac{t^{kn+j}}{(kn+j)!} 
\sum_{F \in \mathcal{P}^{0,j,k}_{kn+j,\mathscr R}} 
\prod_{i=1}^s x_i^{\Gamma_i\text{-}\mathrm{mch}(F)} =  \\
\frac{\sum_{n \geq 0} 
\frac{t^{kn+j}}{(kn+j)!} GEC^{0,j,k}_{kn+j,\Gamma_1, \ldots, \Gamma_s,\mathscr R}(x_1-1, \ldots,x_s-1)}{1-\sum_{n \geq 1} 
\frac{t^{kn}}{(kn)!} GC^{0,0,k}_{kn,\Gamma_1, \ldots, \Gamma_s,\mathscr R}
(x_1-1, \ldots,x_s-1)}.
\end{multline}
\end{theorem}

\begin{definition}
We say that $Q \in \mathcal{MP}^{i,0,k}_{i+kn,\Gamma}$ is a 
{\bf generalized joint $\Gamma_1, \ldots, \Gamma_s,\mathscr R$-start-cluster} 
if 
we can write $Q=B_1B_2\cdots B_m$ where $B_i$ are blocks 
of consecutive columns 
in $Q$ such that  
\begin{enumerate}
\item $B_1$ is a single column of height $i$, 	
\item for $2< a \leq m$, 
either $B_a$ is a single column or $B_a$ consists of $r$-columns 
where $r \geq 2$, $\red(B_a)$ is a joint 
$\Gamma_1, \ldots, \Gamma_s$-cluster in $\mathcal{MP}_{kr,\Gamma_1, \ldots, \Gamma_s}$, 
and any pair of consecutive columns in $B_i$ are in $\mathscr R$ and
	\item for $1\leq i\leq m-1$, the pair $(\last(B_i),\first(B_{i+1})$ is 
not in $\mathscr R$.  
\end{enumerate}
\end{definition}
Let $\mathcal{GSC}^{i,0,k}_{i+kn,\Gamma,\mathscr R}$ denote the set of all generalized 
$\Gamma,\mathscr R$-start-clusters which start with a column of height $i$ and 
which is followed by $n$ columns of height $k$. 
Given $Q =B_1 B_2 \ldots B_m \in \mathcal{GEC}^{i,0,k}_{i+kn,\Gamma,\mathscr R}$, we 
define the weight of $B_i$, $w_{\Gamma_1, \ldots, \Gamma_s,\mathscr R}(B_i)$, to be 1 if $B_i$ is a single column and 
$\prod_{j=1}^s x_i^{m_{\Gamma_j}(\red(B_i))}$ if $B_i$ is order isomorphic to a joint $\Gamma_1, \ldots, \Gamma_s$-cluster. We 
define the weight of $Q$, $w_{\Gamma_1, \ldots, \Gamma_s,\mathscr R}(Q)$, 
to be 
$(-1)^{m-1}\prod_{i=1}^m w_{\Gamma_1, \ldots, \Gamma_s,\mathscr R}(B_i)$. We 
then let 
\begin{equation}\label{GSPRC2}
GSC^{i,0,k}_{i+kn,\Gamma_1, \ldots, \Gamma_s,\mathscr R}(x) = 
\sum_{Q \in \mathcal{GC}^{i,0,,k}_{i+kn,\Gamma_1, \ldots, \Gamma_s,\mathscr R}} w_{\Gamma_1, \ldots, \Gamma_s,\mathscr R}(Q). 
\end{equation}

Let $\mathcal{P}^{i,0,k}_{i+kn,\mathscr R}$ denote the set of all elements 
$F \in \mathcal{P}^{i,0,k}_{i+kn}$ such that the relation $\mathscr R$ holds 
for any pair of consecutive columns in $F$. 
Then we have the following theorem. 

\begin{theorem}\label{thm:imainGammajoint}
Let $\mathscr R$ be a binary relation on pairs of columns $(C,D)$  
which are filled 
with integers which are increasing from bottom to top. 
Let $\Gamma_1, \ldots, \Gamma_s$ be a sequence of a sets of patterns 
such that $\Gamma_i \subseteq \mathcal{P}^{0,0,k}_{r_ik}$ 
where $r_i \geq 2$ for $i=1, \ldots, s$. Then 
\begin{multline}\label{eq:stprclusterjoint}
\sum_{n \geq 0} \frac{t^{i+kn}}{(i+kn)!} 
\sum_{F \in \mathcal{P}^{i,0,k}_{i+kn,\mathscr R}} \prod_{i=1}^s 
x_i^{\Gamma_i\text{-}\mathrm{mch}(F)} =  \\
\frac{\sum_{n \geq 0} 
\frac{t^{i+kn}}{(i+kn)!} GSC^{i,0,k}_{i+kn,\Gamma_1, \ldots, \Gamma_s,\mathscr R}(x_1-1,\ldots,x_s-1)}{1-\sum_{n \geq 1} 
\frac{t^{kn}}{(kn)!} GC^{0,0,k}_{kn,\Gamma_1, \ldots, \Gamma_s,\mathscr R}
(x_1-1,\ldots,x_s-1)}.
\end{multline}
\end{theorem}

\begin{definition}
	We say that $Q \in \mathcal{MP}^{i,j,k}_{i+kn+j,\Gamma}$ is a 
	{\bf generalized joint $\Gamma_1, \ldots, \Gamma_s,\mathscr R$-start-end-cluster} if 
	we can write $Q=B_1B_2\cdots B_m$ where $B_i$ are blocks 
	of consecutive columns 
	in $Q$ such that  
	\begin{enumerate}
		\item $m\geq 2$,
		\item $B_1$ is a column of height $i$, 
		\item $B_m$ is a column of height $j$, 	
		\item for $1<i < m$, 
		either $B_i$ is a single column or $B_i$ consists of $r$-columns 
		where $r \geq 2$, $\red(B_i)$ is a joint 
$\Gamma_1, \ldots, \Gamma_s$-cluster in $\mathcal{MP}_{kr,\Gamma_1, 
\ldots, \Gamma_s}$, 
		and any pair of consecutive columns in $B_i$ are in $\mathscr R$ and
		\item for $1\leq i\leq m-1$, the pair $(\last(B_i),\first(B_{i+1}))$ is 
		not in $\mathscr R$.  
	\end{enumerate}
\end{definition}

Let $\mathcal{GSEC}^{i,j,k}_{i+kn+j,\Gamma_1, \ldots, \Gamma_s,\mathscr R}$ denote the set of all generalized $\Gamma_1, \ldots, \Gamma_s,\mathscr R$-start-end-clusters which have $n$ columns of height $k$ between a column of height $i$ and a column of height $j$. 
Given $Q =B_1 B_2 \ldots B_m \in \mathcal{GSEC}^{i,j,k}_{i+kn+j,\Gamma_1, \ldots, \Gamma_s,\mathscr R}$, we 
define the weight of $B_i$, $\omega_{\Gamma_1, \ldots, \Gamma_s,\mathscr R}(B_i)$, to be 1 if $B_i$ is a single column and 
$\prod_{j=1}^s x^{m_{\Gamma_j}(\red(B_i))}$ if $B_i$ is order isomorphic to a 
joint $\Gamma_1, \ldots, Gamma_s$-cluster. Then we 
define the weight of $Q$, $\omega_{\Gamma_1, \ldots, \Gamma_s,\mathscr R}(Q)$, to be 
$(-1)^{m-1}\prod_{i=1}^m \omega_{\Gamma_1, \ldots, \Gamma_s,\mathscr R}(B_i)$. We 
let 
\begin{equation}
GSEC^{i,j,k}_{kn+j,\Gamma_1, \ldots, \Gamma_s,\mathscr R}(x) = 
\sum_{Q \in \mathcal{GSEC}^{i,j,k}_{i+kn+j,\Gamma_1, \ldots, \Gamma_s,\mathscr R}} \omega_{\Gamma_1, \ldots, \Gamma_s,\mathscr R}(Q).
\end{equation}

\begin{theorem}\label{thm:ijmainGammajoint}
Let $\mathscr R$ be a binary relation on pairs of columns $(C,D)$  which are filled with integers which are increasing from bottom to top. Let $\Gamma_1, \ldots, \Gamma_s$ be a sequence of a sets of patterns such that $\Gamma_i \subseteq \mathcal{P}^{0,0,k}_{r_ik}$ where $r_i \geq 2$ for $i=1, \ldots, s$. Then 
\begin{multline}\label{eqn:seprclusterjoint}
	\sum_{n \geq 0} \frac{t^{i+kn+j}}{(i+kn+j)!} 
	\sum_{F \in \mathcal{P}^{i,j,k}_{i+kn+j,\mathscr R}} 
\prod_{i=1}^s x_i^{\Gamma_i\text{-}\mathrm{mch}(F)} =  \\
	\frac{A(x_1-1, \ldots, x_s-1)B(x_1-1, \ldots, x_s-1)}{1-\sum_{n \geq 1} 
		\frac{t^{kn}}{(kn)!} GC^{0,0,k}_{kn,\Gamma_1, \ldots, \Gamma_s,\mathscr R}(x_1-1,\ldots,x_s-1)}\\
	+\sum_{n \geq 0} 
	\frac{t^{i+kn+j}}{(i+kn+j)!} GSEC^{i,j,k}_{i+kn+j,\Gamma_1, \ldots, \Gamma_s,\mathscr R}(x_1-1,\ldots,x_s -1),
\end{multline}
where 
$$A(x_1-1, \ldots, x_s-1) = \sum_{n \geq 0} 
		\frac{t^{i+kn}}{(i+kn)!} GSC^{i,0,k}_{i+kn,\Gamma_1, \ldots, \Gamma_s,\mathscr R}(x_1-1,\ldots,x_s-1)$$
and 
$$B(x_1-1, \ldots, x_s-1) = \sum_{n \geq 0} 
		\frac{t^{kn+j}}{(kn+j)!} GEC^{0,j,k}_{kn+j,\Gamma_1, \ldots, \Gamma_s,\mathscr R}(x_1-1, \ldots, x_s-1). $$
\end{theorem}

\subsubsection{Example}

In this section, we shall consider joint 
distribution of $\tau=1423$ and $\pi=162534$ in 
up-down permutations. That is, we shall 
 compute  the generating function
\begin{equation}\label{eqn:ex_joint_ud}
F_{\tau,\pi}(x,y,t):=
1+ \sum_{j=0}^1 
\sum_{n\geq 1}\frac{t^{2n+j}}{(2n+j)!}
\sum_{\sg\in \mathcal{E}^{0,j,2}_{2n+j}}x^{\tau\text{-mch}(\sg)}y^{\pi\text{-mch}(\sg)}t^n.
\end{equation}

As in the previous sections, we can think of up-down permutations as arrays 
as in Figure \ref{fig:ud9_wF}. 

\fig{1.2}{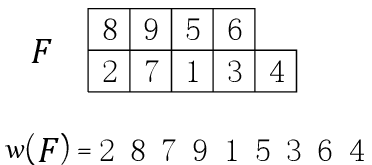}{$F\in\mathcal{P}_{9,\mathscr R}^{0,1,2}$ and $w(F)\in\mathcal{UD}_{9}$.}

We let $\mathscr R$ be the binary relation such that $\mathscr R$ 
holds for a pair of columns $(C,D)$ if and only if the top element of column $C$ is greater than the bottom element of column $D$. 
Then 
\begin{eqnarray*}
\mathcal{E}^{0,0,2}_{2n} &=& \sum_{F \in \mathcal{P}_{2n,\mathscr R}^{0,0,2}}
w(F) \ \mbox{and} \\
 \mathcal{E}^{0,2,2}_{2n+1} &=& \sum_{F \in \mathcal{P}_{2n+1,\mathscr R}^{0,1,2}} w(F).
\end{eqnarray*} 
In particular, $\tau=1423 = w(P)$ and $\pi=162534 =w(Q)$ where 
$P=\begin{array}{|c|c|}
\hline 4 & 3 \\
\hline 1 & 2 \\
\hline
\end{array}$  and $Q=\begin{array}{|c|c|c|}
\hline 6&5&4 \\
\hline 1 & 2 &3\\
\hline
\end{array}$.

Thus the generating function $F_{\tau,\pi}(x,y,t)$ in (\ref{eqn:ex_joint_ud}) is equivalent to 
\begin{multline}\label{eqn:ex_joint_ud_array}
F_{P,Q}(x,y,t):=\\
\sum_{n\geq 0}\left(\frac{t^{2n}}{(2n)!}\sum_{F\in \mathcal{P}_{2n,\mathscr R}^{0,0,2}}x^{P\text{-mch}(F)}y^{Q\text{-mch}(F)}+\frac{t^{2n+1}}{(2n+1)!}\sum_{F\in \mathcal{P}_{2n+1,\mathscr R}^{0,1,2}}x^{P\text{-mch}(F)}y^{Q\text{-mch}(F)}\right).
\end{multline}

Thus we must consider the following two generating functions
\begin{equation}
	A_{P,Q}(x,y,t):=\sum_{n\geq 0}\frac{t^{2n}}{(2n)!}\sum_{F\in\mathcal{P}_{2n,\mathscr R}^{0,0,2}}x^{P\text{-mch}(F)}y^{Q\text{-mch}(F)},
\end{equation}
and
\begin{equation}
	B_{P,Q}(x,y,t):=\sum_{n\geq 0}\frac{t^{2n+1}}{(2n+1)!}\sum_{F\in\mathcal{P}_{2n+1,\mathscr R}^{0,1,2}}x^{P\text{-mch}(F)}y^{Q\text{-mch}(F)}
\end{equation}
because $F_{P,Q}(x,y,t)=A_{P,Q}(x,y,t)+B_{P,Q}(x,y,t)$.

We shall compute $A_{P,Q}(x,y,t)$ first. By Theorem \ref{thm:gjointcluster}, we have
\begin{equation}\label{eqn:joint_ud_even}
A_{P,Q}(x,y,t)=\frac{1}{1-\sum_{n\geq 1}\frac{t^{2n}}{(2n)!}GC_{2n,P,Q,\mathscr R}^{0,0,2}(x-1,y-1)}.
\end{equation}

We shall start by discussing structures of joint $(P,Q)$-clusters. 
It is easy to see that  
 the Hasse diagram of a joint $(P,Q)$-clusters is of the form 
pictured in Figure \ref{fig:pcluster_cp}. 

\fig{1.40}{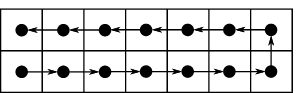}{The Hasse diagram of a joint $(P,Q)$-cluster.} 

Let $\mathcal{CM}_{2n,P,Q}$ be the set of joint $(P,Q)$-clusters consisting of two rows and $n$ columns. Let $C_{2n,P,Q}(x,y)$ denote the joint cluster polynomial. We can 
use the exact same reasoning that we used to compute 
$C^{0,0,2}_{2n,P_{2,2},P_{2,3}}(x,y)$ to prove that 
\begin{eqnarray*}
C_{2,P,Q}(x,y)&=& 1, \\	
C_{4,P,Q}(x,y)&=& x, \mbox{and} \\
C_{2n,P,Q}(x,y)&=&(x+1)yC_{2n-4,P,Q}(x,y)+(xy+x+y)C_{2n-2,P,Q}(x,y),~~\text{for }n\geq 3.
\end{eqnarray*}

Next we shall compute $GC^{0,0,2}_{2n,P,Q,\mathscr R}(x,y)$. In fact our computation is very similar to the computation of $GC^{0,0,2}_{2n,P}(x)$ in Section \ref{sec:ud_in_du}.

Suppose $J\in\mathcal{GC}_{2n,P,Q,\mathscr R}^{0,0,2}$ has $m$ blocks, i.e., $J=B_1B_2\ldots B_m$. For an array $F$, we let $\col(F)$ to denote the number of columns in $\col(F)$. Each $B_i$ is either order isomorphic to a joint $(P,Q)$-cluster or a single column.  We let $\mathcal{GC}_{\col = (b_1,b_2,\ldots,b_m)}$ to denote the set of generalized clusters such that $\col(B_i)=b_i$. Given $\col(B_i)$ for each $1\leq i\leq m$, we can represent the fillings for the set of such generalized clusters as the set of linear extensions 
of a poset whose Hasse diagram corresponds $(b_1, \ldots, b_m)$. That is, 
for any block $(\last(B_i),\first(B_{i+1})$ is not in $\mathscr R$ so 
that there must be an arrow directed from the top of the last column of $B_i$ to the 
bottom of the first column of $B_{i+1}$. We let $\Gamma(b_1,b_2,\ldots,b_m)$ denote the Hasse diagram corresponding to $\mathcal{GC}_{\col = (b_1,b_2,\ldots,b_m)}$. For example, $\Gamma (3,1,1,5,1)$ is pictured in Figure \ref{fig:updown_gc_ex}.

\fig{1.40}{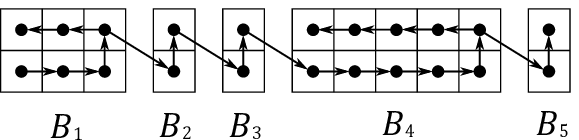}{$\Gamma (3,1,1,5,1)$.}

These are 
the same type of Hasse diagrams that we saw in Section \ref{sec:ud_in_du}.
Then using the same reasoning that we used in Section \ref{sec:ud_in_du}, 
it follows that 
\begin{equation}\label{eqn:joint_up_gc_compute}
\sum_{J\in\mathcal{GC}_{\col = (b_1,b_2,\ldots,b_m)}}\omega_{P,Q,\mathscr R}(J)=(-1)^{m-1}\text{LE}(\Gamma(b_1,b_2,\ldots,b_m))\prod_{i=1}^m C_{2b_i,P,Q}(x,y),
\end{equation}
where $\text{LE}(\Gamma(b_1,b_2,\ldots,b_m))$ is the number of linear extensions of $\Gamma(b_1,b_2,\ldots,b_m)$ and by convention, we let $C_{2}(x)=1$. Therefore, to compute (\ref{eqn:du_gc_compute}), we only need to count linear extensions of $\Gamma(b_1,b_2,\ldots,b_m)$. Continue using $\Gamma (3,1,1,5,1)$ as example, the Hasse diagram in  Figure \ref{fig:updown_gc_ex} is actually a tree-like diagram, as drawn in Figure \ref{fig:downup_gc_ex_str}. 
Hence  
$$
\text{LE}(\Gamma (3,1,1,5,1))=\binom{6}{4}\binom{18}{2}.
$$
The only difference between the generalized clusters 
considered in Section \ref{sec:ud_in_du} and our current situation 
is that the blocks with more than one column are weighted as 
joint $P,Q$-clusters.

Using a computer program, we computed that 
\begin{eqnarray}\label{eqn:ud_gc_list}
GC_{2,P,Q,\mathscr R}^{0,0,2}(x,y)&=&1\nonumber\\
GC_{4,P,Q,\mathscr R}^{0,0,2}(x,y)&=&x-1\nonumber\\
GC_{6,P,Q,\mathscr R}^{0,0,2}(x,y)&=&\begin{pmatrix}
1\\
x\\
x^2
\end{pmatrix}^T\begin{pmatrix}
1&1\\
-4&2\\
1&1
\end{pmatrix}\begin{pmatrix}
1\\
y
\end{pmatrix}\nonumber\\
GC_{8,P,Q,\mathscr R}^{0,0,2}(x,y)&=&\begin{pmatrix}
1\\
x\\
x^2\\
x^3
\end{pmatrix}^T\begin{pmatrix}
-1& -7& 1\\
9&-12&3\\
-12&-3&3\\
1&2&1
\end{pmatrix}\begin{pmatrix}
1\\
y\\
y^2
\end{pmatrix}
\nonumber\\
GC_{10,P,Q,\mathscr R}^{0,0,2}(x,y)&=&\begin{pmatrix}
1\\
x\\
x^2\\
x^3\\
x^4
\end{pmatrix}^T\begin{pmatrix}
1& 22& -10& 1\\ 
-16& 0& -27& 4\\
 55 & -63& -21& 6\\
  -33& -38& -1& 4\\ 
  1& 3& 3& 1
\end{pmatrix}\begin{pmatrix}
1\\
y\\
y^2\\
y^3
\end{pmatrix}\nonumber\\
GC_{12,P,Q,\mathscr R}^{0,0,2}(x,y)&=&
\begin{pmatrix}
1\\
x\\
x^2\\
x^3\\
x^4\\
x^5
\end{pmatrix}^T\begin{pmatrix}
-1 & -50 & 2 & -14 & 1\\ 
25 & 189 & -111 & -52 & 5\\
-164 & 336 & -339 & -68 & 10\\
288 & -91 & -331 & -32 & 10\\
-88 & -184 & -99 & 2 & 5\\
1 & 4 & 6 & 4 & 1
\end{pmatrix}\begin{pmatrix}
1\\
y\\
y^2\\
y^3\\
y^4
\end{pmatrix}
\nonumber\\
&\cdots&
\end{eqnarray}

Using these initial terms, we can 
use our formula for $GC_{2n,P,Q,\mathscr R}(x,y,t)$ given in  
(\ref{eqn:joint_ud_even}) to compute that  
\begin{eqnarray*}
&&A_{P,Q}(x,y,t)\\
&=&1+\frac{t^2}{2!}+(4+x)\frac{t^4}{4!}+\left(36+24x+x^2y\right)\frac{t^6}{6!}+\left(593+680 x+64 x^2+47 x^2 y+x^3 y^2\right)\frac{t^8}{8!}\\
&&+\left(15676 + 25691 x + 6481 x^2 + 2199 x^2 y + 396 x^3 y + 77 x^3 y^2 + 
x^4 y^3\right)\frac{t^{10}}{10!}+\cdots.
\end{eqnarray*}

Next we shall compute $B_{P,Q}(x,y,t)$. By Theorem \ref{thm:jmainGammajoint},
\begin{equation}\label{eqn:joint_up_odd}
	B_{P,Q}(x,y,t)=
	\frac{\sum_{n \geq 0} 
		\frac{t^{2n+1}}{(2n+1)!} GEC^{0,1,2}_{2n+1,P,Q,\mathscr R}(x-1, y-1)}{1-\sum_{n \geq 1} 
		\frac{t^{2n}}{(2n)!} GC^{0,0,2}_{2n,P,Q,\mathscr R}
		(x-1,y-1)}.
\end{equation}

Again, the computation of joint generalized end-cluster polynomials $GEC_{2n+1,P,Q,\mathscr R}^{0,1,2}(x,y,t)$ is the same as the computation of generalized end-cluster polynomials in Section \ref{sec:ud_in_du}. That is,  
the only difference between the generalized clusters 
considered in Section \ref{sec:ud_in_du} and our current situation 
is that the blocks with more than one column are weighted as 
joint $P,Q$-clusters. 

Using a computer program, we computed that 
\begin{eqnarray}\label{eqn:ud_gec_list}
GEC_{1,P,Q,\mathscr R}^{0,1,2}(x,y)&=&1\nonumber\\
GEC_{3,P,Q,\mathscr R}^{0,1,2}(x,y)&=&-1\nonumber\\
GEC_{5,P,Q,\mathscr R}^{0,1,2}(x,y)&=&-2x+1\nonumber\\
GEC_{7,P,Q,\mathscr R}^{0,1,2}(x,y)&=&\begin{pmatrix}
1\\
x\\
x^2
\end{pmatrix}^T\begin{pmatrix}
-1&-3\\
6&-6\\
-3&-3
\end{pmatrix}\begin{pmatrix}
1\\
y
\end{pmatrix}\nonumber\\
GEC_{9,P,Q,\mathscr R}^{0,1,2}(x,y)&=&\begin{pmatrix}
1\\
x\\
x^2\\
x^3
\end{pmatrix}^T\begin{pmatrix}
1& 13& -4\\
-12&18&-12\\
25&-3&-12\\
-4&-8&-4
\end{pmatrix}\begin{pmatrix}
1\\
y\\
y^2
\end{pmatrix}
\nonumber\\
GEC_{11,P,Q,\mathscr R}^{0,1,2}(x,y)&=&\begin{pmatrix}
1\\
x\\
x^2\\
x^3\\
x^4
\end{pmatrix}^T\begin{pmatrix}
-1 & -34 & 19 & -5\\ 
20 & 46 & 42 & -20\\
-94 & 179 & 12 & -30\\
90 & 84 & -26 & -20\\ 
-5 & -15 & -15 & -5
\end{pmatrix}\begin{pmatrix}
1\\
y\\
y^2\\
y^3
\end{pmatrix}\nonumber\\
GEC_{13,P,Q,\mathscr R}^{0,1,2}(x,y)&=&
\begin{pmatrix}
1\\
x\\
x^2\\
x^3\\
x^4\\
x^5
\end{pmatrix}^T\begin{pmatrix}
1 & 70 & 68 & 28 & -6\\ 
-30 & -432 & 566 & 88 & -30\\
250 & -434 & 1254 & 72 & -60\\
-612 & 684 & 1046 & -32 & -60\\
300 & 592 & 254 & -68 & -30\\
-6 & -24 & -36 & -24 & -6
\end{pmatrix}\begin{pmatrix}
1\\
y\\
y^2\\
y^3\\
y^4
\end{pmatrix}
\nonumber\\
&\cdots&
\end{eqnarray}

Using these initial terms, we can use our formula 
for $GEC^{0,1,2}_{2n+1,P,Q,\mathscr R}(x,y,t)$ given in  (\ref{eqn:joint_up_odd}) to compute that 
\begin{eqnarray*}
	&&B_{P,Q}(x,y,t)=\\
	&&t+2\frac{t^3}{3!}+(13+3x)\frac{t^5}{5!}+\left(165 + 103 x + 4 x^2 y\right)\frac{t^7}{7!}+\left(3507 + 3832 x + 340 x^2 + 252 x^2 y + 5 x^3 y^2\right)\frac{t^9}{9!}\\
	&&+\left(113597 + 178871 x + 43395 x^2 + 15015 x^2 y + 2442 x^3 y + 
	462 x^3 y^2 + 11 x^4 y^3\right)\frac{t^{11}}{11!}+\cdots.
\end{eqnarray*}

Finally, taking sum of $A_{P,Q}(x,y,t)$ and $B_{P,Q}(x,y,t)$, we get the generating function for joint distribution of patterns $1423$ and $162534$ in up-down permutations,
\begin{eqnarray*}
	&&1+t+\frac{t^2}{2!}+2\frac{t^3}{3!}+(4+x)\frac{t^4}{4!}+(13+3x)\frac{t^5}{5!}+\left(36+24x+x^2y\right)\frac{t^6}{6!}+\left(165 + 103 x + 4 x^2 y\right)\frac{t^7}{7!}\\
	&+&\left(593+680 x+64 x^2+47 x^2 y+x^3 y^2\right)\frac{t^8}{8!}+\left(3507 + 3832 x + 340 x^2 + 252 x^2 y + 5 x^3 y^2\right)\frac{t^9}{9!}\\
	&+&\left(15676 + 25691 x + 6481 x^2 + 2199 x^2 y + 396 x^3 y + 77 x^3 y^2 + 
	x^4 y^3\right)\frac{t^{10}}{10!}+\cdots
\end{eqnarray*}

\end{document}